\documentclass[11pt]{amsart}

\usepackage{color}
\usepackage{bbm}
\usepackage{graphicx}
\usepackage{amsfonts,amssymb,amsmath,amsthm,mathrsfs,enumerate,multicol,esint}

\usepackage[pagebackref,hypertexnames=false, colorlinks, citecolor=red, linkcolor=red]{hyperref} 
\usepackage[backrefs]{amsrefs}
\usepackage{bbm}   
\usepackage{subcaption}

\begin{document}
\baselineskip 17.5pt
\hfuzz=6pt

\newtheorem{theorem}{Theorem}[section]
\newtheorem{prop}[theorem]{Proposition}
\newtheorem{lemma}[theorem]{Lemma}
\newtheorem{definition}[theorem]{Definition}
\newtheorem{cor}[theorem]{Corollary}
\newtheorem{example}[theorem]{Example}
\newtheorem{remark}[theorem]{Remark}
\newcommand{\ra}{\rightarrow}
\renewcommand{\theequation}
{\thesection.\arabic{equation}}
\newcommand{\ccc}{{\mathcal C}}
\newcommand{\one}{1\hspace{-4.5pt}1}

\def\RR{\mathbb R}

\newcommand{\unit}{\mathbbm 1} 

\def\gfz{\genfrac{}{}{0pt}{}}

\newcommand{\al}{\alpha}
\newcommand{\be}{\beta}
\newcommand{\ga}{\gamma}
\newcommand{\Ga}{\Gamma}
\newcommand{\de}{\delta}
\newcommand{\ben}{\beta_n}
\newcommand{\De}{\Delta}
\newcommand{\ve}{\varepsilon}
\newcommand{\ze}{\zeta}
\newcommand{\Th}{\chi}
\newcommand{\ka}{\kappa}
\newcommand{\la}{\lambda}
\newcommand{\laj}{\lambda_j}
\newcommand{\lak}{\lambda_k}
\newcommand{\La}{\Lambda}
\newcommand{\si}{\sigma}
\newcommand{\Si}{\Sigma}
\newcommand{\vp}{\varphi}
\newcommand{\om}{\omega}
\newcommand{\Om}{\Omega}

\newcommand{\ca}{\mathcal A}
\newcommand{\cs}{\mathcal S}
\newcommand{\csrn}{\mathcal S (\rn)}
\newcommand{\cm}{\mathcal M}
\newcommand{\cb}{\mathcal B}
\newcommand{\ce}{\mathcal E}
\newcommand{\cd}{\mathcal D}
\newcommand{\cdp}{D'}
\newcommand{\csp}{\mathcal S'}

\def\R{{\mathbb R}}
\def\ls{\lesssim}
\def\ep{\epsilon}

\newcommand{\PP}{{\mathop P \limits^{    \circ}}_{-}}
\newcommand{\lab}{\label}
\newcommand{\med}{\textup{med}}
\newcommand{\pv}{\textup{p.v.}\,}
\newcommand{\loc}{\textup{loc}}
\newcommand{\intl}{\int\limits}
\newcommand{\intlrn}{\int\limits_{\rn}}
\newcommand{\intrn}{\int_{\rn}}
\newcommand{\iintl}{\iint\limits}
\newcommand{\dint}{\displaystyle\int}
\newcommand{\diint}{\displaystyle\iint}
\newcommand{\dintl}{\displaystyle\intl}
\newcommand{\diintl}{\displaystyle\iintl}
\newcommand{\liml}{\lim\limits}
\newcommand{\suml}{\sum\limits}
\newcommand{\supl}{\sup\limits}
\newcommand{\f}{\frac}
\newcommand{\df}{\displaystyle\frac}
\newcommand{\p}{\partial}
\newcommand{\Ar}{\textup{Arg}}
\newcommand{\abssig}{\widehat{|\si_0|}}
\newcommand{\abssigk}{\widehat{|\si_k|}}
\newcommand{\tT}{\tilde{T}}
\newcommand{\tV}{\tilde{V}}
\newcommand{\wt}{\widetilde}
\newcommand{\wh}{\widehat}
\newcommand{\lp}{L^{p}}
\newcommand{\nf}{\infty}
\newcommand{\rrr}{\mathbf R}
\newcommand{\li}{L^\infty}
\newcommand{\rn}{\mathbf R^n}
\newcommand{\tf}{\tfrac}
\newcommand{\zzz}{\mathbf Z}

\newcommand{\MM}{\mathcal{M}}

\newcommand{\contain}[1]{\left( #1 \right)}
\newcommand{\ContainA}[1]{( #1 )}
\newcommand{\ContainB}[1]{\bigl( #1 \bigr)}
\newcommand{\ContainC}[1]{\Bigl( #1 \Bigr)}
\newcommand{\ContainD}[1]{\biggl( #1 \biggr)}
\newcommand{\ContainE}[1]{\Biggl( #1 \Biggr)}

\allowdisplaybreaks

\title[Two weight Commutators]{Two weight Commutators in the Dirichlet and Neumann Laplacian settings}
\thanks{{\it {\rm 2010} Mathematics Subject Classification:} Primary: 42B35, 42B25.}

\author{Xuan Thinh Duong}
\address{Xuan Thinh Duong, Department of Mathematics, Macquarie University, Sydney}
\email{xuan.duong@mq.edu.au}

\author{Irina Holmes}
\address{Irina Holmes, Department of Mathematics, Washington University - St. Louis\\
         St. Louis, MO 63130-4899 USA}
\email{irina.holmes@email.wustl.edu}

\author{Ji Li}
\address{Ji Li, Department of Mathematics, Macquarie University, Sydney}
\email{ji.li@mq.edu.au}

\author{Brett D. Wick}
\address{Brett D. Wick, Department of Mathematics\\
         Washington University - St. Louis\\
         St. Louis, MO 63130-4899 USA
         }
\email{wick@math.wustl.edu}

\author{Dongyong Yang}
\address{Dongyong Yang, School of Mathematical Sciences\\
         Xiamen University\\
         Xiamen, China
         }
\email{dyyang@xmu.edu.cn}

\begin{abstract}
In this paper we establish the characterization of the weighted BMO via two weight commutators in the settings of the Neumann Laplacian $\Delta_{N_+}$ on the upper half space $\mathbb{R}^n_+$ and  the reflection Neumann Laplacian $\Delta_N$ on $\mathbb{R}^n$ with respect to the weights associated to $\Delta_{N_+}$ and $\Delta_{N}$ respectively. This in turn yields a weak factorization for the corresponding weighted Hardy spaces, where in particular, the weighted class associated to  $\Delta_{N}$ is strictly larger than the Muckenhoupt weighted class and contains non-doubling weights.    In our study, we also make contributions to  the classical Muckenhoupt--Wheeden weighted Hardy space (BMO space respectively) by showing that it can be characterized via area function (Carleson measure respectively)  involving the semigroup generated by the Laplacian on $\mathbb{R}^n$ and that the duality of these weighted Hardy and BMO spaces holds for Muckenhoupt $A^p$ weights with $p\in (1,2]$ while the previously known related results cover only $p\in (1,{n+1\over n}]$. We also point out that this two weight commutator theorem might not be true in the setting of general operators $L$, and in particular we show that it is not true when $L$ is the Dirichlet Laplacian  $\Delta_{D_+}$ on $\mathbb{R}^n_+$.
\end{abstract}

\maketitle



\section{Introduction and Statement of Main Results}
\setcounter{equation}{0}

The theory of Hardy and BMO spaces has been developed successfully as an important part of modern harmonic analysis in
the last 50 years. Hardy spaces $H^1$ and BMO spaces have played a central role and are known as substitutes of $L^1$ and $L^{\infty}$ spaces, respectively, in
the Calder\'on--Zygmund theory of singular integrals. Practical applications gave rise to the necessity of putting suitable weights on
function spaces and the Muckenhoupt $A^p$ classes have been the standard class of weights for singular integrals in the
Calder\'on-Zygmund classes. Combination of the three concepts, singular integrals, function spaces and weights, forms a
central part of the Calder\'on--Zygmund theory and interest in this topic has been extensive.

\subsection{Background and Main Results}

A recent notable result concerning the two weight problem which also gives a characterisation of weighted BMO spaces was achieved in \cite{HLW}.
More specifically, let $R_j = \frac{\partial}{\partial x_j} \Delta^{-1/2}$ be the $j$-th Riesz transform on the Euclidean space $\mathbb R^n$,
$1 < p < \infty$, the weights $\mu, \lambda$ in the Muckenhoupt class $A^p$ and the weight $\nu = \mu^{1/p} \lambda^{-1/p}$. Let $L^p_{w}(\R^n)$ denote the space of functions that are $p$ integrable relative to the measure $w(x)dx$.  Then (\cite{HLW}*{Theorem 1.2}) there
exist constants $0 < c < C < \infty$, depending only on $n, p, \mu, \lambda$, such that
\begin{equation}\label{HLW}
c\| b \| _{{\rm BMO}_\nu(\R^n)} \le \sum _{i=1}^{n} \| [b, R_i] : L^p_{\mu}(\R^n) \rightarrow L^p_{\lambda}(\R^n) \|  \le C \| b \| _{{\rm BMO}_\nu(\R^n)}
\end{equation}
in which $[b, R_i] (f)(x) = b(x) R_i (f)(x) - R_i (bf)(x)$ denotes the commutator of the Riesz transform $R_i$ and the function $b \in {\rm BMO}_\nu(\R^n)$, i.e., the Muckenhoupt--Wheeden weighted BMO space (introduced in \cite{MW76}, see also the definition in Section 1.2 below).  This result provided a characterization of the boundedness of the commutators $[b, R_i]:L^p_{\mu}(\R^n) \rightarrow L^p_{\lambda}(\R) $ in terms of a triple of information $b,\mu$ and $\lambda$.  This result extended important work of Bloom in \cite{B} to handle Riesz transforms and more general Calder\'on-Zygmund operators.  It was additionally inspired by the foundational work of Coifman, Rochberg and Weiss \cite{CRW}, where they characterized the boundedness of the commutators $[b,R_i]$ acting on unweighted Lebesgue spaces in terms of BMO; extending the work of Nehari \cite{Ne} about Hankel operators to higher dimensions.


The theory of the classical harmonic analysis, including the Riesz transforms, $A^p$ weights, BMO and commutators,  is intimately connected to the Laplacian $\Delta$; changing the Laplacian $\Delta$ to other differential operators $L$ introduces new challenges and directions to explore.  Several natural questions arise from \eqref{HLW}, in which the Laplacian plays an essential role.
\begin{itemize}
\item  Question 1: Can we establish \eqref{HLW} for Riesz transforms $\nabla L^{-\frac{1}{2}}$ associated to operators $L$ other than the Laplacian on $\mathbb R^n$?

\item Question 2: What type of weighted BMO spaces are suitable for the estimate \eqref{HLW} for Riesz transforms $\nabla L^{-\frac{1}{2}}$? 

\item Question 3: Is there a new type of $A^p$ weight associated to the operator $L$ and can we obtain \eqref{HLW} for the weights $\mu, \lambda$ in this new weighted class?
\end{itemize}

The main aim of this paper is to address these questions with the cases of $L$ as the Dirichlet Laplacian and the Neumann Laplacian on half spaces.  The Dirichlet Laplacian $\Delta_{D_+}$ and the Neumann Laplacian $\Delta_{N_+}$ on half spaces $\mathbb R^n_+ = \mathbb R^{n-1} \times (0,\infty)$ serve as prototypes of differential operators with boundary value problems, see for example \cite{S}*{Section 3}.  The operators are amenable to a deeper analysis and will allow us to resolve these questions in a satisfactory, and very interesting and surprising way.

We now state our main results, while precise definitions of differential operators and function spaces 
will be given in the related sections.  

We begin with the study of the commutator in the setting of the Neumann Laplacian $\Delta_{N_+}$ with the $i$-th Riesz transform $R_{N_+,i}:={\partial\over\partial x_i} \Delta_{N_+}^{-{1\over2}}$, which provides a positive answer to \eqref{HLW} with $\nu\in A^2(\R^n_+)$.  We now denote by ${\rm BMO}_{\Delta_{N_+},\nu}(\R^n_+)$ the  weighted BMO space  associated with $\Delta_{N_+}$ on $\R^n_+$.
Then we have
\begin{theorem}\label{t:upp-low Neum Lap ha-sp commu}
Suppose $1<p<\infty$, $\mu,\lambda \in  A^p(\R^n_+)$ and $\nu = \mu^{\frac{1}{p}} \lambda^{-\frac{1}{p}}$. Then there are constants $0<c<C<\infty$, depending only on $n,p,\mu,\lambda$ such that for $i=1,\ldots,n$,
$$c\|b\|_{ {\rm BMO}_{\Delta_{N_+},\nu}({\mathbb R}^n_+)} \leq \big\| [b,R_{N_+,i}] : L^{p}_{\mu}(\mathbb{R}^n_+)\to  L^{p}_{\lambda}(\mathbb{R}^n_+)\big\| \leq C \|b\|_{ {\rm BMO}_{\Delta_{N_+},\nu}({\mathbb R}^n_+)}. $$
\end{theorem}

Concerning the new class of $A^p$ weights as posed in Question 3, we now take a natural even reflection of the
Neumann Laplacian $\Delta_{N_+}$ on $\R^n_+$, denoted it by $\Delta_N$. It is direct that the reflection Neumann Laplacian $\Delta_N$ is also
a non-negative self-adjoint operator on $\R^n$.
Then associated to the Neumann Laplacian $\Delta_N$ on $\mathbb R^n$ we define a new class of weights, denoted by $A^p_{\Delta_N}(\mathbb R^n)$, which
strictly contains the Muckenhoupt class $A^p$ and contains even certain particular element which does not satisfy the doubling condition, while all classical Muckenhoupt weights do.
To see that this new class $A^p_{\Delta_N}(\mathbb R^n)$ is well-defined and connected to $\Delta_N$, we first prove that
$R_{N,i}:= {\partial\over\partial x_i} \Delta_N^{-\frac{1}{2}}$, the $i$-th Riesz transforms associated to $\Delta_N$, is bounded on $L^p_w(\R^n)$ if and only if
$w$ is in $A^p_{\Delta_N}(\mathbb R^n)$.
We also establish the exp-log bridge between this $A^p_{\Delta_N}(\mathbb R^n)$ and the weighted BMO space
${\rm BMO}_{\Delta_N}(\mathbb R^n)$, which extends the classical results in \cite{GR}*{Theorem 2.17 and Corollary 2.19}.





%

Our next main result is to show that the two weight commutator theorem \eqref{HLW} is true for the reflection Neumann Laplacian $\Delta_N$.
\begin{theorem}\label{t:upp-low Neum Lap wh-sp commu}
Suppose $1<p<\infty$ and $\mu,\lambda \in  A^p_{\Delta_N}(\mathbb R^n)$. Set $\nu = \mu^{\frac{1}{p}} \lambda^{-\frac{1}{p}}$. Then there are constants $0<c<C<\infty$, depending only on $n,p,\mu,\lambda$ such that for $i=1,\ldots,n$,
$$c\|b\|_{ {\rm BMO}_{\Delta_N,\nu}({\mathbb R}^n)} \leq \big\| [b,R_{N,i}] : L^{p}_{\mu}(\mathbb{R}^n)\to  L^{p}_{\lambda}(\mathbb{R}^n)\big\| \leq C \|b\|_{ {\rm BMO}_{\Delta_N,\nu}({\mathbb R}^n)}. $$
\end{theorem}

We remark that this theorem extends the result \eqref{HLW} for the Laplacian in \cite{HLW} and the unweighted result in \cite{LW}.
It is interesting to note that our theorem holds true for:

(i) $b\in {\rm BMO}_{\Delta_N,\nu}({\mathbb R}^n)$,  which strictly contains ${\rm BMO}_{\nu}({\mathbb R}^n)$, the classical weighted BMO space introduced by Muckenhoupt--Wheeden \cite{MW76}, and,

(ii) $\mu,\lambda$ belong to $A^p_{\Delta_N}(\mathbb R^n)$ which covers not only all the standard Muckenhoupt weights $A^p$ but also some weights beyond $A^p$ which are non-doubling weights.


%
%
%
%

In our study of the weighted BMO space  ${\rm BMO}_{\Delta_N,\nu}({\mathbb R}^n)$ associated with $\Delta_N$, we also make new contributions to classical weighted Hardy and BMO spaces introduced by Muckenhoupt--Wheeden \cite{MW78} and further studied by Garc\'ia-Cuerva \cite{Ga} and Wu \cite{Wu}. In particular, we obtain:
\begin{itemize}
\item[(1)] a new characterization of the Muckenhoupt--Wheeden weighted Hardy and BMO spaces by using the semigroup generated by the Laplacian on $\R^n$ for $w\in A^p(\mathbb R^n)$ with $1<p\leq2$ (see Theorems \ref{t-coin Hardy space cla and Lapl} and \ref{t: coinc clas BMO and Lapl} below);
\item[(2)] the duality of $H^1_w(\R^n)$ and ${\rm BMO}_w(\R^n)$ for $w\in A^p(\mathbb R^n)$ with $1<p\leq2$ (Theorem \ref{t:dual BMO cla}). 
    This extends the classical result of Muckenhoupt--Wheeden \cite{MW78}, which works only for $w\in A^p(\mathbb R^n)$  with  $1<p\leq {n+1\over n}$, and for $p>{n+1\over n}$, their weighted BMO type spaces were defined through the subtraction of polynomials. 
\end{itemize}

Note that we also introduce the weighted Hardy space $H^1_{\Delta_N,\nu}({\mathbb R}^n)$ associated with $\Delta_N$ and prove that it is the predual of ${\rm BMO}_{\Delta_N,\nu}({\mathbb R}^n)$ (Theorem \ref{t-dual}). Similar to \cite{CRW}, our Theorem \ref{t:upp-low Neum Lap wh-sp commu} yields the weak factorization of $H^1_{\Delta_N,\nu}({\mathbb R}^n)$ as follows. The proof is known and hence omitted.
\begin{cor}
Let all the notation and assumptions be the same as in Theorem \ref{t:upp-low Neum Lap wh-sp commu}. Then for every $i=1,\ldots,n$, every $f\in H^1_{\Delta_N,\nu}({\mathbb R}^n)$ can be written as
$$ f(x)= \sum_{j=1}^\infty g_j^i(x) R_{N,i}(h_j^i)(x) + h_j^i(x) R_{N,i}(g_j^i)(x)$$
with $h_j^i\in L^p_\mu(\R^n)$, $g_j^i \in L^{p'}_{\lambda'}(\R^n)$, $p'={p\over p-1}$, and $\lambda'=\lambda^{-{1\over p-1}}$  satisfying
$$ \sum_{j=1}^\infty\|g_j^i\|_{L^{p'}_{\lambda'}(\R^n)}\|h_j^i\|_{L^p_\mu(\R^n)}
\approx \\ \|f\|_{H^1_{\Delta_N,\nu}({\mathbb R}^n)}. $$
\end{cor}

With these positive answers (Theorems \ref{t:upp-low Neum Lap ha-sp commu} and \ref{t:upp-low Neum Lap wh-sp commu}) to Questions 1, 2 and 3, it is natural to expect to extend these results to more settings associated with a general differential operator $L$. However, we point out that it might not be true to have this two weight commutator type theorem associated with a general operator $L$  even when $L$ possesses ``smooth'' regularity such as Gaussian estimates on heat kernel and its derivatives. More specifically, we show that the BMO space can not be characterised by the boundedness of the commutator when $L$ is the Dirichlet Laplacian $\Delta_{D_+}$ on $\R^n_+$, .

To be more precise, the  Dirichlet Laplacian $\Delta_{D_+}$ with the Riesz transform $\nabla \Delta_{D_+}^{-{1\over2}}$ provides a negative answer to \eqref{HLW}.
Suppose $1<p<\infty$, we denote by $A^p(\R^n_+)$ the Muckenhoupt weights on $\R^n_+$, and for
$\mu,\lambda \in  A^p(\R^n_+)$, we set $\nu = \mu^{\frac{1}{p}} \lambda^{-\frac{1}{p}}$.
Also we denote by ${\rm BMO}_{\Delta_{D_+},\nu}(\R^n_+)$ the  weighted BMO space  associated with $\Delta_{D_+}$ on $\R^n_+$.
We
 note that for $b$ in ${\rm BMO}_{\Delta_{D_+},\nu}(\R^n_+)$, the commutator $[b,\nabla \Delta_{D_+}^{-{1\over2}}]$ possesses the upper bound, i.e., it is bounded from $L^p_\mu(\R^n_+)$ to $L^p_\lambda(\R^n_+)$, since $\nabla \Delta_{D_+}^{-{1\over2}}$ is a standard Calder\'on--Zygmund operator on $\R^n_+$ and hence the upper bound follows from \cite{HLW}. However,  the BMO space can NOT be characterised by the boundedness of the commutator for certain $A^p$ weights.
 To see this, we just use the simple weights $\mu=\lambda=1$ to get the counter example. In this case, we have that $\nu=\mu^{\frac{1}{p}} \lambda^{-\frac{1}{p}}=1$. Then  ${\rm BMO}_{\Delta_{D_+},\nu}(\R^n_+)$ becomes ${\rm BMO}_{\Delta_{D_+}}(\R^n_+)$, i.e., the unweighted BMO associated with $\Delta_{D_+}$ as defined and studied in \cite{DDSY}.
We have the following result.
\begin{theorem}\label{t:Dirich commuta counterex}
There exist $b_0\not \in {\rm BMO}_{\Delta_{D_+}}(\R^n_+)$ and a constant $0<C_{b_0}<\infty$, such that
$$ \| [b_0,\nabla \Delta_{D_+}^{-{1\over2}}] : L^{p}(\mathbb{R}^n_+)\to  L^{p}(\mathbb{R}^n_+)\| \leq C_{b_0}. $$
\end{theorem}

\subsection{Structure and Main Methods}

We first point out that the proof of Theorem \ref{t:upp-low Neum Lap ha-sp commu} follows from that of Theorem \ref{t:upp-low Neum Lap wh-sp commu} since in the second theorem we consider $\Delta_N$ which is an even reflection of the Neumann Laplacian $\Delta_{N_+}$ on $\R^n$ as studied in the first theorem (see Remark \ref{remark Th1.1}). Thus, we just explain the method of the proof of Theorem \ref{t:upp-low Neum Lap wh-sp commu}.

To begin with, we note that the structure of the reflection plays an important role here for the study of the BMO space, $A^p$ weights, Riesz transforms and commutators associated with Neumann Laplacian $\Delta_N$.

We first introduce the class of weights $A^p_{\Delta_N}(\mathbb R^n)$ associated with $\Delta_N$
for $p\in[1, \infty)$  (Definition \ref{d-Ap Delta}) and  point out that as in the assumption in Theorem \ref{t:upp-low Neum Lap wh-sp commu},
for $1<p<\infty$ and $\mu,\lambda \in  A^p_{\Delta_N}(\mathbb R^n)$, the new weight  $\nu := \mu^{\frac{1}{p}} \lambda^{-\frac{1}{p}}$
is in $A^2_{\Delta_N}(\mathbb R^n)$.
Then we introduce the weighted BMO space
${\rm BMO}_{\Delta_N,\nu}(\mathbb R^n)$ associated with $\Delta_N$ for $\nu\in A^2_{\Delta_N}(\mathbb R^n)$.
To study Theorem \ref{t:upp-low Neum Lap wh-sp commu}, we first establish the following important property of ${\rm BMO}_{\Delta_N,\nu}(\mathbb R^n)$ as follows (Theorem \ref{p-coinc Dz BMO}): for $\nu\in A^2_{\Delta_N}(\mathbb R^n)$, $f\in {\rm BMO}_{\Delta_N,\nu}(\mathbb R^n)$ if and only if $f_{+,e} \in {\rm BMO}_{\Delta,\nu_{+,e}}(\mathbb R^n)$ and $f_{-,e} \in {\rm BMO}_{\Delta,\nu_{-,e}}(\mathbb R^n)$.  Here ${\rm BMO}_{\Delta,\nu_{+,e}}(\mathbb R^n)$ and ${\rm BMO}_{\Delta,\nu_{-,e}}(\mathbb R^n)$ are the BMO spaces associated with the standard Laplacian on $\R^n$, introduced in Definition \ref{def: BMOwLapalce}, and $f_{\pm,e,o}$ is the even/odd extension to $\R^n$ of the function  $f_{\pm}$, the restriction of $f$ on $\R^n_{\pm}$. And since $\nu\in A^2_{\Delta_N}(\mathbb R^n)$, we get that both $\nu_{+,e}$ and $\nu_{-,e}$ are in the classical Muckenhoupt class $A^2(\R^n)$.

Thus, to study Theorem \ref{t:upp-low Neum Lap wh-sp commu}, we need to further understand the property and structure
of the space ${\rm BMO}_{\Delta,\nu}(\mathbb R^n)$, especially for $\nu\in A^2(\R^n)$. Our key result here (Theorem \ref{t: coinc clas BMO and Lapl}) is to show that for $\nu\in A^2(\R^n)$, ${\rm BMO}_{\Delta,\nu}(\mathbb R^n)$ coincides with the Muckenhoupt--Wheeden weighted BMO space ${\rm BMO}_{\nu}(\mathbb R^n)$ and they have equivalent norms, where ${\rm BMO}_{\nu}(\mathbb R^n)$ is defined (see \cite{MW76}) as the set of all $f\in L^1_{loc}(\R^n)$, such that $$ \|f\|_{{\rm BMO}_{\nu}(\mathbb R^n)}:=\sup_Q {1\over \nu(Q)}\int_Q\big| f-\langle f\rangle_Q  \big|dx <\infty. $$
Here and throughout the whole paper we use $\langle f\rangle_Q:={1\over |Q|}\int_Q f(y)dy$ to denote the average of $f$ over the cube $Q$.
To prove this result, we first prove directly that ${\rm BMO}_{\nu}(\mathbb R^n)\subset {\rm BMO}_{\Delta,\nu}(\mathbb R^n)$ with
$\|f\|_{{\rm BMO}_{\Delta,\nu}(\mathbb R^n)} \leq C\|f\|_{{\rm BMO}_{\nu}(\mathbb R^n)}$.

To show the reverse inclusion, we aim to prove in the following result ${\rm BMO}_{\Delta,\nu}(\mathbb R^n) \subset \big(H^1_{\Delta,\nu}(\mathbb R^n)\big)^*=\big(H^{1,p,\beta}_{\nu}(\mathbb R^n)\big)^* = {\rm BMO}_{\nu}(\mathbb R^n)$
with
$\|f\|_{{\rm BMO}_{\nu}(\mathbb R^n)} \leq C\|f\|_{{\rm BMO}_{\Delta,\nu}(\mathbb R^n)}$,
where  $\beta$ is a non-negative integer, $H^{1,p,\beta}_{\nu}(\mathbb R^n) $ is the atomic Hardy space
whose atom has cancellation up to order $\beta$,
 and $H^1_{\Delta,\nu}(\mathbb R^n)$
is the Hardy space defined via the Littlewood area function associated with $\Delta$, respectively.

We now point out that, for the duality of $H^{1,p,0}_{\nu}(\mathbb R^n)$ and  ${\rm BMO}_{\nu}(\mathbb R^n)$,
the classical results \cite{MW78} and \cite{Ga} hold only for $\nu\in A^p(\R^n)$ with $1<p\leq {n+1\over n}$.
Our contribution to this weighted function space is that we obtain $\big(H^{1,p,0}_{\nu}(\mathbb R^n)\big)^* = {\rm BMO}_{\nu}(\mathbb R^n)$
for $\nu\in A^p(\R^n)$ with $1<p\leq 2$. To prove this result, we need the following elements:

(1) A John--Nirenberg inequality for ${\rm BMO}_{\nu}(\mathbb{R}^n)$, which shows that ${\rm BMO}_{\nu}(\mathbb{R}^n)$ is equivalent to ${\rm BMO}_{\nu,r}(\mathbb{R}^n)$ for all $p\in(1, \infty)$, $r\in[1, p']$ and $\nu\in A^p(\mathbb R^n)$, where
$\frac1p+\frac1{p'}=1$;

(2) The equivalence between ${\rm BMO}_{\nu,2}(\mathbb{R}^n)$ 
for all $r\in[1, 2]$ and the weighted Carleson measure space $CM_\nu$ for all $\nu\in A^2(\mathbb R^n)$;

(3) The duality between the weighted Carleson measure space $CM_\nu$ and the weighted Hardy space $H^1_{\nu,wavelet}(\R^n)$ defined via wavelets basis satisfying 0 order cancellation only,  where $p\in (1,\infty)$
and $\nu \in A^p(\R^n)$;

(4) The coincidence between $H^1_{\nu,wavelet}(\R^n)$ and $H^{1,p,\beta}_{\nu}(\R^n)$ with  a non-negative integer $\beta$.

%

%
%
%

We describe the implications in the following diagram:
\begin{equation}
\begin{matrix}
\hskip-0.8cm{\rm BMO}_{\nu}(\mathbb R^n) &     &     &\\
\hskip-1.2cm\bigg\Updownarrow{(1)} &    &  &\\
\hskip-0.6cm{\rm BMO}_{\nu,2}(\mathbb R^n) &    & H^{1,p,\beta}_{\nu}(\mathbb R^n) &\\ 
\hskip-1.2cm\bigg\Updownarrow{(2)} &    & \bigg\Updownarrow{(4)}&\\
\hskip-1.5cm{CM}_\nu  & \underleftrightarrow{\rm \ \ \ \ (3) \ duality\ \ \ \ \ \ \ }& \ \ \ \  H^1_{\nu, wavelet}(\mathbb R^n)&
\end{matrix}
\end{equation}

We point out (1)  first appeared in \cite{MW76}, where the Muckenhoupt $A^p$ characteristic was not tracked.
We present a modern proof (Theorem \ref{T:MW}) using the techniques of sparse operators and
show that
$$\|b\|_{{\rm BMO}_{\nu}(\mathbb{R}^n)} \leq \|b\|_{{\rm BMO}_{\nu,r}(\mathbb{R}^n)} \leq C_{n,p,r} [\nu]_{A^p}^{\max\{1,{1\over p-1}\}}\|b\|_{{\rm BMO}_{\nu}(\mathbb{R}^n)}.$$

For (2), we point out that we can only prove this equivalence ${\rm BMO}_{\nu,2}(\mathbb{R}^n)$ when we have
the $L^2_{\nu^{-1}}$ norm in the definition of the weighted BMO space.
Hence, it is not clear whether it is true for $\nu\in A^p(\R^n)$ with $p>2$, since in this case due to the John-Nirenberg inequality in step (1), the BMO space is equivalent to  ${\rm BMO}_{\nu,r}(\mathbb{R}^n)$ with $r\leq p'<2$. Moreover, the John-Nirenberg inequality fails in general when $r>p'$.
As a consequence, throughout this diagram we can only make it work for $\nu\in A^p(\R^n)$ with $1<p\leq2$.

For (3), we note that this duality result first appeared in \cite{Wu}, where the order of cancellation of the wavelet basis was not tracked. However, this order plays a key role to us. Thus, we prove that this duality holds for wavelet basis satisfying zero order cancellation only (Theorem \ref{t: Wu}). Combining this duality argument and the results in (1) and (2),
we obtain that for $\nu\in A^p(\R^n)$ with $1<p\leq2$, the definition of $H^1_{\nu, wavelet}(\mathbb R^n)$
is independent of the order of cancellation of the wavelet basis (Theorem \ref{t: H1waveletcharacterization}). To the best
of our knowledge this is not explicitly known before.


For (4), we note that for each non-negative integer $\beta$, we choose a wavelet basis that satisfies cancellation
condition  of order at least $\beta$. Then the proof of the equivalence between the weighted Hardy space $H^1_{\nu,wavelet}(\R^n)$
and the atomic Hardy space $H^{1,p,\beta}_{\nu}(\mathbb R^n)$ follows from a standard approach.
Next we point out that the results of steps (2) and (3) together imply that
 $H^1_{\nu,wavelet}(\R^n)$ is independent of the choice of wavelet basis. Thus, we further obtain that
the spaces $H^{1,p,\beta}_{\nu}(\mathbb R^n)$ are equivalent for arbitrary integers $\beta\geq0$.

Combining the results from above, we obtain that
for $\nu\in A^2_{\Delta_N}(\mathbb R^n)$, $f\in {\rm BMO}_{\Delta_N,\nu}(\mathbb R^n)$ if and only if $f_{+,e} \in {\rm BMO}_{\nu_{+,e}}(\mathbb R^n)$ and
$f_{-,e} \in {\rm BMO}_{\nu_{-,e}}(\mathbb R^n)$.
To obtain the upper bound
in Theorem \ref{t:upp-low Neum Lap wh-sp commu}, we make good use of the structure of reflection which allows us to go back to the classical Riesz transform, and hence the upper bound follows from the result in \cite{HLW}.
To obtain the lower bound, we use the fundamental technique of Fourier expansions (studied in \cite{CRW}, \cite{J}), and the structure of the odd and even extension.


This paper is organized as follows.  In Section \ref{s:Pre}, we 
collect the basic facts related to the Neumann Laplacian $\Delta_{N_+}$ and the reflection Neumann Laplacian  $\Delta_{N}$, and their related Riesz transforms.
In Section \ref{s:NeumannAp},  we introduce the class of weights $A^p_{\Delta_N}(\mathbb R^n)$
for $p\in[1, \infty)$ associated with $\Delta_N$, and provide the proofs of Theorems \ref{t-Ap Nuem char Riesz trans} and \ref{t-Ap Neumm log-exp brid}.
In Section \ref{s:BloomBMOHardy}, we show that for any $w\in A^p(\mathbb R^n)$ with
$p\in[1, 2]$ and for any $\beta\geq0$, the dual space of $H^{1,p,\beta}_{w}(\mathbb R^n)$ is
${\rm BMO}_{w}(\mathbb{R}^n)$ (Theorem \ref{t:dual BMO cla}).
In Section \ref{s:BloomBMOHardyNuemann}, we prove that for  $w\in A^p(\mathbb R^n)$ with $1<p\leq2$,
$H^1_{w}(\mathbb{R}^n)$
is  equivalent to
$H^1_{\Delta,w}(\mathbb{R}^n)$, and then we further obtain that ${\rm BMO}_{w}(\mathbb{R}^n)$
and ${\rm BMO}_{\Delta,w}(\mathbb{R}^n)$ coincide.
In Section \ref{s:BloomBMOHardyLaplace}, we provide the characterisation of ${\rm BMO}_{\Delta_N,\nu}(\mathbb R^n)$.
Section \ref{s:upp-low Neum Lap commu} is devoted the proofs of Theorems \ref{t:upp-low Neum Lap ha-sp commu} and
\ref{t:upp-low Neum Lap wh-sp commu}.
In Section \ref{s:Dirichlet} we study the property of the Dirichlet Laplacian on $\R^n_+$ and its corresponding Riesz transform, and then present the proof of Theorem \ref{t:Dirich commuta counterex}.


\section{Preliminaries}
\setcounter{equation}{0}
\label{s:Pre}

We now recall some notation and basic facts introduced in \cite{DDSY}*{Section 2}. For any subset $A\subset \mathbb{R}^n$ and a function $f: \mathbb{R}^n\rightarrow\mathbb{C}$ by $f|_A$ we denote the restriction of $f$ to $A$.   Next we set $\mathbb{R}^n_+=\{(x',x_n)\in \mathbb{R}^n: x'=(x_1,\ldots,x_{n-1})\in \mathbb{R}^{n-1},x_n>0\}$.
For any function $f$ on
$\mathbb{R}^n$, we set
$$ f_+=f|_{\mathbb{R}^n_+}\ \ \ {\rm and}\ \ \ f_-=f|_{\mathbb{R}^n_-}. $$
 For any $x=(x',x_n)\in\mathbb{R}^n$ we set $\widetilde{x}=(x',-x_n)$. If $f$ is any function defined on $\mathbb{R}^n_+$, its even extension defined on $\mathbb{R}^n$ is
\begin{align}\label{f e}
f_e(x)=f(x),\ {\rm if }\ x\in \mathbb{R}^n_+;\ \ f_e(x)=f(\widetilde{x}),\ {\rm if }\ x\in \mathbb{R}^n_-.
\end{align}

\subsection{The Neumann Laplacian}
\label{s:Neumann}

We denote by $\Delta_n$ the Laplacian on $\mathbb{R}^n$. Next we recall the Neumann Laplacian on $\mathbb{R}^n_+$ and $\mathbb{R}^n_-$.

Consider the Neumann problem on the half line $(0,\infty)$ (see \cite{S}*{(7), page 59 in Section 3.1}):
\begin{align}\label{Neumann}
\left\{
\begin{array}{lcc}
 w_t-w_{xx}  =0 & {\rm for\ }  0<x<\infty, 0<t<\infty, \\
 w(x,0)=\phi(x), &  \\
 w_x(0,t)=0. &
\end{array}
\right.
\end{align}
Denote this corresponding Laplacian by $\Delta_{1,N_+}$. According to \cite{S}*{(7), Section 3.1}, we see that
$ w(x,t) = e^{-t\Delta_{1,N_+}}(\phi)(x). $
 For $n>1$, we write $\mathbb{R}_+^n= \mathbb{R}^{n-1}\times \mathbb{R}_+$. And we define the Neumann Laplacian on $\mathbb{R}^n_+$ by
$$ \Delta_{n,N_+} = \Delta_{n-1} + \Delta_{1,N_+}, $$
where $\Delta_{n-1}$ is the  Laplacian on $\mathbb{R}^{n-1}$. 
Similarly we can define Neumann Laplacian  $\Delta_{n,N_-}$ on $\mathbb{R}^n_-$.




In the remainder of the paper, we omit the index $n$,   we denote by $\Delta$ the Laplacian on $\mathbb{R}^n$, denote the Neumann Laplacian on $\mathbb{R}^n_+$ by $\Delta_{N_+}$, and
Neumann Laplacian on $\mathbb{R}^n_-$ by~$\Delta_{N_-}$.


The Laplacian and Neumann Laplacian $\Delta_{N_{\pm}}$ are positive definite self-adjoint operators. By the spectral theorem one can define the semigroups generated by these operators $\{\exp(-t\Delta), t\geq 0\}$ and $\{\exp(-t\Delta_{N_\pm}), t\geq 0\}$. By $p_t(x,y)$, $p_{t,\Delta_{N_+}}(x,y)$ and $p_{t,\Delta_{N_-}}(x,y)$ we denote the heat kernels corresponding to the semigroups generated by $\Delta$, $\Delta_{N_+}$ and $\Delta_{N_-}$,
respectively.  Then we have
\begin{eqnarray}\label{heat ker classical}
p_{t}(x,y) = \frac{1}{(4\pi t)^{\frac{n}{2}}}e^{-\frac{|x-y|^2}{4t}}.
\end{eqnarray}
From the reflection method (see \cite{S}*{(9), page 60 in Section 3.1}), we get
\begin{eqnarray*}
p_{t,\Delta_{N_+}}(x,y)  = \frac{1}{(4\pi t)^{\frac{n}{2}}}e^{-\frac{|x'-y'|^2}{4t}}\left( e^{-\frac{|x_n-y_n|^2}{4t}}+e^{-\frac{|x_n+y_n|^2}{4t}} \right), \ \ x,y\in \mathbb{R}^n_+;\\
p_{t,\Delta_{N_-}}(x,y) = \frac{1}{(4\pi t)^{\frac{n}{2}}}e^{-\frac{|x'-y'|^2}{4t}}\left( e^{-\frac{|x_n-y_n|^2}{4t}}+e^{-\frac{|x_n+y_n|^2}{4t}} \right), \ \ x,y\in \mathbb{R}^n_-.
\end{eqnarray*}

For any function $f$ on $\mathbb{R}^n_+$, we have
\begin{align}\label{e:identity semigroup}
\exp(-t\Delta_{N_+})f(x)=\exp(-t\Delta)f_e(x)
 \end{align}
 for all $t\geq0$ and $x\in \mathbb{R}^n_+$. Similarly, for any function $f$ on $\mathbb{R}^n_-$, $\exp(-t\Delta_{N_-})f(x)=\exp(-t\Delta)f_e(x)$ for all $t\geq0$ and $x\in \mathbb{R}^n_-$.

Now let $\Delta_N$ be the uniquely determined unbounded operator acting on $L^2(\mathbb{R}^n)$ such that
\begin{align}\label{Delta N}
 (\Delta_Nf)_+=\Delta_{N_+}f_+ \ \ \ {\rm and}\ \ \ (\Delta_Nf)_-=\Delta_{N_-}f_-
\end{align}
for all $f: \mathbb{R}^n\rightarrow \mathbb{R}$ such that $f_+\in W^{1,2}(\mathbb{R}^n_+)$ and $f_-\in W^{1,2}(\mathbb{R}^n_-)$. Then
$\Delta_N$ is a positive self-adjoint operator and
\begin{align}\label{Delta N-exp}
\ (\exp(-t\Delta_N)f)_+=\exp(-t\Delta_{N_+})f_+\ \ \ {\rm and}\ \ \
(\exp(-t\Delta_N)f)_-=\exp(-t\Delta_{N_-})f_-.
\end{align}
The heat kernel of $\exp(-t\Delta_N)$, denoted by $p_{t,\Delta_N}(x,y)$, is then given as:
\begin{eqnarray}
\label{def-of-ptN}
p_{t,\Delta_N}(x,y)=\frac{1}{(4\pi t)^{\frac{n}{2}}}e^{-\frac{|x'-y'|^2}{4t}}\big( e^{-\frac{|x_n-y_n|^2}{4t}}+e^{-\frac{|x_n+y_n|^2}{4t}} \big)H(x_ny_n),
\end{eqnarray}
where $H: \mathbb{R}\rightarrow\{0,1\}$ is the Heaviside function given by
\begin{equation}
\label{e:hvyside}
H(t)=0,\ \ {\rm if}\ t<0;\ \ \ \ \ H(t)=1,\ \ {\rm if}\ t\geq0.
\end{equation}

Let us note that

\begin{itemize}
\item[$(\alpha)$] All the operators $\Delta, \Delta_{N_+}, \Delta_{N_-}, $ and  $\Delta_{N}$ are self-adjoint and they generate bounded analytic positive semigroups acting on all $L^p(\mathbb{R}^n)$ spaces for $1\leq p\leq \infty$;

\item[$(\beta)$] Suppose that $p_{t,L}(x,y)$ is the kernel corresponding to the semigroup generated by one of the operators $L$ listed in $(\alpha)$. Then the kernel $p_{t,L}(x,y)$ satisfies Gaussian bounds:
\begin{align}
\label{size of Neumann heat kernel}
 |p_{t,L}(x,y)|\leq \frac{C}{t^{\frac{n}{2}}} e^{-c\frac{|x-y|^2}{t}},
\end{align}
for all $x,y\in \Omega$, where $\Omega=\mathbb{R}^n$ for $\Delta, \Delta_{N}$; $\Omega=\mathbb{R}^n_+$ for $\Delta_{N_+}$ and $\Omega=\mathbb{R}^n_-$ for $\Delta_{N_-}$.
\end{itemize}

Next we consider the smoothness property of the heat kernel for $\Delta_N$, $\Delta_{N_+}$, and $\Delta_{N_-}$.

\begin{prop}[\cite{LW}] Suppose that $L$ is one of the operators $ \Delta_{N_+}$, $ \Delta_{N_-}$ and $\Delta_{N}$. Then for $x,x',y\in\mathbb{R}^n_+$ $($or $x,x',y\in\mathbb{R}^n_-$$)$ with
 $|x-x'|\leq \frac{1}{2}|x-y|$, we have
\begin{align}\label{smooth x}
|p_{t,L}(x,y)- p_{t,L}(x',y)| \leq C\frac{|x-x'|}{(\sqrt{t}+|x-y|)}\frac{\sqrt{t}}{(\sqrt{t}+|x-y|)^{n+1}};
\end{align}
symmetrically, for $x,y,y'\in\mathbb{R}^n_+$ $($or $x,x',y\in\mathbb{R}^n_-$$)$ with
$|y-y'|\leq \frac{1}{2}|x-y|$, we have
\begin{align}\label{smooth y}
|p_{t,L}(x,y)- p_{t,L}(x,y')| \leq C\frac{|y-y'|}{(\sqrt{t}+|x-y|)}\frac{\sqrt{t}}{(\sqrt{t}+|x-y|)^{n+1}}.
\end{align}
\end{prop}

\subsection{The Riesz Kernels Associated to the Neumann Laplacian}

A fundamental object in our study are the Riesz transforms associated to the Neumann Laplacian.  Recall that the Riesz transforms associated to the Neumann Laplacian are given by: $R_N= \nabla \Delta_N^{-\frac{1}{2}}$.
We collect the formula for these kernels in the following proposition.

\begin{prop}
\label{p:RieszKernel}
Denote 
by $R_{N,j}(x,y)$ the kernel of the $j$-th Riesz transform $\frac{\partial}{\partial x_j} \Delta_N^{-\frac{1}{2}}$ of $\Delta_N$.  Then for $1\leq j\leq n-1$ and for $x,y\in\mathbb{R}^n_+$ (or $x,y\in\mathbb{R}^n_-$) we have:
$$
R_{N,j}(x,y)= - C_n \bigg( \frac{x_j-y_j}{|x-y|^{n+1}} +   \frac{x_j-y_j}{(|x'-y'|^2+|x_n+y_n|^2)^{\frac{n+1}{2}}}\bigg);
$$
and for $j=n$ and for $x,y\in\mathbb{R}^{n}_+$ (or $x,y\in\mathbb{R}^n_-$) we have:
$$
R_{N,n}(x,y)= - C_n \bigg( \frac{x_j-y_j}{|x-y|^{n+1}} +   \frac{x_n+y_n}{(|x'-y'|^2+|x_n+y_n|^2)^{\frac{n+1}{2}}}\bigg),
$$
where $C_n=\frac{\Gamma\big(\frac{n+1}{2}\big)}{(\pi )^{\frac{n+1}{2}}}  $.
\end{prop}

The kernels $R_{N,j}(x,y)$ are Calder\'on--Zygmund kernels.
\begin{prop}[\cite{LW}]
Denote by $R_N(x,y)$ the kernel of the vector of Riesz transforms $\nabla \Delta_N^{-\frac{1}{2}}$.  Then:
\begin{align}\label{Riesz kernel}
R_N(x,y)=\big( R_{N,1}(x,y),\ldots, R_{N,n}(x,y) \big)H(x_ny_n),
\end{align}
with $H(t)$ the Heavyside function defined in \eqref{e:hvyside}.  Moreover, we have that for $x\not=y$
\begin{align*}
|R_{N}(x,y)|\leq C_n  \frac{1}{|x-y|^{n}};
\end{align*}
and that for $x,x_0,y\in\mathbb{R}^n_+$ $($or $x,x_0,y\in\mathbb{R}^n_-)$
  with $|x-x_0|  \leq \frac{1}{2} |x-y|$,
\begin{align*}
|R_{N}(x,y)-R_{N}(x_0,y)|+|R_{N}(y,x)-R_{N}(y,x_0)|\leq C\frac{|x-x_0|}{|x-y|^{n+1}}.
\end{align*}
\end{prop}

\section{\texorpdfstring{Muckenhoupt weights associated with the Neumann Laplacian $\Delta_N$}{Muckenhoupt weights associated with the Neumann Laplacian}}
\setcounter{equation}{0}
\label{s:NeumannAp}

In this section, we introduce and study a class of weights associated with
$\Delta_N$. To this end,
we first recall the classical Muckenhoupt $A^p$ weights on $\mathbb R^n$.


\begin{definition}
Suppose $w\in L^1_{loc}(\mathbb R^n)$, $w\geq0$, and $1<p<\infty$. We say that $w$ is a  Muckenhoupt $A^p(\mathbb R^n)$ weight
if there exists a constant $C$ such that
\begin{align} \label{Ap}
\sup_{Q} \left\langle w\right\rangle_Q   \left\langle w^{-{1\over p-1}}\right\rangle_Q^{p-1}\leq C<\infty,
\end{align}
where the supremum is taken over all cubes $Q$ in $\mathbb R^n$.
We denote by $[w]_{A^p}$ the smallest constant $C$ such that \eqref{Ap} holds.

The class $A^1(\mathbb R^n)$ consists of the weights $w$ satisfying for some $C > 0$ that
$$\left\langle w\right\rangle_Q \leq C {\rm ess}\inf_{x\in Q} w(x)$$
 for any $Q\subset \R^n$.
We denote by $[w]_{A^1}$ the smallest constant $C$ such that the above inequality holds.
\end{definition}


We now recall some basic properties of the Muckenhoupt $A^p(\mathbb R^n)$ weights.
If $w\in A^p(\mathbb R^n)$ with $p>1$, then the ``conjugate'' weight
\begin{align}\label{conjugate weight}
w'= w^{1-p'} \in A^{p'}(\mathbb R^n)
\end{align}
 with $[w']_{A^{p'}}= [w]_{A^p}^{p'-1}$, where $p'$ is the conjugate index of $p$, i.e., $1/p +1/p' =1$.
Moreover, suppose $\mu,\lambda\in A^p(\mathbb R^n)$ with $1<p<\infty$. Define
\begin{align}\label{Bloom weight}
\nu= \mu^{1\over p}\lambda^{-{1\over p}}.
\end{align}
Then we have that $\nu\in A^2(\mathbb R^n)$, see \cite{HLW}*{Lemma 2.19}. Moreover, we have the following
fundamental result (see \cite{HLW}*{equation (2.21)}): for any ball $B \subset \mathbb R^n$,
\begin{align}\label{Bloom weight2}
\Big({\mu(B)\over |B|}\Big)^{1\over p} \Big({\lambda'(B)\over |B|}\Big)^{{1\over p'}}\ls
{1\over  \Big({\mu'(B)\over |B|}\Big)^{1\over p'} \Big({\lambda(B)\over |B|}\Big)^{{1\over p}}}
\ls {1\over   {\nu^{-1}(B)\over |B|} } \ls {\nu(B)\over |B|}.
\end{align}

We now define the  Muckenhoupt weights associated with the Neumann Laplacian $\Delta_N$.
\begin{definition}\label{d-Ap Delta}
Suppose $w\in L^1_{loc}(\mathbb R^n)$, $w\geq0$, and $1<p<\infty$. We say that $w$ is a  Muckenhoupt weight associated with the Neumann Laplacian $\Delta_N$, denoted by $A^p_{\Delta_N}(\mathbb R^n)$,
if both $w_{+,e}$ and $w_{-,e}$ are in classical $A^p(\mathbb R^n)$.
And we define $[w]_{A^p_{\Delta_N}} = [w_{+,e}]_{A^p}+[w_{-,e}]_{A^p}$.
\end{definition}

From Definition \ref{d-Ap Delta}, we first observe that the class $A^p(\mathbb R^n)$ of Muckenhoupt weights
is a proper subset of $A^p_{\Delta_N}(\mathbb R^n)$.

\begin{prop}\label{p-inclusion Ap and Ap Delta}
Suppose $1<p<\infty$. Then we have
$A^p(\mathbb R^n) \subsetneq A^p_{\Delta_N}(\mathbb R^n)$.

\end{prop}
\begin{proof}
Suppose $1<p<\infty$ and $w\in A^p(\mathbb R^n)$.
By definition, it is direct that both $w_{+,e}$ and $w_{-,e}$
are in $A^p(\mathbb R^n)$, with
$$ [w_{+,e}]_{A^p}+ [w_{-,e}]_{A^p} \leq C [w]_{A^p}. $$
Hence, we obtain that
$w\in A^p_{\Delta_N}(\mathbb R^n)$, which shows that $A^p (\mathbb R^n)\subset A^p_{\Delta_N}(\mathbb R^n)$.

Next, for any fixed $p\in (1,\infty)$, we choose the function $w(x)$ as follows.  Let $\alpha\in (0, \frac{p}{p'})$.
For $x=(x_1,\ldots,x_{n-1},{x_n})\in\mathbb R^n $, define
\begin{equation}\label{e-example Ap Delta}
    w(x)=\left\{
                \begin{array}{ll}
                  x_n^{\alpha},& \quad x_n>0;\\[5pt]
                  1,& \quad x_n<0.\\
                                 \end{array}
              \right.
\end{equation}

Then it is clear that $w(x)$ is not in classical $A^p(\mathbb R^n)$. In
fact, choose the cube $Q_a=[-a, a]^n$ with $a>1$. Then
we have
$$ \left\langle w\right\rangle_{Q_a}   \left\langle w^{-{1\over p-1}}\right\rangle_{Q_a}^{p-1}\sim a^\alpha\to\infty
$$
as $a\to \infty$.

However, both $w_{+,e}(x)$ and $w_{-,e}(x)$ are in $A^p(\mathbb R^n)$.
As a consequence, we have that $w(x)\in A^p_{\Delta_N}(\mathbb R^n)$, which shows that
$A^p(\mathbb R^n) \subsetneq A^p_{\Delta_N}(\mathbb R^n).$
\end{proof}

\begin{remark}\label{r-Ap Delta nondoubl}
From the example given in \eqref{e-example Ap Delta}, we can
further see that a weight $w\in A^p_{\Delta_N}(\mathbb R^n)$ might not satisfy the so-called
 doubling condition, that is, there exists a positive constant
 $C_0$ such that for any cube $Q\subset \mathbb R^n$,
 $w(2Q)\le C_0 w(Q).$

 In fact, let $w$ be as in \eqref{e-example Ap Delta} with
 $\alpha=\frac12$. Choose $b\in(0, \frac19)$ and $Q_b=[-\frac{5b}{16}, \frac{5b}{16}]^{n-1}\times[\frac{b}{16}, \frac{11b}{16}]$.
 Then we see that $2Q_b=[-\frac{5b}{8}, \frac{5b}{8}]^{n-1}\times[-\frac b4, b]$,
 and
 $$w(Q_b)\sim b^{n+\frac12},\,\, w(2Q_b)\sim b^n,$$
 which implies that $w$ is non-doubling.
 \end{remark}

\begin{theorem}\label{t-Ap Nuem char Riesz trans}
Suppose $w$ is a positive locally integrable function and $1<p<\infty$. Then for every $l\in\{1,\ldots,n\}$, $w\in A^p_{\Delta_N}(\mathbb R^n)$
if and only if there exits a positive constant $C$ such that for all $f\in L^{p}_{w}(\mathbb R^n)$,
\begin{align}\label{Tw}
 \| R_{N,l}(f)\|_{L^{p}_{w}(\mathbb R^n)}\leq C \|f\|_{L^{p}_{w}(\mathbb R^n)}.
\end{align}
\end{theorem}
\begin{proof}
Suppose $1<p<\infty$ and $w\in A^p_{\Delta_N}(\mathbb R^n)$. Then for $f\in L^{p}_{w}(\mathbb R^n)$ and $x\in\mathbb R^{n}_+$,
\begin{align*}
R_N(f)(x)& = \int_{\mathbb{R}^n} R_N(x,y) f(y)\,dy = \int_{\mathbb{R}^n_+} R_N(x,y) f_+(y)\,dy
  =    \int_{\mathbb{R}^n} R(x,y) f_{+,e}(y)\,dy\\
  &=  \nabla\Delta^{-\frac{1}{2}} f_{+,e}(x),
\end{align*}
where $R(x, y)$ is the kernel of $\nabla\Delta^{-\frac{1}{2}}$.
Similarly, for $x\in\mathbb R^{n}_-$, we also have
$R_N(f)(x) =\nabla\Delta^{-\frac{1}{2}} f_{-,e}(x).$  Hence,
\begin{align*}
&\int_{\mathbb R^n}|R_N(f)(x)|^p\, w(x)dx\\
&\quad = \int_{\mathbb R_{+}^n}|R_N(f)(x)|^p\, w(x)dx+\int_{\mathbb R_{-}^n}|R_N(f)(x)|^p\, w(x)dx\\
  &\quad= \int_{ R_{+}^n}  |\nabla\Delta^{-\frac{1}{2}} f_{+,e}(x)|^p w_{+,e}(x)dx+ \int_{\mathbb R_{-}^n} | \nabla\Delta^{-\frac{1}{2}} f_{-,e}(x)|^p w_{-,e}(x)dx\\
  &\quad\leq \int_{\mathbb R^n}  |\nabla\Delta^{-\frac{1}{2}} f_{+,e}(x)|^p w_{+,e}(x)dx+ \int_{\mathbb R^n} | \nabla\Delta^{-\frac{1}{2}} f_{-,e}(x)|^p w_{-,e}(x)dx\\
  &\quad\leq C\|f_{+,e}\|^p_{L^p_{w_{+,e}}(\mathbb R^n)}
+C\|f_{-,e}\|^p_{L^p_{w_{-,e}}(\mathbb R^n)}\\
&\quad\leq C\|f\|_{L^p_w(\mathbb R^n)},
\end{align*}
which implies that \eqref{Tw} holds.

Conversely, suppose  that \eqref{Tw} holds. Then for $f\in L^{p}_{w}(\mathbb R^n)$,
it is direct that $f_+\in L^{p}_{w}(\mathbb R^n)$. Hence, we have
\begin{align*}
&\int_{\mathbb R^n}|R_N(f_+)(x)|^p\, w(x)dx \leq C \|f_+\|_{L^{p}_{w}(\mathbb R^n)}^p.
\end{align*}
By noting that
\begin{align*}
\int_{\mathbb R^n}|R_N(f_+)(x)|^p\, w(x)dx&=\int_{\mathbb R_{+}^n}|R_N(f_+)(x)|^p\, w_+(x)dx
\end{align*}
and that
 $\|f_+\|_{L^{p}_{w}(\mathbb R^n)}^p=  \|f_{+}\|_{L^{p}_{w_{+}}(\mathbb R^n)}^p$, 
 we have
\begin{align}\label{Tw+}
&\int_{\mathbb R_{+}^n}|R_N(f_+)(x)|^p\, w_+(x)dx  \leq C \|f_{+}\|_{L^{p}_{w_{+}}(\mathbb R^n)}^p.
\end{align}
Symmetrically we obtain that
\begin{align}\label{Tw-}
&\int_{\mathbb R_{-}^n}|R_N(f_-)(x)|^p\, w_-(x)dx  \leq C \|f_{-}\|_{L^{p}_{w_{-}}(\mathbb R^n)}^p.
\end{align}

Now consider the $j$-th Riesz transform $R_{N,j}$ for $j=1,\ldots,n-1$.
For any fixed cube $Q\subset \mathbb R^n_+$, consider $Q' $ the translation of $Q$ along the $j$-th direction only for the length $4\ell(Q)$.
Then it is obvious that $Q'$ is also in $ \mathbb R^n_+$.
Now for every $f\in L^1(\mathbb R^n_+)$, $f\geq0$ and supp$f\subset Q$, and for every $x\in Q'$, from the definition of $R_{N,j}(f)(x)$, we have
\begin{align*}
|R_{N,j}(f)(x)|
&=C_n \int_{\mathbb R^n_+} \bigg( {|x_j-y_j|\over |x-y|^{n+1}} +   \frac{|x_j-y_j|}{(|x'-y'|^2+|x_n+y_n|^2)^{\frac{n+1}{2}}}\bigg)\, f(y)dy\\
&\geq C_n \int_{\mathbb R^n_+}  {1\over |x-y|^{n}} f(y)dy\\
&\geq C_n  \langle f \rangle_Q.
\end{align*}
It follows that for all $0<\alpha< C_n \langle f\rangle_Q$, we have
$ Q'\subseteq \{ x\in\mathbb R^{n}_+:\ |R_{N,j}(f)(x)|>\alpha \}. $
Since  \eqref{Tw+} holds, we obtain that
$$ w_+(Q') \leq {C\over \alpha^p} \int_Q f(y)^pw_+(y)dy $$
for all  $0<\alpha< C_n  \langle f \rangle_Q$.
As a consequence, we have
$$  \langle f \rangle_Q^p \leq {C\over w_+(Q')}\int_Q f(y)^pw_+(y)dy. $$
In particular, by taking $f=\unit_Q$, we obtain that $w_+(Q')\leq Cw_+(Q)$.


Symmetrically we can also reverse the roles of $Q$ and $Q'$ to obtain
$$ \langle g\rangle_{Q'}^p \leq {C\over w_+(Q)}\int_{Q'} g(y)^pw_+(y)dy $$
for all functions $g\in L^1(\mathbb R^n_+)$, $g\geq0$ and supp$g\subset Q'$.
By taking $g=\unit_{Q'}$, we obtain that $w_+(Q)\leq Cw_+(Q')$.

Then we have that for any cube $Q$ and $f\ge 0$,
\begin{align*}
w_+(Q)  \langle f \rangle_Q^p& \leq C w_+(Q)  {1\over w_+(Q')}\int_Q f(y)^pw_+(y)dy\leq C\int_Q f(y)^pw_+(y)dy,
\end{align*}
which shows that
\begin{align*}
 \langle f \rangle_Q
&\leq \bigg({C\over w_+(Q)}\int_Q f(y)^pw_+(y)dy\bigg)^{1\over p}.
\end{align*}
Now by taking $f(x) = w_+^{-{1\over p-1}}(x)\unit_Q(x)$, we get that
$\langle w_+\rangle_Q\big\langle w_+^{-{1\over p-1}}\big\rangle^{ p-1} \leq C.$
This shows that $w_+(x)$ is an $A^p$ weight in $\mathbb R^n_+$, and hence we get that
$w_{+,e}(x)$ is an $A^p$ weight in $\mathbb R^n$.
Symmetrically we have $w_{-,e}(x)$ is an $A^p$ weight in $\mathbb R^n$. Thus, we have
$w\in A^p_{\Delta_N}(\mathbb{R}^n)$.

Finally, we consider the $n$-th Riesz transform $R_{N,n}$.
For any fixed cube $Q\subset \mathbb R^n_+$, consider $Q' $ the translation of $Q$ in the positive sense along the $n$-th direction only for the length $4\ell(Q)$.
Then it is obvious that $Q'$ is also in $ \mathbb R^n_+$.
Moreover, for any $x\in Q'$ and $y\in Q$, we have
$x_n-y_n>4\ell(Q)$ and $x_n-y_n \approx |x-y|$.

Now for every $f\in L^1(\mathbb R^n_+)$, $f\geq0$ and supp$f\subset Q$, and for every $x\in Q'$, we have
\begin{align*}
|R_{N,n}(f)(x)|
&=C_n \int_{\mathbb R^n_+} \bigg( {x_n-y_n\over |x-y|^{n+1}} +   \frac{ x_n+y_n}{(|x'-y'|^2+|x_n+y_n|^2)^{\frac{n+1}{2}}}\bigg)\, f(y)dy\\
&\geq C_n \int_{\mathbb R^n_+}  {1\over |x-y|^{n}} f(y)dy\\
&\geq C_n \left< f\right>_Q.
\end{align*}
Then, following the same estimates as those for $R_{N,j}$ with $j<n$, we obtain that
$w_{+,e}(x)$ is an $A^p$ weight in $\mathbb R^n$.

As for $w_{-}$, for any fixed cube $Q\subset \mathbb R^n_-$, consider $Q' $ the translation of $Q$ in the negative direction along the $n$-th direction only for the length $4\ell(Q)$.
Then it is obvious that $Q'$ is also in $ \mathbb R^n_-$.
Moreover, for any $x\in Q'$ and $y\in Q$, we have
$x_n-y_n<-4\ell(Q)$ and $|x_n-y_n| \approx |x-y|$.
Now for every $f\in L^1(\mathbb R^n_-)$, $f\geq0$ and supp$f\subset Q$, and for every $x\in Q'$, we have
\begin{align*}
|R_{N,n}(f)(x)|
&=C_n \int_{\mathbb R^n_+} \bigg( {|x_n-y_n|\over |x-y|^{n+1}} +   \frac{ |x_n+y_n|}{(|x'-y'|^2+|x_n+y_n|^2)^{\frac{n+1}{2}}}\bigg)\, f(y)dy\\
&\geq C_n \int_{\mathbb R^n_+}  {1\over |x-y|^{n}} f(y)dy\\
&\geq C_n \left< f\right>_Q.
\end{align*}
Then, following the same estimates as those for $R_{N,j}$ with $j<n$, we obtain that
$w_{-,e}(x)$ is an $A^p$ weight in $\mathbb R^n$.  Combining all these fact, we get that $w\in A^p_{\Delta_N}(\mathbb R^n)$.
\end{proof}

\begin{theorem}\label{t-Ap Neumm log-exp brid}
Suppose $1<p<\infty$ and
$w\in A^p_{\Delta_N}(\mathbb R^n)$. Then we have
$\log w \in {\rm BMO}_{\Delta_N}(\mathbb{R}^n)$.
Conversely, for every $f\in {\rm BMO}_{\Delta_N}(\mathbb{R}^n)$, there exists $\delta>0$ such that
$e^{\delta f} \in A^p_{\Delta_N}(\mathbb R^n)$.
\end{theorem}

\begin{proof}
Suppose $1<p<\infty$ and
$w\in A^p_{\Delta_N}(\mathbb R^n)$. Then we have $(\log w)_{+,e} = \log w_{+,e}$. Since $w_{+,e}\in A^p(\mathbb R^n)$, we get that
$ \log w_{+,e} $ is in BMO$(\mathbb R^n)$, which shows that $(\log w)_{+,e}\in {\rm BMO}(\mathbb R^n)$. Similarly, we obtain that $(\log w)_{-,e}\in {\rm BMO}(\mathbb R^n)$. As a consequence, we get that
$\log w \in {\rm BMO}_{\Delta_N}(\mathbb{R}^n)$.

Conversely,  for every $f\in {\rm BMO}_{\Delta_N}(\mathbb{R}^n)$, we know that both  $f_{+,e}$ and  $f_{-,e}$
are in ${\rm BMO}(\mathbb R^n)$, which shows that there exist $\delta_1,\delta_2>0$ such that
$e^{\overline{\delta_1} f_{+,e}} ,  e^{\overline{\delta_2} f_{-,e}} \in A^p(\mathbb R^n)$ for all $\overline{\delta_1}\in(0,\delta_1)$ and $\overline{\delta_2}\in(0,\delta_2)$, respectively. Then we set $\delta=\min\{\delta_1, \delta_2\}$ and it is direct to see that $e^{\delta f_{+,e}} ,  e^{\delta f_{-,e}} \in A^p(\mathbb R^n)$, i.e.,
$\big(e^{\delta f}\big)_{+,e} , \big( e^{\delta f}\big)_{-,e} \in A^p(\mathbb R^n)$. Hence, we get that
$e^{\delta f} \in A^p_{\Delta_N}(\mathbb R^n)$.
\end{proof}

\section{Characterizations of \texorpdfstring{$H^1_{w}(\mathbb{R}^n)$}{weight Hardy spaces},  \texorpdfstring{${\rm BMO}_{w}(\mathbb{R}^n)$}{weighted BMO} and duality}
\setcounter{equation}{0}
\label{s:BloomBMOHardy}

In this section,  we make an intensive study of the classical weighted BMO and Hardy spaces introduced in Muckenhoupt--Wheeden \cites{MW76,MW78} and further studied by Garc\'ia--Cuerva \cite{Ga}. 
We begin with the definition of Muckenhoupt--Wheeden  weighted BMO space as follows.


\subsection{\texorpdfstring{The John--Nirenberg inequality for ${\rm BMO}_{w}(\mathbb{R}^n)$}{The John--Nirenberg inequality for weighted BMO}}

\begin{definition}[\cite{MW76}]\label{def: BMOw}
Suppose $1<p<\infty$ and $w\in A^p(\mathbb R^n)$. The weighted BMO space is defined as
${\rm BMO}_{w}(\mathbb{R}^n):=\{f\in L^1_{\rm loc}(\mathbb R^n): \|f\|_{{\rm BMO}_{w}(\mathbb{R}^n)}<\infty\},$
where
$$\|f\|_{{\rm BMO}_{w}(\mathbb{R}^n)}=\sup_Q\frac1{w(Q)}\int_Q\left|f(x)-\left< f\right>_Q\right|\,dx.$$
\end{definition}

The following result, which is a weighted version of the John-Nirenberg theorem, appeared first
in \cite{MW}, where the Muckenhoupt $A^p$ characteristic was not tracked. 

\begin{theorem}\label{T:MW}
Suppose $1<p<\infty$ and $w\in A^p(\mathbb R^n)$.
Let $b\in {\rm BMO}_{w}(\mathbb{R}^n)$. Then for any $1 \leq r \leq p'$, we have
\begin{align}
\|b\|_{{\rm BMO}_{w}(\mathbb{R}^n)} \approx \|b\|_{{\rm BMO}_{w,r}(\mathbb{R}^n)}:=
\bigg(\sup_Q\frac1{w(Q)}\int_Q\left|b(x)-\left< b\right>_Q\right|^r\, w^{1-r}(x)dx\bigg)^{1\over r}.
\end{align}
In particular, we have
\begin{align}\label{eq:MW}
\|b\|_{{\rm BMO}_{w}(\mathbb{R}^n)} \leq \|b\|_{{\rm BMO}_{w,r}(\mathbb{R}^n)} \leq C_{n,p,r} [w]_{A^p}^{\max\{1,{1\over p-1}\}}\|b\|_{{\rm BMO}_{w}(\mathbb{R}^n)}.
\end{align}
\end{theorem}

\begin{proof}

In what follows, using the techniques of sparse operators, we provide a modern proof of the dyadic version of this result, i.e., the suprema is
taken over all dyadic cubes $Q$ only in both the norm of ${\rm BMO}_{w}(\mathbb{R}^n)$ and ${\rm BMO}_{w,r}(\mathbb{R}^n)$.
Then we can get back to the general version by using the $1/3$-trick and the bridge
between continuous and dyadic BMO spaces via intersection, see for example \cite{LPW}.

\subsubsection{Sparse collections and operators}
Suppose $\cd$ is a dyadic lattice on $\R^n$, that is, a collection of cubes with the properties:
	\begin{itemize}
	\item Every $Q\in\cd$ has sidelength $l(Q) = 2^{-k}$ for some integer $k$;
	\item $P \cap Q \in \{ P, Q, \emptyset \}$ for every $P, Q \in \cd$;
	\item The cubes $Q \in \cd$ with $l(Q) = 2^{-k}$, for some fixed integer $k$, partition $\R^n$.
	\end{itemize}
Recall that every dyadic interval $I$ in $\R$ is associated with two Haar functions:
	$$ h_I^{0} := \frac{1}{\sqrt{|I|}} (\unit_{I-} - \unit_{I_+}) \text{ and } h_I^1 := \frac{1}{\sqrt{|I|}}\unit_I, $$
the first one being cancellative (it has mean $0$), the second being non-cancellative (it has a non-zero mean).
Given a dyadic grid $\cd$ on $\R^{n}$, every dyadic cube $Q = I_1 \times\cdots\times I_n$, where all $I_i$ are dyadic intervals in $\R$ with common length $l(Q)$,
is associated with $2^n-1$ cancellative Haar functions:
	$$ h^\ep_{Q}(x) := h_{I_1\times\ldots\times I_n}^{(\ep_1, \ldots, \ep_n)}(x_1, \ldots, x_n) := \prod_{i=1}^n h_{I_i}^{\ep_i}(x_i),$$
where $\ep \in \{0,1\}^n\setminus \{(1, \ldots, 1)\}$ is the signature of $h_Q^{\ep}$. To simplify notation, we assume that signatures are never the identically $1$ signature, in which case
the corresponding Haar function would be non-cancellative.
The cancellative Haar functions form an orthonormal basis for $L^2(\R^n)$. We write
	$$ f = \sum_{Q\in\cd}  \left\langle f, h_Q^\ep\right\rangle h_Q^\ep, $$
where 
$\left< f, g\right> := \int_{\R^n} f(x)g(x)\,dx$, and summation over $\ep$ is assumed.

\begin{definition} \label{D:Sparse}
Given $0<\eta<1$, a collection $\cs \subset \cd$ of dyadic cubes is said to be $\eta$-sparse provided that for every $Q\in\cs$, there is a measurable subset $E_Q \subset Q$ such that
$|E_Q| \geq \eta |Q|$ and the sets $\{E_Q\}_{Q\in\cs}$ are pairwise disjoint.
\end{definition}

\begin{definition}
Given $\La > 1$, a family $\cs \subset \cd$ of dyadic cubes is said to be $\La$-Carleson provided that
	$$ \sum_{P\in\cs, P\subset Q} |P| \leq \La |Q|, $$
for all $Q\in\cs$.
\end{definition}

A remarkable property is that \cites{L,LN}:
	$ \cs \text{ is } \eta\text{-sparse } \Leftrightarrow \cs \text{ is } \frac{\displaystyle1}{\displaystyle\eta}\text{-Carleson}. $

An important particular case of sparse collections that appears frequently in practice is the following. Suppose we have a family of dyadic cubes $\cs \subset \cd$. For every
$Q\in\cs$, let $ch_{\cs}(Q)$ denote the collection of \textit{maximal} elements of $\cs$ that are \textit{strictly} contained in $Q$ -- the ``$\cs$-children'' of $Q$. Now suppose the
family $\cs$ has the property that:
	$$ \sum_{P \in ch_{\cs}(Q)} |P| \leq \frac{1}{\alpha} |Q|,$$
for some $\alpha > 1$. Then $\cs$ is a $(1 - \frac{1}{\alpha})$-sparse collection. This is easy to see by taking the sets $E_Q$ in Definition \ref{D:Sparse} to be
$E_Q := Q \setminus \bigcup_{P\in ch_{\cs}(Q)}P$.

\begin{definition}
Given a sparse collection $\cs \subset \cd$, a sparse operator is one of the form:
	$$ \ca_{\cs}f(x) := \sum_{Q\in\cs} \left< f\right>_Q \unit_Q(x). $$
\end{definition}

For any $\eta$-sparse operator $\ca_{\cs}$ we have the $A^2(\mathbb R^n)$ bound (see, for example, \cites{L,LN}):
	\begin{equation}\label{E:SparseA2}
	\| \ca_{\cs} \|_{L^2(w) \rightarrow L^2(w)} \lesssim \frac{1}{\eta}[w]_{A^2}, \forall w\in A^2,
	\end{equation}
from which it follows by extrapolation \cites{L,LN}:
	\begin{equation} \label{E:SparseAp}
	\| \ca_{\cs} \|_{L^p(w) \rightarrow L^p(w)} \lesssim \frac{1}{\eta} [w]_{A^p}^{\max\{1, \frac{1}{p-1}\}}.
	\end{equation}

\subsubsection{The weighted BMO decomposition} \label{Ss:BMOdecomp}
Turning back to the weighted BMO question, let $w$ be a weight on $\R^n$ and $b \in BMO_{\cd}(w)$. For a fixed cube $Q_0 \in \cd$ and $\alpha > 1$, consider the collection:
	$$\ce := \left\{ \text{maximal subcubes } R \in \cd(Q_0) \text{ such that } \left<w\right>_R > \alpha \left<w\right>_{Q_0} \right\},$$
and denote $E := \bigcup_{R\in\ce} R$. This is simply the collection in the Calder\'{o}n-Zygmund decomposition of $w$ over $Q_0$, so we immediately have:
	$$ \alpha \left< w\right>_{Q_0} < \left<w\right>_R \leq 2^n \alpha \left< w\right>_{Q_0}\ {\rm for\ all\ } R \in \ce,\quad{\rm and}\quad  \sum_{R\in\ce} |R| \leq \frac{1}{\alpha}|Q_0|.  $$

Now, instead of forming the usual ``good'' and ``bad'' functions for $w$, we let
	$$ a(x) := \unit_{Q_0}b(x) - \sum_{R\in\ce} (b(x) - \left<b\right>_R) \unit_R(x). $$
The ``good'' function $a$ will have the usual properties resulting from a Calder\'on-Zygmund decomposition:
	$$ a(x) = \left\{ \begin{array}{l} b(x) \text{, if } x \in Q_0\setminus E \\ \left<b\right>_R \text{, if } x\in R, R\in\ce, \end{array}
	\right.$$
and for all $Q \in \cd(Q_0)$, $Q \not\subset E$:
	$$ \left<a\right>_Q = \left<b\right>_Q \text{ and } \left\langle a, h_Q^\ep\right\rangle = \left\langle b, h_Q^\ep\right\rangle. $$
More importantly though, the function $a$ will belong to the \textit{unweighted} dyadic $BMO_{\cd}[Q_0]$ space over $Q_0$ (this simply means that we are taking supremum over the subcubes of $Q_0$):
	\begin{equation} \label{E:aBMO}
	\|a\|_{BMO_{\cd}[Q_0]} \leq 2\alpha \left<w\right>_{Q_0} \|b\|_{BMO_{\cd}(w)}.
	\end{equation}
To see this, let $Q \in \cd(Q_0)$. If $Q \subset R$ for some $R \in \ce$, then $a(x) = \left<b\right>_R$ for all $x \in Q$. Otherwise, suppose $Q \not\subset E$:
	\begin{align*}
	\frac{1}{|Q|} \int_Q |a - \left<a\right>_Q|\,dx &= \frac{1}{|Q|} \int_Q | b - \sum_{R\in\ce}(b - \left<b\right>_R)\unit_R - \left<b\right>_Q |\,dx \\
		&\leq \frac{1}{|Q|}\int_Q |b - \left<b\right>_Q|\,dx + \frac{1}{|Q|} \sum_{R\in\ce, R\subsetneq Q} \int_R |b - \left<b\right>_R|\,dx \\
		&\leq \left<w\right>_Q \|b\|_{BMO_{\cd}(w)} + \frac{1}{|Q|} \|b\|_{BMO_{\cd}(w)} \sum_{R\in\ce, R\subsetneq Q} w(R)\\
		&\leq \alpha \left<w\right>_{Q_0} \|b\|_{BMO_{\cd}(w)} + \left<w\right>_Q \|b\|_{BMO_{\cd}(w)}\\
		&\leq 2\alpha \left<w\right>_{Q_0} \|b\|_{BMO_{\cd}(w)},
	\end{align*}
where the last two inequalities follow because $Q$ was not selected for $\ce$.

A very useful application of this decomposition is that reduction to \textit{unweighted} BMO allows for an efficient way to handle Haar coefficients and averages of weighted BMO functions. For instance, the decomposition
	\begin{equation} \label{E:HaarDecomp}
	\unit_{Q_0}(x) (a(x) - \left<a\right>_{Q_0}) = \sum_{Q\in\cd(Q_0)} \left\langle a, h_Q^\ep\right\rangle h_Q^{\ep}
	\end{equation}
and an appeal to the dyadic square function yields the well-known expression for the dyadic BMO norm:
	$$ \|a\|_{BMO_{\cd}[Q_0]} \simeq \sup_{Q\in\cd(Q_0)} \bigg( \frac{1}{|Q|} \sum_{P\in\cd(Q)} |\langle a, h_P^\ep\rangle|^2 \bigg)^{1/2}. $$
From this, another frequently used inequality follows trivially:
	$$ |\langle a, h_Q^\ep\rangle| \lesssim \sqrt{|Q|} \|a\|_{BMO_{\cd}[Q_0]}, \forall Q \in \cd(Q_0), $$
all stated here locally over some $Q_0\in\cd$, but obviously hold for general $a \in BMO_{\cd}(\R^n)$.
The decomposition above yields a similar inequality for $b \in BMO_{\cd}(w)$, namely:
	$$ |\langle b, h_Q^\ep\rangle| \lesssim \sqrt{|Q|} \left<w\right>_Q \|b\|_{BMO_{\cd}(w)}, \forall Q \in \cd. $$
To see this, let $Q \in \cd$ and let $a$, supported on $Q$, be the BMO decomposition of $b$ over $Q$. Since $Q$ itself is not selected for the collection,
$\left\langle b, h_Q^\ep\right\rangle = \left\langle a, h_Q^\ep\right\rangle$, and the claim follows from \eqref{E:aBMO}.

\subsubsection{Proof of \eqref{eq:MW}}
We have $b \in BMO_{\cd}(w)$, for $w\in A^p(\mathbb R^n)$ and $1 < r < p'$.
First note that $\|b\|_{BMO_{\cd}(w)} \leq \|b\|_{BMO_{\cd}(w, r)}$ is a trivial consequence of H\"{o}lder's inequality.
For the other direction, fix $Q_0 \in \cd$ and we show that
	$$ \left( \frac{1}{w(Q_0)} \int_{Q_0} |b - \left<b\right>_{Q_0}|^r \,w^{1-r}(x)dx \right)^{1/r} \leq C_{n,p,r} [w]_{A^p}^{\max\{1, \frac{1}{p-1}\}} \|b\|_{BMO_{\cd}(w)}.$$
Note that $w^{1-r} \in A^r$ is the weight conjugate to $w$ as viewed in $A^{r'} \supset A^p$. So, by duality,
	$$ \left\| \unit_{Q_0} (b - \left< b\right>_{Q_0}) \right\|_{L^r(w^{1-r})}
	=  \sup\left\{ \left| \left< \unit_{Q_0}(b - \left< b\right>_{Q_0}), f \right> \right| : f \in L^{r'}(w), \: \|f\|_{L^{r'}(w)} \leq 1 \right\},$$
and it then suffices to show that
	$$  \left| \left< \unit_{Q_0}(b - \left< b\right>_{Q_0}), f \right> \right| \leq C_{n,p,r} [w]_{A^p}^{\max(1, \frac{1}{p-1})} \|b\|_{BMO_{\cd}(w)} \|f\|_{L^{r'}(w)} w(Q_0)^{1/r}. $$	
Now, by \eqref{E:HaarDecomp}, we have
	$$  \left| \left< \unit_{Q_0}(b - \left< b\right>_{Q_0}), f \right> \right| \leq \sum_{Q\in\cd(Q_0)} |\langle b, h_Q^\ep\rangle| |\langle f, h_Q^\ep\rangle|. $$

For some $\alpha > 1$, consider the decomposition of $b$ in Section \ref{Ss:BMOdecomp}: let $\ce_b \subset\cd(Q_0)$ be the maximal subcubes $R \in \cd(Q_0)$
such that $\left<w\right>_R > 2\alpha \left<w\right>_{Q_0}$, with $E_b := \bigcup_{R\in\ce_b}R$, and $a := \unit_{Q_0}b - \sum_{R\in\ce} (b - \left<b\right>_R)\unit_R$.
Then consider the usual Calder\'{o}n-Zygmund decomposition of $f$ over $Q_0$ at level $2\alpha$, namely let $\ce_f \subset \cd(Q_0)$ be the collection of maximal subcubes of
$Q_0$ such that $\left< |f| \right>_{R} > 2\alpha \left< |f|\right>_{Q_0}$, with $E_f := \bigcup_{R\in\ce_f} R$, and the ``good'' function
$\gamma(x) := \unit_{Q_0}(x) f(x) - \sum_{R\in\ce_f}(f(x) - \left<f\right>_R)\unit_R(x)$.
Now let $\ce$ be the collection of maximal subcubes $Q\in\cd(Q_0)$ contained in $E_b \cup E_f$, and $E := E_b \cup E_f$ (see Figure \ref{fig1}).

\begin{figure}
\begin{subfigure}{.33\textwidth}
  \centering
  \includegraphics[width=1\linewidth]{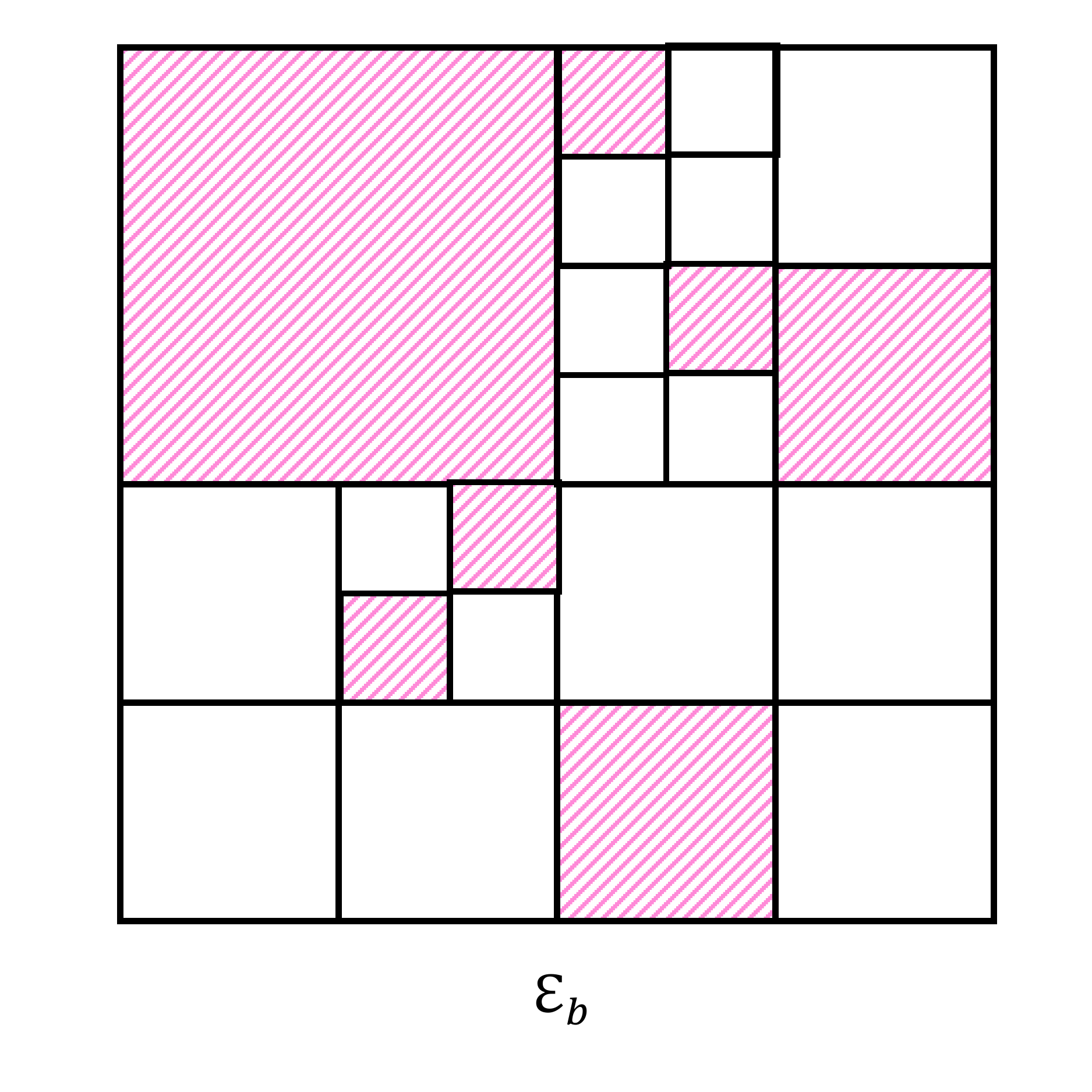}
\end{subfigure}%
\begin{subfigure}{.33\textwidth}
  \centering
  \includegraphics[width=1\linewidth]{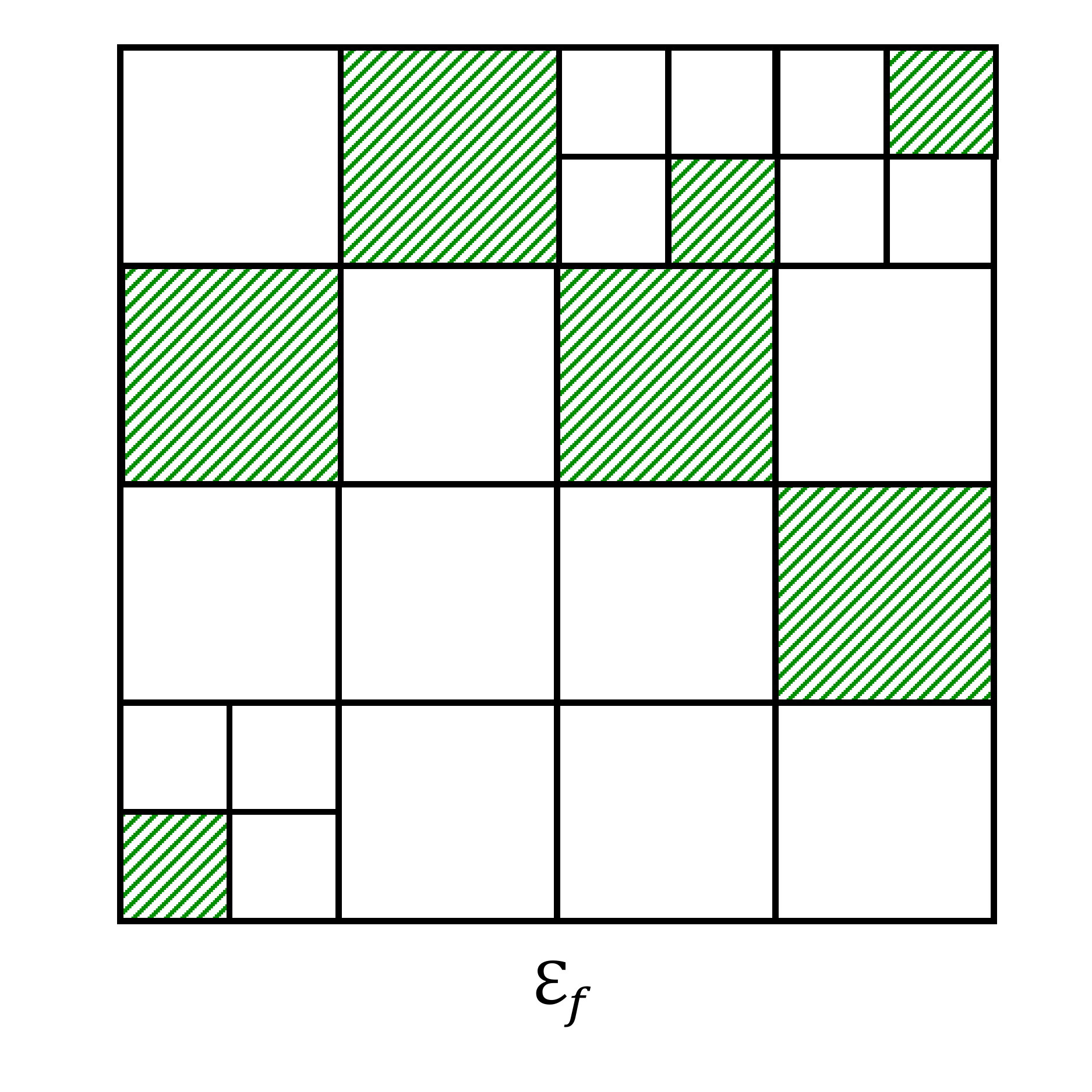}
\end{subfigure}
\begin{subfigure}{.33\textwidth}
  \centering
  \includegraphics[width=1\linewidth]{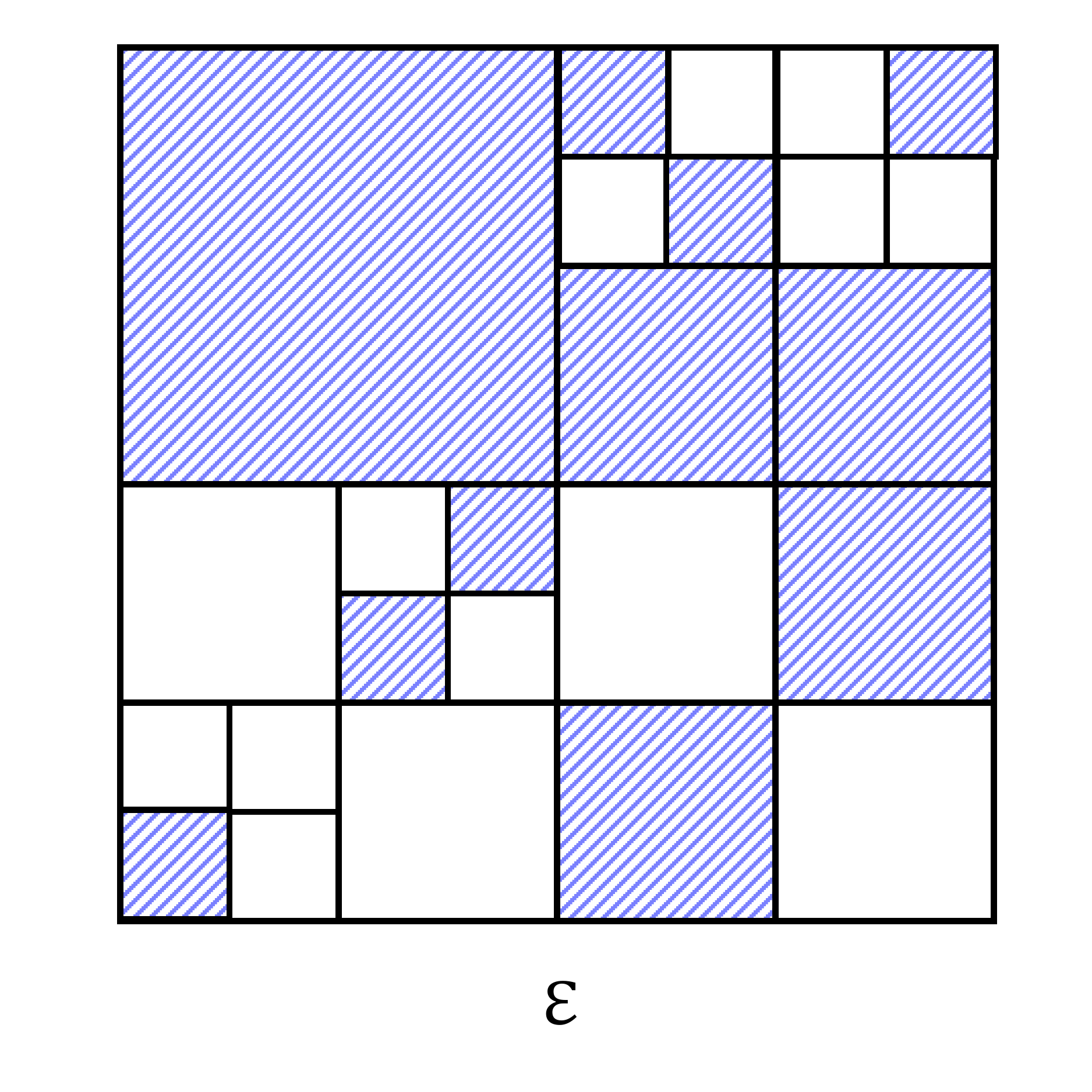}
\end{subfigure}
\caption{An example of the collections $\ce_b$, $\ce_f$, and $\ce$ -- the maximal cubes contained in $\ce_b \cup \ce_f$.}
\label{fig1}
\end{figure}

Then
	$$ \sum_{R\in\ce} |R| = |E_b \cup E_f| \leq |E_b| + |E_f| \leq \sum_{R\in\ce_b}|R| + \sum_{R\in\ce_f}|R| \leq \frac{1}{\alpha}|Q_0|. $$
Now
	\begin{align} \label{E:Temp1}
	&\sum_{Q\in\cd(Q_0)} |\langle b, h_Q^\ep\rangle| |\langle f, h_Q^\ep\rangle| \\
	&\quad=
		\sum_{Q\in \cd(Q_0),\ Q \not\subset E} |\langle b, h_Q^\ep\rangle| |\left\langle f, h_Q^\ep\right\rangle|
		+ \sum_{R\in\ce} \sum_{Q\in\cd(R)} |\langle b, h_Q^\ep\rangle| |\left\langle f, h_Q^\ep\right\rangle|.\nonumber
	\end{align}
If $Q \not\subset E$, then $Q\not\subset E_b$ and $Q\not\subset E_f$, so
	$$ \left\langle b, h_Q^\ep\right\rangle= \left\langle a, h_Q^\ep\right\rangle \text{ and } \left\langle f, h_Q^\ep\right\rangle = \left\langle \gamma, h_Q^\ep\right\rangle, $$
and the first term in \eqref{E:Temp1} becomes:
	\begin{align*}
	\sum_{Q\in \cd(Q_0); Q \not\subset E} |\langle a, h_Q^\ep\rangle| |\langle \gamma, h_Q^\ep\rangle| &\leq
		\bigg( \sum_{Q\in\cd(Q_0)} |\langle a, h_Q^\ep\rangle|^2 \bigg)^{1/2} \bigg( \sum_{Q\in\cd(Q_0)} |\langle \gamma, h_Q^\ep\rangle|^2 \bigg)^{1/2} \\
	&\leq \sqrt{|Q_0|} \|a\|_{BMO_{\cd}(Q_0)} \|\gamma\|_{L^2(Q_0)}\\
	&\leq \sqrt{|Q_0|} 4\alpha \left< w\right>_{Q_0} \|b\|_{BMO_{\cd}(w)} \cdot (2^{n+1} \alpha |Q_0| )^{1/2} \left< |f|\right>_{Q_0}\\
	&\leq C_{n,\alpha} \|b\|_{BMO_{\cd}(w)} \left< |f|\right>_{Q_0} w(Q_0).
	\end{align*}
Then we have
	\begin{align} \label{E:Temp2}
	&\sum_{Q\in\cd(Q_0)} |\langle b, h_Q^\ep\rangle| |\langle f, h_Q^\ep\rangle|\\
	& \quad\leq
		C_{n,\alpha} \|b\|_{BMO_{\cd}(w)} \left< |f|\right>_{Q_0} w(Q_0) + \sum_{R\in\ce} \sum_{Q\in\cd(R)} |\left\langle b, h_Q^\ep\right\rangle| |\langle f, h_Q^\ep\rangle|.\nonumber
	\end{align}
Form a sparse collection $\cs \subset \cd(Q_0)$ by adding $Q_0$ as the first cube, letting $R \in \ce$ be the $\cs$-children of $Q_0$,
and recursing on the second term in \eqref{E:Temp2}. Then we have
	\begin{align*}
	\sum_{Q\in\cd(Q_0)} |\langle b, h_Q^\ep\rangle| |\langle f, h_Q^\ep\rangle| & \leq
		C_{n,\alpha} \|b\|_{BMO_{\cd}(w)} \sum_{Q\in\cs} \left< |f|\right>_Q w(Q)\\
		&= C_{n,\alpha} \|b\|_{BMO_{\cd}(w)} \int_{Q_0} \bigg( \sum_{Q\in\cs} \left< |f|\right>_Q \unit_Q(x) \bigg) \,w(x)dx \\
		&= C_{n,\alpha} \|b\|_{BMO_{\cd}(w)} \int_{Q_0} \ca_{\cs}|f|\,w(x)dx,
	\end{align*}
where $\ca_{\cs}|f| := \sum_{Q\in\cs} \left< |f|\right>_Q\unit_Q$. Then, by \eqref{E:SparseAp}:
	\begin{align*}
	\int_{Q_0} \ca_{\cs} |f|\, w(x)dx &= \int_{Q_0} \ca_{\cs} |f| w^{1/r'} w^{1/r}\,dx \\
	&\leq \| \ca_{\cs} |f| \|_{L^{r'}(w)} w(Q_0)^{1/r} \\
	& \leq C_{n, \alpha, p, r} [w]_{A^{r'}}^{\max\{1, \frac{1}{r'-1}\}} \|f\|_{L^{r'}(w)} w(Q_0)^{1/r}\\
	& \leq C_{n, \alpha, p, r} [w]_{A^p}^{\max\{1, \frac{1}{p-1}\}} \|f\|_{L^{r'}(w)} w(Q_0)^{1/r},
	\end{align*}
and the result follows by choosing some value for $\alpha$, say $\alpha = 2$.
%
%
%
\end{proof}

\subsection{\texorpdfstring{The Weighted Hardy Space $H^1_{w,wavelet}(\mathbb{R}^n)$ via Daubechies Wavelets}{The Weighted Hardy Space via Daubechies Wavelets}}

We recall the Daubechies wavelets \cite{D}, and the
weighted Hardy space $H^1_{w,wavelet}(\mathbb{R}^n)$, weighted Carleson measure space $CM_w(\R^n)$ from Wu \cite{Wu}.

%


The compactly supported wavelets discussed in  \cite{D} is  as follows: for any $m\in \mathbb Z^+:=\{0\}\cup \mathbb N$, there is a collection of functions
$\{ \psi^\epsilon, \phi: \epsilon = 1,2,\ldots, 2^{n}-1 \}$ on $\R^n$ such that

\begin{itemize}
  \item [a)] $\psi^\varepsilon\in C^1$;
  \item [b)] $\psi^\varepsilon$ is compactly supported;
  \item [c)]  The collection $\{2^{jn}/2\psi^\varepsilon(2^jx-\gamma):\, j\in\mathbb Z, \, \gamma\in\mathbb Z^n\}$,
  and $\varepsilon\in\{1, 2,\,\cdots, 2^n-1\}$ form an orthonormal basis of $L^2(\mathbb R^n)$;
  \item [d)]$\int_{\mathbb R^n}\psi^\varepsilon(x)x^kdx=0$, for $k\in\{0,\,1,\,\cdots, m\}$;
  \item [e)] $\phi$ is continuous and compactly supported;
  \item [f)] For every $1\le\varepsilon< 2^n$, $\psi^\epsilon(x)$ is a finite linear
  combination of $\{\phi(x-\gamma), \gamma\in\mathbb Z^n\}$;
  \item [g)]$\int_{\mathbb R^n}\phi(x)dx\not=0$.
\end{itemize}
We denote by $\{\psi^\epsilon\}$ the wavelet system of order $m$.

\begin{definition}[\cite{Wu}]\label{def: Wu H1}
Suppose $m\in\mathbb Z^+$, $1<p<\infty $ and $w\in A^p(\mathbb R^n)$.
Suppose $\{ \psi^\epsilon \}$ is the wavelet system of order $m$.
The
weighted Hardy space $H^1_{w,wavelet}(\mathbb{R}^n)$ is defined as follows
$$ H^1_{w,wavelet}(\mathbb{R}^n):=\{ f\in L^1_w(\R^n): S_\psi(f) \in L^1_w(\R^n)  \}, $$
where
$$ S_\psi(f)(x) =\bigg( \sum_{Q\, {\rm dyadic}}\sum_\epsilon |\langle f, \psi_Q^\epsilon \rangle|^2 {  \unit_{2Q}(x)\over |Q| } \bigg)^{1\over2}. $$
\end{definition}

Next, we prove the key result in this subsection, which has not been addressed in \cite{Wu}
\begin{theorem}\label{t: H1waveletcharacterization}
Suppose $m\in\mathbb Z^+$, $1<p\leq2 $ and $w\in A^p(\mathbb R^n)$.
The definition of $H^1_{w,wavelet}(\mathbb{R}^n)$ is independent of the choice of the wavelet system $\{ \psi^\epsilon \}$ of order $m$. In particular, $H^1_{w,wavelet}(\mathbb{R}^n)$ can be characterized via a wavelet system $\{ \psi^\epsilon \}$ of order $0$.
\end{theorem}

To see this, we first recall the weighted Carleson measure space in \cite{Wu}.
\begin{definition} [\cite{Wu}]\label{def: Wu CMw}
Suppose  $\{ \psi^\epsilon \}$ is a wavelet system of order $m$ for any fixed $m\in\mathbb Z^+$. Then the weighted Carleson measure space $CM_w$ is defined as
$$ CM_w=\bigg\{ f\in L^1_{loc}(\mathbb R^n):  \|f\|_{CM_w}:=\sup_{P}  \bigg({1\over w(P)} \sum_{\substack{Q {\rm dyadic} \\ Q\subset P}}\sum_\epsilon  |\langle f, \psi_Q^\epsilon \rangle|^2 {  |Q|\over w(Q) } \bigg)^{1\over2} <\infty  \bigg \}.$$
\end{definition}

We now establish the duality of $H^1_{w,wavelet}(\mathbb{R}^n)$ with $CM_w(\R^n)$ with the corresponding wavelet system  $\{ \psi^\epsilon \}$ satisfying zero order cancellation only. We note that this duality result was first obtained in \cite{Wu} without clearly tracking the order of cancellation.
\begin{theorem}\label{t: Wu}
 Suppose $1<p<\infty $ and $w\in A^p(\mathbb R^n)$. Suppose $H^1_{w,wavelet}(\mathbb{R}^n)$ and $CM_w$ are both defined via the wavelet system $\{ \psi^\epsilon \}$ of order $0$. 
Then we have
$$     \big( H^1_{w,wavelet}(\mathbb{R}^n)\big)^* = CM_w.  $$
\end{theorem}
\begin{proof}
We prove this duality via the following weighted discrete sequence space $s^1_w$ and $c^1_w$ (the unweighted one was first introduced in \cite{HLL10}).

Consider the complex-valued sequence $\{s_Q\}_{Q {\rm\ dyadic}}$ indexed by the system of dyadic cubes $Q$ in $\R^n$. Define
$$s^1_w:=\big\{  \{s_Q\}:\ \|  \{s_Q\} \|_{s_w^1}<\infty \big\} \quad{\rm and}\quad  c^1_w:=\big\{  \{s_Q\}:\ \|  \{t_Q\} \|_{c_w^1}<\infty \big\}, $$
where
\begin{align*}
 \|  \{s_Q\} \|_{s_w^1}:= \bigg\|\bigg(\sum_{Q\, {\rm dyadic}}  |s_Q|^2 {  \unit_{2Q}(x)\over |Q| } \bigg)^{1\over2}\bigg\|_{L^1_w(\R^n)}\ {\rm and}
 \end{align*}
\begin{align*}
 \|  \{t_Q\} \|_{c_w^1}:=\sup_{P}\bigg({1\over w(P)} \sum_{\substack{Q {\rm dyadic} \\ Q\subset P}}  | t_Q|^2 {  |Q|\over w(Q) } \bigg)^{1\over2}.
\end{align*}
We now prove that for $w\in A^p(\mathbb R^n)$ with $1<p<\infty$,  the duality of $s_w^1$ is $c_w^1$ with respect to the inner product $\sum_{Q\, {\rm dyadic}} s_Q \cdot \overline {t}_Q$.

To see this, we first show that for each $\{t_Q\} \in c_w^1$, the linear functional
$$\ell( \{s_Q\}) := \sum_{Q\, {\rm dyadic}}   s_Q \cdot \overline{t_Q},  \quad \{s_Q\} \in s_w^1 $$
is bounded on $s_w^1$ with $\|\ell\|\leq C  \|  \{t_Q\} \|_{c_w^1}$.

In fact, for any $ \{s_Q\} \in s_w^1$, we define
\begin{align}
\Omega_k&:=\bigg\{ x\in \R^n: \bigg(\sum_{Q\, {\rm dyadic}}  |s_Q|^2 {  \unit_{2Q}(x)\over |Q| } \bigg)^{1\over2} >2^k\bigg\};\nonumber\\
\widetilde{\Omega}_k&:=\left\{x\in\R^n:\ M_w(\unit_{\Omega_k})(x) >{1\over2}\right\};\label{Omegakt}\\
B_k&:=\{ Q \,{\rm dyaidc}: w(Q\cap \Omega_k)>{w(Q)}/2,\ w(Q\cap \Omega_{k+1})\leq{w(Q)}/2 \}.\label{Bk}
\end{align}
Here $M_w$ is the classical weighted Hardy--Littlewood maximal function on $\mathbb R^n$. Then from H\"older's inequality we have
\begin{align*}
|\ell( \{s_Q\})|&\leq
\bigg|\sum_k \sum_{ \overline Q\in B_k,\ {\rm maximal}}  \sum_{\substack{  Q\in B_k\\ Q\subset \overline Q }}  s_Q \cdot \overline{t_Q}\bigg|\\
&\leq
\sum_k \sum_{ \overline Q\in B_k,\ {\rm maximal}}  \Big(\sum_{\substack{  Q\in B_k\\ Q\subset \overline Q }}  |s_Q|^2 {w(Q)\over |Q|} \Big)^{1\over2} \Big(\sum_{\substack{  Q\in B_k\\ Q\subset \overline Q }}   |\overline{t_Q}|^2 {|Q|\over w(Q)} \Big)^{1\over 2}\\
&\leq  \|  \{t_Q\} \|_{c_w^1} \sum_k \sum_{ \overline Q\in B_k,\ {\rm maximal}}  w(\overline Q  )^{1\over2} \Big(\sum_{\substack{  Q\in B_k\\ Q\subset \overline Q }}  |s_Q|^2 {w(Q)\over |Q|} \Big)^{1\over2}\\
&\leq  \|  \{t_Q\} \|_{c_w^1} \sum_k \bigg(\sum_{ \overline Q\in B_k,\ {\rm maximal}}  w(\overline Q  )\bigg)^{1\over2} \bigg(\sum_{ \overline Q\in B_k,\ {\rm maximal}} \sum_{\substack{  Q\in B_k\\ Q\subset \overline Q }}  |s_Q|^2 {w(Q)\over |Q|} \bigg)^{1\over2}\\
&\leq  \|  \{t_Q\} \|_{c_w^1} \sum_k   w(\widetilde {\Omega}_k  )^{1\over2} \bigg(\sum_{Q\in B_k}   |s_Q|^2 {w(Q)\over |Q|} \bigg)^{1\over2}.
\end{align*}
Next, we claim that
\begin{align}\label{eeeeee claim}
\bigg(\sum_{Q\in B_k}   |s_Q|^2 {w(Q)\over |Q|} \bigg)^{1\over2} \leq C 2^k w(\widetilde {\Omega}_k  )^{1\over2}.
\end{align}
In fact, by noting that
$$  \int_{\widetilde{\Omega}_k\backslash \Omega_{k+1}} \sum_{Q\, {\rm dyadic}}  |s_Q|^2 {  \unit_{2Q}(x)\over |Q| } w(x)dx \leq  2^{2k+2}  w(\widetilde {\Omega}_k  )  $$
and that
\begin{align*}
\int_{\widetilde{\Omega}_k\backslash \Omega_{k+1}} \sum_{Q\, {\rm dyadic}}  |s_Q|^2 {  \unit_{2Q}(x)\over |Q| } w(x)dx
&\geq \sum_{Q\in B_k}   |s_Q|^2 {w(Q\cap \big(\widetilde{\Omega}_k\backslash \Omega_{k+1})  \big)\over |Q|}\geq {1\over 2} \sum_{Q\in B_k}   |s_Q|^2 {w(Q)\over |Q|},
\end{align*}
we obtain that the claim  \eqref{eeeeee claim} holds. This yields that
\begin{align*}
|\ell( \{s_Q\})|
&\leq  C\|  \{t_Q\} \|_{c_w^1} \sum_k  2^k w(\widetilde {\Omega}_k  ) \leq C \|  \{t_Q\} \|_{c_w^1} \sum_k  2^k w( \Omega_k  ) \leq C \|  \{t_Q\} \|_{c_w^1}   \|  \{s_Q\} \|_{s_w^1},
\end{align*}
which implies that $\ell$ is a bounded linear  functional on $s_w^1$ with $\|\ell\|\leq C  \|  \{t_Q\} \|_{c_w^1}$.

Conversely, for any bounded linear functional $\ell$ on $s_w^1$, following the argument in \cite{HLL10}*{p. 673, proof of Theorem 4.2}, there exists a unique sequence $\{t_Q\}$ such that
$ \ell(\{s_Q\}) =\sum_{Q\, {\rm dyadic}} s_Q \cdot \overline {t}_Q  $ for $\{s_Q\}\in s_w^1$.
Now it suffices to prove that this $\{t_Q\}$ is in $c^1_w$ with $ \|\{t_Q\}\|_{c^1_w} \leq C\|\ell\|$.

To see this, for any fixed dyadic cube $P$, we consider $\mathcal D_P =\{  Q \}_{Q\, {\rm dyadic}, Q\subset P}$ and we
define the measure on  $\mathcal D_P$ by $dm(Q)= {|Q|\over w(P)}$, $Q\in \mathcal D_P$. Then we have
\begin{align*}
\bigg({1\over w(P)} \sum_{\substack{Q {\rm dyadic} \\ Q\subset P}}  | t_Q|^2 {  |Q|\over w(Q) }\bigg)^{1\over 2}
&= \bigg\| \Big\{  {t_Q \over w(Q)^{1\over2}}  \Big\} \bigg\|_{\ell^2( \mathcal{D}, dm )}\\
&=\sup_{ \{s_Q\}: \|\{s_Q\}\|_{\ell^2( \mathcal{D}, dm )} \leq 1} \bigg| \sum_{ \substack{Q\,{\rm dyadic} \\ Q\subset P}} s_Q  { |Q|\over w(P) w(Q)^{1\over2} } \cdot \overline{t_Q} \bigg|\\
&\leq  \sup_{ \{s_Q\}: \|\{s_Q\}\|_{\ell^2( \mathcal{D}, dm )} \leq 1} \|\ell\| \bigg\| \bigg\{ { s_Q|Q|\over w(P) w(Q)^{1\over2} } \bigg\}_{\substack{Q\,{\rm dyadic} \\ Q\subset P}} \bigg\|_{s^1_w},
\end{align*}
where the inequality above follows from the fact that $\ell$ is a bounded linear functional on $s_w^1$ and the fact that when  $\{s_Q\}$ is in $\ell^2( \mathcal{D}, dm )$, 
$\bigg\{ { s_Q|Q|\over w(P) w(Q)^{1\over2} } \bigg\}_{\substack{Q\,{\rm dyadic} \\ Q\subset P}}$ is in $s_w^1$ with the norm
\begin{align*}
 \bigg\| \bigg\{ { s_Q|Q|\over w(P) w(Q)^{1\over2} } \bigg\}_{\substack{Q\,{\rm dyadic} \\ Q\subset P}} \bigg\|_{s^1_w}
 & = {1\over w(P)} \int_P\Big( \sum_{\substack{Q\,{\rm dyadic} \\ Q\subset P}}  { |s_Q|^2 |Q|\over w(Q) } \unit_{2Q}(x) \Big)^{1\over 2} w(x)dx\\
 & \leq \bigg({1\over w(P)} \int_P \sum_{\substack{Q\,{\rm dyadic} \\ Q\subset P}}  { |s_Q|^2 |Q|\over w(Q) } \unit_{2Q}(x)  w(x)dx\bigg)^{1\over2}\\
&  \leq C \|\{s_Q\}\|_{\ell^2( \mathcal{D}, dm )} <\infty.
 \end{align*}
As a consequence, we get that
\begin{align*}
\bigg({1\over w(P)} \sum_{\substack{Q {\rm dyadic} \\ Q\subset P}}  | t_Q|^2 {  |Q|\over w(Q) }\bigg)^{1\over 2}
&\leq C\|\ell\| \sup_{ \{s_Q\}: \|\{s_Q\}\|_{\ell^2( \mathcal{D}, dm )} \leq 1} \|\{s_Q\}\|_{\ell^2( \mathcal{D}, dm )}\leq C\|\ell\|.
\end{align*}
By taking the supremum over all dyadic cubes $P$, we obtain that
$   \|\{t_Q\}\|_{c^1_w} \leq C\|\ell\|.$

Combining these two parts, we obtain that  the duality of $s_w^1$ is $c_w^1$.

We now define the lifting operator $T_L$ and projection operator $T_P$ as follows:

\quad for any locally integrable function $f$, we define $T_L(f) = \{ \sum_\epsilon \langle f, \psi_Q^\epsilon \rangle\}_{Q\, {\rm dyadic}}$;

\quad for any complex sequence $\{s_Q\}_{Q\, {\rm dyadic}}$, we define $T_P(  \{s_Q\}) = \sum_Q \sum_\epsilon s_Q \psi_Q^\epsilon$.

We now show that $T_L$ maps $H^1_{w,wavelet}(\R^n)$ to $s^1_w$ and $CM_w(\R^n)$ to $c^1_w$.
In fact, for any  $f\in H^1_{w,wavelet}(\R^n)$, by definition, we obtain that
$T_L(f) $ is in $s^1_w$ with $\|T_L(f)\|_{s^1_w} \leq C\|f\|_{H^1_{w,wavelet}(\R^n)}$ and similarly
for any  $b\in CM_w(\R^n)$, by definition, we obtain that
$T_L(b) $ is in $c^1_w$ with $\|T_L(b)\|_{c^1_w} \leq C\|f\|_{CM_w(\R^n)}$.

Next we show that $T_P$ maps $s^1_w$ to $H^1_{w,wavelet}(\R^n)$  and $c^1_w$ to $CM_w(\R^n)$.
To see this, for any $\{s_Q\} \in s^1_w$,  we have
\begin{align*}
\| T_P(\{s_Q\}) \|_{H^1_{w,wavelet}(\R^n)} &=\bigg\|\bigg( \sum_{Q'\, {\rm dyadic}}\sum_{\epsilon'} \Big|\Big\langle \sum_Q \sum_\epsilon s_Q \psi_Q^\epsilon, \psi_{Q'}^{\epsilon'} \Big\rangle\Big|^2 {  \unit_{2Q'}(x)\over |Q'| } \bigg)^{1\over2}\bigg\|_{L^1_w(\R^n)}\\
&\leq C\|\{s_Q\} \|_{s^1_w},
\end{align*}
where the inequality follows from the orthonormality property of the wavelet basis $\{\psi_Q^\epsilon\}$.
Similarly, we obtain that for any $\{t_Q\} \in c^1_w$,
$$\| T_P(\{t_Q\}) \|_{CM_{w,}(\R^n)} \leq C\|\{t_Q\} \|_{c^1_w}.$$

Thus, from the duality of the weighted sequence spaces $s^1_w$ and $c^1_w$,  the boundedness of the
lifting operator $T_L$ and projection operator $T_P$, and the fact that
$T_L \circ T_P $ is the identity operator, we obtain that the duality of $H^1_{w,wavelet}(\mathbb{R}^n)$ is $CM_w$.
\end{proof}

We now prove the following proposition.
\begin{prop}\label{prop: CM}
Suppose $1<p\leq2 $ and $w\in A^p(\mathbb R^n)$.
For any wavelet system $\{ \psi^\epsilon \}$ of order $m$ with $m\in\mathbb Z^+$, we have that the Carleson measure space $CM_w $ defined via $\{ \psi^\epsilon \}$ coincides with ${\rm BMO}_{w,2}(\R^n)$, i.e.,
$$ CM_w = {\rm BMO}_{w,2}(\R^n) $$
and they have equivalent norms. This implies $ \big( H^1_{w,wavelet}(\mathbb{R}^n)\big)^* = {\rm BMO}_{w,2}(\R^n)$.
\end{prop}

\begin{proof}
To see this, it suffices to show that for $f\in L^1_{loc}(\R^n)$,
 $$\|f\|_{{\rm BMO}_{w,2}(\R^n)}\approx \|f\|_{CM_w}.$$
%
%
Fix a cube $Q$ and expanding $f$ via $\psi^\epsilon$. We see that
$$(f-\left< f\right>_Q)\unit_Q=\sum_{\substack{P {\rm\ dyaidc}\\ P\subset Q}}  \langle f, \psi_P^\epsilon\rangle \psi^\epsilon_P=:F_Q.$$
Next, from the property of an $A^2(\mathbb{R}^n)$ weight and the upper and lower bounds of the wavelet square function $S_\psi$, we have that
\begin{align}
\int_Q|f(x)-\left< f\right>_Q|^2w^{-1}(x)dx&=\|F_Q\|_{L^2(w^{-1})}^2\nonumber\\
&\approx  \|S_\psi(F_Q)\|^2_{L^2(w^{-1})}\label{eeee1}\\
&=\sum_{\substack{P {\rm\ dyaidc}\\ P\subset Q}}|\langle f, \psi_P^\epsilon\rangle|^2\langle w^{-1}\rangle_P\nonumber\\
&\approx C\sum_{\substack{P {\rm\ dyaidc}\\ P\subset Q}}|\langle f, \psi_P^\epsilon\rangle|^2\frac1{\langle w\rangle_P},\label{eeee2}
\end{align}
which implies that
$$ \|f\|_{CM_w}\approx \|f\|_{{\rm BMO}_{w,2}(\R^n)}, $$
and hence we have ${\rm BMO}_{w,2}(\R^n)= CM_w$ with equivalent norms.

Here we note that \eqref{eeee2} follows from the fact that  $w\in A^p(\mathbb R^n)$ for some $1<p\leq2$, then
$w\in A^2(\mathbb{R}^n)$. To be more precise, for any dyadic cube $P$, we have
\begin{align*}
1& ={1\over |P|}\int_P w(x)^{1\over 2}w(x)^{-{1\over 2}} dx \leq \bigg({1\over |P|}\int_P w(x)dx\bigg)^{1\over 2} \bigg({1\over |P|}\int_P w(x)^{-1} dx\bigg)^{1\over 2}\\
&\leq [w]_{A^2}^{1\over 2}\leq C[w]_{A^p}^{1\over 2}.
\end{align*}
By taking the square on both sides, we have $ 1\leq \langle w\rangle_P \langle w^{-1}\rangle_P \leq C[w]_{A^p}$, which implies that
 \eqref{eeee2}  holds.

The equivalence in \eqref{eeee1} is a standard result (see for example, \cite{GM}*{Theorem 4.16}), which requires a wavelet system $\{ \psi^\epsilon \}$ of order $0$ only, i.e.,
the wavelet function satisfies only the zero order cancellation:
 $\int_{\mathbb R^n}\psi^\varepsilon(x)dx=0.$
This finishes the proof of Proposition \ref{prop: CM}.
\end{proof}

\begin{proof}[\bf Proof of Theorem  \ref{t: H1waveletcharacterization}]

Suppose $1<p\leq2 $ and $w\in A^p(\mathbb R^n)$.
For any wavelet system $\{ \psi^\epsilon \}$ of order $m$ with $m\in\mathbb Z^+$, we have the Carleson measure space $CM_w $ defined via $\{ \psi^\epsilon \}$.

Now from Proposition \ref{prop: CM}, we see that $ CM_w = {\rm BMO}_{w,2}(\R^n) $
and they have equivalent norms.
Since the definition of ${\rm BMO}_{w,2}(\R^n)$ is independent of the wavelet system $\{ \psi^\epsilon \}$,
we see that weighted Carleson measure space $CM_w$ is  independent of the wavelet system $\{ \psi^\epsilon \}$, which, together with Theorem \ref{t: Wu},  shows that
Theorem \ref{t: H1waveletcharacterization} holds.
\end{proof}

\subsection{\texorpdfstring{The Weighted Hardy Space $H^1_{w}(\mathbb{R}^n)$ via the Littlewood--Paley area function}{The Weighted Hardy Space via the Littlewood--Paley area function}}

Suppose $\beta$ is a non-negative integer.
We denote by $\mathcal R^\beta$ the class of functions $\varphi$ such that
 $\varphi\in\mathcal S(\R^n)$ with
$$ \int_{\R^n} \varphi(x)  x^\alpha dx=0 $$
for the multi-index $\alpha$ satisfying $|\alpha|\leq \beta$,
  and that
$$ \int_0^\infty |\hat{\varphi}( t \xi)|^2 {dt\over t}=C_\varphi\not=0,\quad \xi\not=0. $$
And we have the standard dilation
$\varphi_t(x)  = t^{-n} \varphi( {x\over t} )$.

We recall the Littlewood--Paley area function as
\begin{align}\label{Sf-phi}
S_{\varphi,\beta}(f)(x) =\bigg( \iint_{\Gamma(x)} |\varphi_t *f(x)|^2 {dydt\over t^{n+1}} \bigg)^{1\over2},
\end{align}
where $\varphi$ is in $\mathcal R^{\beta}$.

\begin{definition}\label{def: MW Hardy}
Suppose $1<p<\infty$, $w\in A^p(\mathbb R^n)$ and
$\varphi \in \mathcal R^\beta$ with $\beta\geq \lfloor np\rfloor-n$,
where $\lfloor \alpha\rfloor$  for a given $\alpha\in\mathbb R$
is the biggest integer $k$ such that $k\le \alpha$.
Let the square function $S_{\varphi,\beta}(f)$ be defined via $\varphi$ as in \eqref{Sf-phi}.
We define the weighted Hardy space $H^1_{w,\varphi,\beta}(\R^n)$
as
$ H^1_{w,\varphi,\beta}(\R^n):=\{ f\in L^1_w(\mathbb R^n):\ S_{\varphi,\beta}(f)\in L^1_w(\R^n)\} $
with the norm given by $\|f\|_{H^1_{w,\varphi,\beta}(\R^n)}:=\|S_{\varphi,\beta}(f)\|_{L^1_w(\R^n)}$.
\end{definition}

We also recall the atoms for the weighted Hardy spaces.
\begin{definition}\label{def: weighted atom}
Suppose $1<p<\infty$, $w\in A^p(\mathbb R^n)$, and $\beta\geq \lfloor np\rfloor-n$.
A function $a$ is called a $(1,p,\beta)$-atom, if
there exists a cube $Q\subset\mathbb R^n$ such that
\begin{align*}
&(1)\ \ {\rm {supp}}(a)\subset Q;\\
&(2)\ \ \int_Q a(x)\, x^\alpha dx=0, \quad {\rm for\ multi-index\ } \alpha {\rm\ with\ } |\alpha|\leq \beta;\\
&(3)\ \ \|a\|_{L^p_{w(Q)}}\le \left[w(Q)\right]^{{1\over p}-1}.
\end{align*}
\end{definition}

\begin{definition}\label{def: weighted atom Hardy}
Suppose $1<p<\infty$, $w\in A^p(\mathbb R^n)$, and $\beta\geq \lfloor np\rfloor-n$.
A function $f$ is said to belong to the space $H^{1,\,p,\,\beta}_{w}(\mathbb{R}^n)$, if
$f=\sum_{j=1}^\infty \lambda_j a_j$
with $\sum_{j=1}^\infty|\lambda_j|<\infty$ and $a_j$ is a $(1,p,\beta)$-atom
for each $j$. Moreover, the norm of $f$ on $H^{1,\,p,\,\beta}_{w}(\mathbb{R}^n)$ is defined by
$$\|f\|_{H^{1,\,p,\,\beta}_{w}(\mathbb{R}^n)}=\inf\left\{\sum_{j=1}^\infty|\lambda_j|\right\},$$
where the infimum is taken over all possible decompositions of $f$ as above.

\end{definition}

Then we have the following theorem.
\begin{theorem}\label{t: Hardytatom}
Suppose $1<p<\infty$, $w\in A^p(\mathbb R^n)$, and $\beta\geq \lfloor np\rfloor-n$.
Then we have $  H^1_{w,\varphi,\beta}(\R^n) =H^{1,\,p,\,\beta}_{w}(\mathbb{R}^n) $
with equivalent norms.
\end{theorem}

\begin{proof}
This type of atomic decomposition follows a standard approach from Chang--Fefferman.
We sketch the proof as follows.  Suppose $f\in H^1_{w,\varphi,\beta}(\R^n) $.
For each $k\in\mathbb Z$, we now define
$\Omega_k:=\left\{x\in\R^n:\ S_{\varphi,\beta}(f)(x) >2^k\right\}$, and then we define $\widetilde{\Omega}_k$ and $B_k$ according to $\Omega_k$, using the same way as in \eqref{Omegakt} and \eqref{Bk}, respectively.

Then from Calder\'on's reproducing formula, we have
that
\begin{align*}
f(x) &= C_\varphi\int_0^\infty \varphi_t * \varphi_t * f(x) {dt\over t}\\
&=C_\varphi \int_0^\infty\int_{\R^n} \varphi_t (x-y) \varphi_t * f(y) {dydt\over t}\\
&=C_\varphi\sum_{Q\ {\rm dyadic}} \iint_{\widehat Q} \varphi_t (x-y) \varphi_t * f(y) {dydt\over t}\\
&=C_\varphi\sum_k \sum_{ \overline Q\in B_k,\ {\rm maximal}}  \sum_{\substack{  Q\in B_k\\ Q\subset \overline Q }}   \iint_{\widehat Q} \varphi_t (x-y) \varphi_t * f(y) {dydt\over t}\\
&=\sum_k \sum_{\overline Q\in B_k,\ {\rm maximal}}  \lambda_{k,\overline Q} a_{k,\overline Q},
\end{align*}
where $ \lambda_{k,\overline Q}:= 2^k w(\overline Q) $ and
$$a_{k,\overline Q} := {C_\varphi\over \lambda_{k,\overline Q}}  \sum_{\substack{Q\in B_k\\ Q\subset \overline Q }}   \iint_{\widehat Q} \varphi_t (x-y) \varphi_t * f(y) {dydt\over t}.$$

First, it is direct that $a_{k,\overline Q}$ is supported in $3\overline Q$ and satisfies the cancellation condition up to order $\beta$ as $\varphi$ does.
Next, by testing $a_{k,\overline Q}$ against an arbitrary function $h\in L^{p'}_{w'}(\mathbb{R}^n)$,  using the definition of
$\lambda_{k,\overline Q}$ and by using the boundedness of the $\mathcal G^\ast_{\varphi,\beta}$ function on $L^{p'}_{w'}(\mathbb{R}^n)$, which had been studied by many authors (see for example \cite{AS}, \cite{MW}),
 we obtain that
\begin{align}\label{atom norm}
\|a_{k,\overline Q}\|_{L^p_{w}(\overline{ Q})}\ls \left[w(3\overline Q)\right]^{{1\over p}-1}.
\end{align}
Thus, we see that each $a_{k,\overline Q}$ is a $(1,p,\beta)$-atom.

To be more precise for the estimate \eqref{atom norm}, we note that
\begin{align*}
\|a_{k,\overline Q}\|_{L^p_{w}(\mathbb R^n)} 
&= \sup_{h: \|h\|_{L^{p'}_{w'}(\mathbb R^n)}\leq1} |\langle a_{k,\overline Q},h \rangle|\\
&= \sup_{h: \|h\|_{L^{p'}_{w'}(\mathbb R^n)}\leq1} {1\over \lambda_{k,\overline Q}}\bigg|\sum_{\substack{ Q\in B_k\\ Q\subset \overline Q }}   \iint_{\widehat Q} \widetilde\varphi_t *h(y) \varphi_t * f(y) {dydt\over t} \bigg|\\
&\ls \sup_{h: \|h\|_{L^{p'}_{w'}(\mathbb R^n)}\leq1} {1\over \lambda_{k,\overline Q}}\int_{\overline Q\cap \Omega_{k+1}^c}\sum_{\substack{  Q\in B_k\\ Q\subset \overline Q }}   \iint_{\widehat Q} |\widetilde\varphi_t *h(y)|\,|\varphi_t * f(y) |\unit_{\{|x-y|<2t\}}(x){dydt\over t^{n+1}}dx \\
&\leq \sup_{h: \|h\|_{L^{p'}_{w'}(\mathbb R^n)}\leq1} {1\over \lambda_{k,\overline Q}}\int_{\overline Q\cap \Omega_{k+1}^c} \bigg( \sum_{\substack{ Q\in B_k\\ Q\subset \overline Q }}   \iint_{\widehat Q \cap \{|x-y|<2t\}} |\widetilde\varphi_t *h(y)|^2 {dydt\over t^{n+1}}\bigg)^{1\over2}\\
&\hskip5cm\times \bigg( \sum_{\substack{ Q\in B_k\\ Q\subset \overline Q }}   \iint_{\widehat Q \cap \{|x-y|<2t\}}|\varphi_t * f(y) |^2{dydt\over t^{n+1}}\bigg)^{1\over2}dx \\
&\ls \sup_{h: \|h\|_{L^{p'}_{w'}(\mathbb R^n)}\leq1} {1\over \lambda_{k,\overline Q}}\int_{\overline Q\cap \Omega_{k+1}^c} \mathcal G_{\varphi,\beta}^*(h)(x) S_{\varphi,\beta}(f)(x) dx,
\end{align*}
where $\widetilde\varphi(x)=\varphi(-x)$,
$$ \mathcal G_{\varphi,\beta}^*(h)(x) = \bigg( \int_0^\infty\int_{\R^n} \Big( {t\over t+|x-y|}\big)^{3n}|\widetilde\varphi_t *h(y)|^2{dydt\over t^{n+1}}  \bigg)^{1\over2} $$
and
in the first inequality above we used the fact that for $Q\in B_k$, $w(Q\cap \Omega_{k+1})\leq w(Q)/2$ which gives $|Q\cap \Omega_{k+1}|\leq C_0|Q|$ with $C_0<1$ (see \cite{Gra}*{Theorem 9.3.3(e)}) and hence further implies  $|Q\cap \Omega_{k+1}^c|> (1-C_0)|Q|$. As a consequence, from H\"older's inequality, we have
\begin{align*}
&\|a_{k,\overline Q}\|_{L^p_{w}(\overline{ Q})} \\
&\quad\ls \sup_{h: \|h\|_{L^{p'}_{w'}(\mathbb R^n)}\leq1} {1\over \lambda_{k,\overline Q}} \bigg(\int_{\overline Q\cap \Omega_{k+1}^c} \mathcal G_{\varphi,\beta}^*(h)(x)^{p'} w'(x)dx\bigg)^{1\over p'}   \bigg(\int_{\overline Q\cap \Omega_{k+1}^c} S_{\varphi,\beta}(f)(x)^p w(x) dx\bigg)^{1\over p}\\
&\quad\ls \sup_{h: \|h\|_{L^{p'}_{w'}(\mathbb R^n)}\leq1}\|h\|_{L^{p'}_{w'}} {1\over \lambda_{k,\overline Q}} 2^{k+1} w(\overline Q)^{1\over p}\\
&\quad\ls \left[w(\overline Q)\right]^{{1\over p}-1},
\end{align*}
which shows that \eqref{atom norm} holds. Moreover, it is direct to see that
$$\sum_k \sum_{\overline Q\in B_k,\ {\rm maximal}}  |\lambda_{k,\overline Q}| \ls \|S_{\varphi,\beta}(f)\|_{L^1_w(\R^n)}.$$
This shows that  $  H^1_{w,\varphi,\beta}(\R^n) \subset H^{1,\,p,\,\beta}_{w}(\mathbb{R}^n) $.

To prove the reverse inclusion, i.e.,  $  H^1_{w,\varphi,\beta}(\R^n) \supset H^{1,\,p,\,\beta}_{w}(\mathbb{R}^n) $, it suffices to prove that there exists a positive constant $C$ such that for every $(1,p,\beta)$-atom $a(x)$, we have
$$\|S_{\varphi,\beta}(a)\|_{L^1_w(\R^n)} \leq C.$$
Actually, this follows from a standard technique by decomposing the whole space $\R^n$ into annuli according to the size of the cube $2Q$, where ${\rm supp}\, a\subset Q$, and then using the cancellation condition of $a(x)$ and the smoothness condition of $\varphi$ to get a suitable decay, which guarantees the summability over these annuli. This step  requires the condition that $\beta\geq \lfloor np\rfloor-n$. We omit the details.
\end{proof}

Following similar techniques in the proof above (see also the result in \cite{GM} for dimension 1), we obtain the following.
\begin{theorem}\label{t: Hardytatom wavelet}
Suppose $1<p<\infty$, $w\in A^p(\mathbb R^n)$, and $\beta\geq \lfloor np\rfloor-n$. Also suppose that $\{ \psi^\epsilon \}$ is a wavelet system defined as in Section 4.2 with the cancellation up to order $\lfloor np\rfloor-n$.
Then we have $  H^1_{w,wavelet}(\R^n) =H^{1,\,p,\,\beta}_{w}(\mathbb{R}^n) $
with equivalent norms.
\end{theorem}

As a direct consequence of Proposition \ref{prop: CM}, and Theorems \ref{t: Hardytatom} and \ref{t: Hardytatom wavelet} above, we obtain
the following duality argument.
\begin{theorem}\label{t:dual BMO cla}
Suppose $1<p\leq2$ and $w\in A^p(\mathbb R^n)$.

{\rm(1)} for any $\beta \geq \lfloor np\rfloor-n$, $\big(H^{1,\,p,\,\beta}_{w}(\mathbb{R}^n)\big)^*= {\rm BMO}_w(\R^n)$,

{\rm(2)} for any $\beta \geq \lfloor np\rfloor-n$, $\big(H^{1}_{w,\varphi,\beta}(\mathbb{R}^n)\big)^*= {\rm BMO}_w(\R^n)$.

\end{theorem}

\begin{remark}\label{r:coinc Hardy spc}
We denote $H^{1}_{w,\varphi,\beta}(\mathbb{R}^n)$ simply by $H^1_w(\mathbb R^n)$.
Then Theorem \ref{t:dual BMO cla} implies that for $w\in A^p(\mathbb R^n)$, $1<p\leq2$,  the weighted Hardy space $H^1_w(\mathbb R^n)$ is independent of the choice of the function $\varphi$ and of the order of cancellation $\beta \geq \lfloor np\rfloor-n$.
\end{remark}

\section{\texorpdfstring{Characterizations of $H^1_{w}(\mathbb{R}^n)$ and ${\rm BMO}_{w}(\mathbb{R}^n)$ via semigroups generated by the Laplacian}{Characterizations of the weighted Hardy space, weighted BMO via semigroups generated by the Laplacian.}}
\setcounter{equation}{0}
\label{s:BloomBMOHardyNuemann}

In this section,  we prove that for the weight $w\in A^p(\mathbb R^n)$ with $1<p\leq2$,
the classical  Hardy space $H^1_{w}(\mathbb{R}^n)$
and weighted BMO space ${\rm BMO}_{w}(\mathbb{R}^n)$ are equivalent to the new Hardy space
$H^1_{\Delta,w}(\mathbb{R}^n)$ and  weighted BMO space ${\rm BMO}_{\Delta, w}(\mathbb{R}^n)$,
respectively. To begin with,
we first consider the weighted Hardy space $H^1_{\Delta, w}(\R^n)$ defined via the  Littlewood--Paley area function associated with $\Delta$ as follows.
\begin{align}\label{SfDelta}
S_\Delta(f)(x) =\bigg( \int_0^\infty \int_{y:\, |x-y|<t}  \Big| t^2\Delta e^{-t^2\Delta}f(y) \Big|^2 {dydt\over t^{n+1}}\bigg)^{1\over2}.
\end{align}
\begin{definition}\label{def: Laplace Hardy}
Suppose $1<p<\infty$ and $w\in A^p(\mathbb R^n)$.
We define the weighted Hardy space associated with the heat semigroup generated by the Laplacian
as
$ H^1_{\Delta, w}(\R^n):=\{ f\in L^1_w(\mathbb R^n):\ S_\Delta(f)\in L^1_w(\R^n)\} $
with the norm given by $\|f\|_{H^1_{\Delta,w}(\R^n)}:=\|S_{\Delta}(f)\|_{L^1_w(\R^n)}$.
\end{definition}

\begin{theorem}\label{t-coin Hardy space cla and Lapl}
Suppose $1<p\leq2$ and $w\in A^p(\mathbb R^n)$.
Then we have $$ H^1_{w}(\mathbb{R}^n)=H^1_{\Delta, w}(\mathbb{R}^n) $$ with equivalent norms.
\end{theorem}
\begin{proof}
We first show that
$H^{1,\,p,\,\beta}_{w}(\mathbb{R}^n) \subset H^1_{\Delta,w}(\R^n)$ for $\beta \geq \lfloor np\rfloor-n.$ To see this,
it suffices to prove that there exists a constant $C$ such that for every $(1,p,\beta)$-atom $a(x)$, we have
$$\|S_{\Delta}(a)\|_{L^1_w(\R^n)} \leq C.$$
This again follows from the standard technique of decomposing the whole space $\R^n$ into annuli according to the size of the ball $B$  which is the support of the atom $a(x)$, and then using the cancellation condition of $a(x)$ and the smoothness condition of $\varphi$ to get a suitable decay, which guarantees the summability over these annuli. This requires the condition that $\beta\geq \lfloor np\rfloor-n$.   Again, we omit the straight forward details.

Next we prove that $H^1_{\Delta,w}(\R^n)\subset H^{1,\,p,\,0}_{w}(\mathbb{R}^n)$.  To see this, we recall the following construction of $\psi$ in \cite{DY}.  Let $\varphi:=-\pi i\unit_{\frac12<|x|<1}$ and
$\psi$ the Fourier transform of $\varphi$. That is,
$$\psi(s):=s^{-1}(2\sin (s/2)-\sin s).$$
Consider the operator
\begin{align}\label{eee psi}
\psi\big(t\sqrt{\Delta}\big):=
\big(t\sqrt{\Delta}\big)^{-1}\big[2\sin\big(t\sqrt{\Delta}/2\big)-\sin\big(t\sqrt{\Delta}\big)\big].
\end{align}
And we recall the following properties for $\psi\big(t\sqrt{\Delta}\big)$.
\begin{prop}[\cite{DY}]\label{p-kernel cancel-compact supp}
For all $t\in(0, \infty)$, $\psi\big(t\sqrt{\Delta}\big)(1)=0$
and the kernel $K_{\psi(t\sqrt{\Delta})}$ of $\psi(t\sqrt{\Delta})$ has support contained in
$\{(x, y)\in\mathbb R_+\times \mathbb R_+:\, |x-y|\le t\}.$
\end{prop}
Now for any  $f\in H^1_{\Delta,w}(\R^n)$, we define
$\Omega_k:=\left\{ x\in \R^n: S_\Delta(f)(x)>2^k\right\}$, and then we define $\widetilde{\Omega}_k$ and $B_k$ according to $\Omega_k$, using the same way as in \eqref{Omegakt} and \eqref{Bk}, respectively.

Then we consider the following reproducing formula
\begin{eqnarray*}
f(x)=C_\psi\dint_0^\infty\psi(t\sqrt{\Delta})
 t^2\Delta e^{-t^2\Delta}  f(x)\,\frac{dt}{t}
\end{eqnarray*}
in the sense of $L^2(\R^n)$, where $\psi(t\sqrt{\Delta})$ is  defined as in \eqref{eee psi} and $C_\psi$ is a constant depending on $\psi$.
We now follow the same way of decomposition as in the proof of Theorem \ref{t: Hardytatom} to further have
\begin{align*}
f(x)&=C_\psi\dint_0^\infty\int_{\R^n}K_{\psi(t\sqrt{\Delta})}(x,y)
 t^2\Delta e^{-t^2\Delta}  f(y)\,\frac{dydt}{t}\\
 &= C_\psi\sum_k \sum_{\overline Q\in B_k,\ {\rm maximal}}  \lambda_{k,\overline Q} a_{k,\overline Q},
\end{align*}
where $ \lambda_{k,\overline Q}:=2^k w(\overline Q)$ and
$$a_{k,\overline Q} := {C_\psi\over \lambda_{k,\overline Q}}  \sum_{\substack{Q\in B_k\\ Q\subset \overline Q }}   \iint_{\widehat Q} K_{\psi(t\sqrt{\Delta})}(x,y)
 t^2\Delta e^{-t^2\Delta}  f(y) {dydt\over t}.$$

Then it is direct to see that
$$\sum_k \sum_{\overline Q\in B_k,\ {\rm maximal}}  \lambda_{k,\overline Q} \ls \|S_{\Delta}(f)\|_{L^1_w(\R^n)}.$$
Moreover, from the property of $\psi(t\sqrt{\Delta})$ in Proposition \ref{p-kernel cancel-compact supp}, we see that
each $a_{k,\overline Q} $ is supported in $3\overline Q$ and
$ \int_{\R^n} a_{k,\overline Q} (x)dx=0. $
Thus, it suffices to verify the $L^p_w(\R^n)$ norm of $a_{k,\overline Q}$.
Following the same approach and estimates as in the proof of \eqref{atom norm}, we obtain that
\begin{align*}
\|a_{k,\overline Q}\|_{L^p_{w}(\mathbb R^n)} 
&\leq \sup_{h: \|h\|_{L^{p'}_{w'}(\mathbb R^n)}\leq1} {1\over \lambda_{k,\overline Q}}\int_{\overline Q\cap \Omega_{k+1}^c} \bigg( \sum_{\substack{ Q\in B_k\\ Q\subset \overline Q }}   \iint_{\widehat Q \cap \{|x-y|<2t\}} |\psi(t\sqrt{\Delta})(h)(y)|^2 {dydt\over t^{n+1}}\bigg)^{1\over2}\\
&\quad\hskip4cm\times \bigg( \sum_{\substack{ Q\in B_k\\ Q\subset \overline Q }}   \iint_{\widehat Q \cap \{|x-y|<2t\}}|t^2\Delta e^{-t^2\Delta}  f(y) |^2{dydt\over t^{n+1}}\bigg)^{1\over2}dx \\
&\ls \sup_{h: \|h\|_{L^{p'}_{w'}(\mathbb R^n)}\leq1} {1\over \lambda_{k,\overline Q}}\int_{\overline Q\cap \Omega_{k+1}^c} \mathcal G_{\Delta}^*(h)(x) S_{\Delta}(f)(x) dx,
\end{align*}
where
$$ \mathcal G_{\Delta}^*(h)(x) := \bigg( \int_0^\infty\int_{\R^n} \Big( {t\over t+|x-y|}\big)^{4n}|\psi(t\sqrt{\Delta})(h)(y)|^2{dydt\over t^{n+1}}  \bigg)^{1\over2}.$$
As proved in \cite{GY}*{Lemma 5.1}, there exists a positive constant $C$ such that for all $v \in A^q(\mathbb R^n)$ with $1 < q <\infty$, the following estimate holds:
$$  \|\mathcal G_{\Delta}^*(h)\|_{L^q_v(\R^n)}\leq C\|h\|_{L^q_v(\R^n)}.  $$

As a consequence, from H\"older's inequality, we have
\begin{align*}
&\|a_{k,\overline Q}\|_{L^p_{w}(\overline{ Q})} \\
&\quad\ls \sup_{h: \|h\|_{L^{p'}_{w'}(\mathbb R^n)}\leq1} {1\over \lambda_{k,\overline Q}} \bigg(\int_{\overline Q\cap \Omega_{k+1}^c} \mathcal G_{\Delta}^*(h)(x)^{p'} w'(x)dx\bigg)^{1\over p'}   \bigg(\int_{\overline Q\cap \Omega_{k+1}^c} S_{\Delta}(f)(x)^p w(x) dx\bigg)^{1\over p}\\
&\quad\ls \left[w(\overline Q)\right]^{{1\over p}-1}.
\end{align*}
This finishes the proof of Theorem \ref{t-coin Hardy space cla and Lapl}.
\end{proof}

 Now we introduce the weighted BMO space associated with the  heat semigroup generated by the Laplacian on $\R^n$ as follows.

\begin{definition}\label{def: BMOwLapalce}
Suppose $1<p<\infty$ and $w\in A^p(\mathbb R^n)$.
$${\rm BMO}_{\Delta, w}(\mathbb{R}^n):=\{f\in L^1_{\rm loc}(\mathbb R^n): \|f\|_{{\rm BMO}_{\Delta, w}(\mathbb{R}^n)}<\infty\},$$
where
$$\|f\|_{{\rm BMO}_{\Delta, w}(\mathbb{R}^n)}:=\sup_P \bigg({1\over w(P)} \sum_{\substack{Q\ {\rm dyadic} \\ Q\subset P }} \iint_{ \widehat{Q}}\Big| t^2\Delta e^{-t^2\Delta} f(y)\Big|^2 {t^n\over w(Q)} {dydt\over t}\bigg)^{1\over2}.$$
\end{definition}

%

%

Then we have our main result in this section.
\begin{theorem}\label{t: coinc clas BMO and Lapl}
Suppose $1<p\leq2$ and $w\in A^p(\mathbb R^n)$.
Then we have $ {\rm BMO}_{w}(\mathbb{R}^n)$  and ${\rm BMO}_{\Delta, w}(\mathbb{R}^n) $
coincide and they have equivalent norms.
\end{theorem}

\begin{proof}
We first prove ${\rm BMO}_{\Delta,w}(\mathbb{R}^n)\supset {\rm BMO}_{w}(\mathbb{R}^n).$
To see this, it suffices to show that for any $f\in {\rm BMO}_{w}(\mathbb{R}^n)$ and
cube $P\subset \mathbb R^n$,
\begin{align}\label{e:inclusion classical weighted BMO}
{\rm I}:=\bigg({1\over w(P)} \sum_{\substack{Q\ {\rm dyadic} \\ Q\subset P }} \iint_{ \widehat{Q}}\Big| t^2\Delta e^{-t^2\Delta} f(y)\Big|^2 {t^n\over w(Q)} {dydt\over t}\bigg)^{1\over2}\ls \|f\|_{{\rm BMO}_{w}(\mathbb{R}^n)}.
\end{align}
Let $q_t(x,y)=-t^2\frac{d}{ds}|_{s=t^2}p_s(x,y)$ be the kernel of the operator $-t^2\Delta e^{-t^2\Delta}$,
where $p_t$ is as in \eqref{heat ker classical}. 
Then we see that
$\int_{\mathbb R^n}q_t(x,y)dy=0$ and
$$|q_t(x, y)|\ls t^{-n}e^{-\frac{|x-y|^2}{ct^2}}$$
for some $c>0$. We see that for any cube $Q\subset P$ and $y\in Q$,
\begin{align}\label{e:pw est ker q-t}
&\Big| t^2\Delta e^{-t^2\Delta} f(y)\Big|\\
&\quad=\Big| t^2\Delta e^{-t^2\Delta} (f-\left< f\right>_{B(y,t)})(y)\Big|\nonumber\\
&\quad\ls  t^{-n}\int_{\mathbb R^n}e^{-\frac{|y-z|^2}{ct^2}}|f(z)-\left< f\right>_{B(y,t)}|dz\nonumber\\
&\quad\ls t^{-n}\left[\int_{|y-z|<t}+\sum_{k=1}^\infty\int_{2^{k-1}t\le|y-z|<2^kt}\right]
e^{-\frac{|y-z|^2}{ct^2}}|f(z)-\left< f\right>_{B(y,t)}|dz\nonumber\\
&\quad\ls  t^{-n}w(B(y, t))\|f\|_{{\rm BMO}_{w}(\mathbb{R}^n)}
+\sum_{k=1}^\infty\bigg[w(B(y, 2^{k}t))\|f\|_{{\rm BMO}_{w}(\mathbb{R}^n)}\nonumber\\
&\quad\quad
+|B(y, 2^kt)|\sum_{j=1}^k|f_{B(y, 2^{j-1}t)}-\left< f\right>_{B(y, 2^jt)}|\bigg] t^{-n}e^{-c2^{2k}}\nonumber\\
&\quad\ls t^{-n}w(B(y, t))\|f\|_{{\rm BMO}_{w}(\mathbb{R}^n)}\nonumber\\
&\quad\quad+\sum_{k=1}^\infty\left[2^{kpn}w(B(y,t))+|B(y, 2^kt)|\sum_{j=1}^k\frac{w(B(y, 2^jt))}{|B(y, 2^jt)|}\right]t^{-n}e^{-c2^{2k}}\|f\|_{{\rm BMO}_{w}(\mathbb{R}^n)}\nonumber\\
&\quad\ls t^{-n}w(B(y,t))\|f\|_{{\rm BMO}_{w}(\mathbb{R}^n)}.\nonumber
\end{align}

Similarly, from the fact that $w(P)\approx w(P(y, \ell(P)))$ for any $y\in P$, we deduce that
for $\alpha>n$,
\begin{align*}
\Big| t^2\Delta e^{-t^2\Delta} f(y)\Big|
&\ls  t^{-n}\int_{\mathbb R^n}e^{-\frac{|y-z|^2}{ct^2}}|f(z)-\left< f\right>_{P(y,\ell(P))}|dz \\
&\ls t^{-n}\left[\int_{|y-z|<\ell(P)}+\sum_{k=1}^\infty\int_{2^{k-1}\ell(P)\le|y-z|<2^k\ell(P)}\right]
e^{-\frac{|y-z|^2}{ct^2}}|f(z)-\left< f\right>_{P(y,\ell(P))}|dz \\
&\ls  t^{-n}e^{-c\frac{[\ell(P)]^2}{t^2}}w(P(y,\ell(P))\|f\|_{{\rm BMO}_{w}(\mathbb{R}^n)}
+\sum_{k=1}^\infty\bigg[w(P(y, 2^{k}\ell(P)))\|f\|_{{\rm BMO}_{w}(\mathbb{R}^n)} \\
&\quad
+|P(y, 2^k\ell(P))|\sum_{j=1}^k|\left< f\right>_{P(y, 2^{j-1}\ell(P))}-\left< f\right>_{P(y, 2^j\ell(P))}|\bigg] t^{-n}e^{-c\frac{[2^{k}\ell(P)]^2}{t^2}} \\
&\ls\frac{t^\alpha}{[\ell(P)]^{\alpha+n}}w(P)\|f\|_{{\rm BMO}_{w}(\mathbb{R}^n)}\\
&\quad+ \sum_{k=1}^\infty\left[2^{kpn}w(P)+|2^kP|\sum_{j=1}^k\frac{w(2^jP)}{|2^jP|}\right]t^{-n}
\frac{t^{\alpha+n}}{2^{k(\alpha+n)\ell(P)^{k(\alpha+n)}}}\|f\|_{{\rm BMO}_{w}(\mathbb{R}^n)} \\
&\ls\frac{t^\alpha}{[\ell(P)]^{n+\alpha}}w(P)\|f\|_{{\rm BMO}_{w}(\mathbb{R}^n)}.
\end{align*}
These two inequalities imply that
\begin{align*}
{\rm I}&\ls \bigg({1\over w(P)} \sum_{\substack{Q\ {\rm dyadic} \\ Q\subset P }}
\iint_{ \widehat{Q}}\frac{w(B(y,t))}{t^n}\frac{t^\alpha w(P)}{[\ell(P)]^{n+\alpha}}\frac{t^{n-1}}{w(Q)} {dydt}\bigg)^{1\over2}\|f\|_{{\rm BMO}_{w}(\mathbb{R}^n)}\\
&\ls \bigg(\sum_{\substack{Q\ {\rm dyadic} \\ Q\subset P }}
\int_Q\int_0^{\ell(Q)}\frac{t^{\alpha-1}}{[\ell(P)]^{n+\alpha}}dydt\bigg)^{1\over2}\|f\|_{{\rm BMO}_{w}(\mathbb{R}^n)}\\
&\ls \|f\|_{{\rm BMO}_{w}(\mathbb{R}^n)}.
\end{align*}
This implies that ${\rm BMO}_{w}(\mathbb{R}^n)\subset {\rm BMO}_{\Delta,w}(\mathbb{R}^n)$.

Next, we prove ${\rm BMO}_{\Delta,w}(\mathbb{R}^n)\subset {\rm BMO}_{w}(\mathbb{R}^n)$
and for any $b\in {\rm BMO}_{\Delta,w}(\mathbb{R}^n)$, $b\in {\rm BMO}_{w}(\mathbb{R}^n)$
and $\|b\|_{{\rm BMO}_{w}(\mathbb{R}^n)}\ls \|b\|_{{\rm BMO}_{\Delta,w}(\mathbb{R}^n)}$.
To show this, we show that
\begin{align}\label{e:inclusion BMODelta1}
{\rm BMO}_{\Delta,w}(\mathbb{R}^n)\subset \big(H^1_{\Delta,w}(\R^n)\big)^*
\end{align}

To this end, for each $b\in {\rm BMO}_{\Delta,w}(\mathbb{R}^n)$, we now define a linear functional on $H^1_{\Delta,w}(\R^n)$ as follows:
\begin{align}\label{e-line func for BMO Neum}
\ell_b(f):=\langle b, f\rangle \quad{\rm for\ all\ } f\in H^1_{\Delta,w}(\R^n).
\end{align}
Now we only need to show that $\ell_b$ is a bounded functional. To see this, for any  $f\in H^1_{\Delta,w}(\R^n)$, we define
$\Omega_k=\left\{ x\in \R^n: S_\Delta(f)(x)>2^k\right\}$, and then we define $\widetilde{\Omega}_k$ and $B_k$ according to $\Omega_k$, using the same way as in \eqref{Omegakt} and \eqref{Bk}, respectively.

Next, from  Calder\'on's reproducing formula,
we have
\begin{align*}
|\langle b, f\rangle| &= c\bigg| \int_{\R^n} b(x) \int_0^\infty\int_{\R^n} t^2\Delta e^{-t^2\Delta}(x,y) t^2\Delta e^{-t^2\Delta}(f)(y) {dydt\over t} dx\bigg|\\
&= c\bigg|   \int_0^\infty\int_{\R^n} t^2\Delta e^{-t^2\Delta}(b)(y) t^2\Delta e^{-t^2\Delta}(f)(y) {dydt\over t} \bigg|\\
&=c\bigg|   \sum_k \sum_{\substack{\overline{Q}\in B_k\\ \overline{Q} {\rm\ maximal}}}\sum_{\substack{Q\in B_k\\ Q\subset \overline{Q}}} \iint_{\widehat Q} t^2\Delta e^{-t^2\Delta}(b)(y) t^2\Delta e^{-t^2\Delta}(f)(y)  {dydt\over t} \bigg|\\
&\leq C\sum_k \sum_{\substack{\overline{Q}\in B_k\\ \overline{Q} {\rm\ maximal}}}\bigg( w(\overline Q) \sum_{\substack{Q\in B_k\\ Q\subset \overline{Q}}} \iint_{\widehat Q}  |t^2\Delta e^{-t^2\Delta}(f)(y)|^2 w(Q) {dydt\over t^{n+1}}  \bigg)^{1\over 2}\\
&\hskip2cm\times\bigg( {1\over w(\overline Q)} \sum_{\substack{Q\in B_k\\ Q\subset \overline{Q}}} \iint_{\widehat Q}  |t^2\Delta e^{-t^2\Delta}(b)(y)|^2  {t^n\over w(Q)} {dydt\over t}  \bigg)^{1\over 2}\\
&\leq C\|b\|_{ {\rm BMO}_{\Delta,w}(\mathbb{R}^n)}  \sum_k \sum_{\substack{\overline{Q}\in B_k\\ \overline{Q} {\rm\ maximal}}} w(\overline Q)^{1\over2} \bigg( \sum_{\substack{Q\in B_k\\ Q\subset \overline{Q}}} \iint_{\widehat Q}  |t^2\Delta e^{-t^2\Delta}(f)(y)|^2 w(Q) {dydt\over t^{n+1}}  \bigg)^{1\over 2}\\
&\leq C\|b\|_{ {\rm BMO}_{\Delta,w}(\mathbb{R}^n)} \sum_k\bigg(  \sum_{\substack{\overline{Q}\in B_k\\ \overline{Q} {\rm\ maximal}}} w(\overline Q)\bigg)^{1\over2} \\
&\hskip2cm\times\bigg(  \sum_{\substack{\overline{Q}\in B_k\\ \overline{Q} {\rm\ maximal}}} \sum_{\substack{Q\in B_k\\ Q\subset \overline{Q}}} \iint_{\widehat Q}  |t^2\Delta e^{-t^2\Delta}(f)(y)|^2 w(Q) {dydt\over t^{n+1}}  \bigg)^{1\over 2}.
\end{align*}
Next, noting that from the definitions of $\widetilde{\Omega}_{k}$ and $\Omega_{k}$, we have
\begin{align*}
\int_{\widetilde{\Omega}_{k}\backslash \Omega_{k+1}} S_\Delta(f)(x)^2 w(x)dx \leq 2^{2(k+1)} w(\widetilde{\Omega}_{k}\backslash \Omega_{k+1} )\leq C 2^{2k} w(\Omega_{k}).
\end{align*}
Moreover, we have
\begin{align*}
&\int_{\widetilde{\Omega}_{k}\backslash \Omega_{k+1}} S_\Delta(f)(x)^2 w(x)dx \\
&\quad=  \int_{\widetilde{\Omega}_{k}\backslash \Omega_{k+1}}  \int_0^\infty \int_{\R^n} \unit_{\{|x-y|<t\}}  \Big| t^2\Delta e^{-t^2\Delta}f(y) \Big|^2 {dydt\over t^{n+1}}w(x)dx \\
&\quad\geq  \int_{\widetilde{\Omega}_{k}\backslash \Omega_{k+1}} \sum_{\substack{\overline{Q}\in B_k\\ \overline{Q} {\rm\ maximal}}} \sum_{\substack{Q\in B_k\\ Q\subset \overline{Q}}} \iint_{\widehat Q} \unit_{\{|x-y|<t\}}  \Big| t^2\Delta e^{-t^2\Delta}f(y) \Big|^2 {dydt\over t^{n+1}}w(x)dx \\
&\quad\geq C \sum_{\substack{\overline{Q}\in B_k\\ \overline{Q} {\rm\ maximal}}}  \sum_{\substack{Q\in B_k\\ Q\subset \overline{Q}}} \iint_{\widehat Q}  \Big| t^2\Delta e^{-t^2\Delta}f(y) \Big|^2 {dydt\over t^{n+1}} w(Q).
\end{align*}
Thus, we obtain that
\begin{align*}
&\bigg( \sum_{\substack{\overline{Q}\in B_k\\ \overline{Q} {\rm\ maximal}}}  \sum_{\substack{Q\in B_k\\ Q\subset \overline{Q}}} \iint_{\widehat Q}  \Big| t^2\Delta e^{-t^2\Delta}f(y) \Big|^2 {dydt\over t^{n+1}} w(Q)\bigg)^{1\over2}
\leq C 2^k w(\Omega_{k})^{1\over 2}.
\end{align*}
Then, we obtain that
\begin{align*}
|\langle b, f\rangle|
&\leq C\|b\|_{ {\rm BMO}_{\Delta,w}(\mathbb{R}^n)}  \sum_k  w(\Omega_k)^{1\over2} 2^k w(\Omega_{k})^{1\over 2}\\
&\leq C\|b\|_{ {\rm BMO}_{\Delta,w}(\mathbb{R}^n)}  \sum_k   2^k w(\Omega_{k})\\
&\leq C\|b\|_{ {\rm BMO}_{\Delta,w}(\mathbb{R}^n)} \|f\|_{  H^1_{\Delta,w}(\R^n)}.
\end{align*}
This shows that $\ell_b\in\big(H^1_{\Delta,w}(\R^n)\big)^\ast$ and
$\|\ell_b\|\ls \|b\|_{ {\rm BMO}_{\Delta,w}(\mathbb{R}^n)}$,
which implies that
\eqref{e:inclusion BMODelta1}
holds.

From Theorem \ref{t-coin Hardy space cla and Lapl}, we deduce that
\begin{align}\label{e:inclusion BMODelta2}
\big(H^1_{\Delta,w}(\R^n)\big)^* = \big(H^{1,\,p,\,\beta}_{w}(\mathbb{R}^n)\big)^*,\quad{\rm for}\, \beta \geq \lfloor np\rfloor-n.
\end{align}
Then for $b\in  {\rm BMO}_{\Delta,w}(\mathbb{R}^n)$, we have $\ell_b$ defined in \eqref{e-line func for BMO Neum}
belongs to $\big(H^{1,\,p,\,\beta}_{w}(\mathbb{R}^n)\big)^*$.
This together with Theorem \ref{t:dual BMO cla} (1) implies that there exists $\tilde b\in {\rm BMO}_{w}(\mathbb{R}^n)$ such that for any $f\in H^{1,\,p,\,\beta}_{w}(\mathbb{R}^n)= H^1_{\Delta,w}(\R^n)$,
$$\ell_b(f)=\langle\tilde b, f\rangle\,\,\textrm{and}\,\, \|\ell_b\|\approx \|\tilde b\|_{{\rm BMO}_{w}(\mathbb{R}^n)},$$
from which and \eqref{e-line func for BMO Neum} it follows that $\langle\tilde b-\tilde b, f\rangle=0$ for all $f\in H^{1,\,p,\,\beta}_{w}(\mathbb{R}^n)$,
and hence $b-\tilde b=0$ in ${\rm BMO}_{w}(\mathbb{R}^n)$. Therefore, we conclude that $b\in {\rm BMO}_{w}(\mathbb{R}^n)$ and
$$\|b\|_{{\rm BMO}_{w}(\mathbb{R}^n)}=\|\tilde b\|_{{\rm BMO}_{w}(\mathbb{R}^n)}\approx \|\ell_b\|
\ls \|b\|_{{\rm BMO}_{\Delta,w}(\mathbb{R}^n)}.$$
This shows ${\rm BMO}_{\Delta,w}(\mathbb{R}^n)\subset {\rm BMO}_{w}(\mathbb{R}^n)$,
which completes the proof of Theorem \ref{t: coinc clas BMO and Lapl}.
\end{proof}

\section{Weighted  ${\rm BMO}_{\Delta_N,w}(\mathbb{R}^n)$,
$H^1_{\Delta_N,w}(\mathbb{R}^n)$ and duality} 
\setcounter{equation}{0}
\label{s:BloomBMOHardyLaplace}

In this section, we introduce  and study the weighted BMO space ${\rm BMO}_{\Delta_N,w}(\mathbb{R}^n)$,
$H^1_{\Delta_N,w}(\mathbb{R}^n)$ in the Neumann setting on $\mathbb R^n$. We characterize
${\rm BMO}_{\Delta_N,w}(\mathbb{R}^n)$ via the  weighted BMO with Neumann on half spaces
$\mathbb R^n_+$ and $\mathbb R^n_-$.
We also show that the dual space of $H^1_{\Delta_N,w}(\mathbb{R}^n)$ is just ${\rm BMO}_{\Delta_N,w}(\mathbb{R}^n)$.

\subsection{\texorpdfstring{Weighted BMO space ${\rm BMO}_{\Delta_N,w}(\mathbb{R}^n)$}{Bloom-type weighted BMO space associated to the Neumann Laplacian}}

To begin with, we define $$ \mathcal{M}=\left\{ f\in L^1_{loc}(\mathbb{R}^n): \ \exists \epsilon>0\ s.t.\ \int_{\mathbb{R}^n} {|f(x)|^2\over 1+|x|^{n+\epsilon}} \,dx <\infty  \right\}. $$

\begin{definition}
Suppose $1<p<\infty$ and $w\in A^p_{\Delta_N}(\mathbb R^n)$.
We say that $f\in\mathcal{M}$ is in the weighted BMO space
associated with $\Delta_N$, denoted by ${\rm BMO}_{\Delta_N,w}(\mathbb{R}^n)$, if
\begin{align}\label{wBMON normN}
 \|f\|_{{\rm BMO}_{\Delta_N, w}(\mathbb{R}^n)}:=\sup_{P\subset \mathbb R^n}
  \bigg({1\over w(P)} \sum_{\substack{Q\ {\rm dyadic} \\ Q\subset P }}
  \iint_{ \widehat{Q}}\Big| t^2\Delta_N e^{-t^2\Delta_N} f(y)\Big|^2 {t^n\over w(Q)} {dydt\over t}\bigg)^{1\over2}<\infty,
  \end{align}
where the supremum is taken over all cubes $P$ in $\mathbb{R}^n$.
\end{definition}

To understand this new BMO space associated with $\Delta_N$, we introduce the following two auxiliary BMO spaces.

\begin{definition}
Suppose $1<p<\infty$ and $w\in A^p(\mathbb R^n_+)$. We say that $f\in\mathcal{M}$ is in weighted BMO space associated
with $\Delta_{N_+}$, denoted by ${\rm BMO}_{\Delta_{N_+},w}(\mathbb{R}^n_+)$, if
\begin{align}\label{wBMON norm}
 \|f\|_{{\rm BMO}_{\Delta_{N_+}, w}(\mathbb{R}_+^n)}:=\sup_{P\subset \mathbb R^n_+}
 \bigg({1\over w(P)} \sum_{\substack{Q\ {\rm dyadic} \\ Q\subset P }}
 \iint_{ \widehat{Q}}\Big| t^2\Delta_{N_+} e^{-t^2\Delta_{N_+}}
 f(y)\Big|^2 {t^n\over w(Q)} {dydt\over t}\bigg)^{1\over2}<\infty,
  \end{align}
where the supremum is taken over all cubes $P$ in $\mathbb{R}^n_+$. Similarly, we define the space
 ${\rm BMO}_{\Delta_{N_-},w}(\mathbb{R}^n_-)$ for $w$ in $A^p(\mathbb R_-^n)$ with $1<p<\infty$.
\end{definition}

\begin{definition}\label{d-BMO even}
Suppose $1<p<\infty$ and $w\in A^p(\mathbb R^n_+)$. We introduce the space
${\rm BMO}_{e,w}(\mathbb{R}^n_+)$ as follows: a function
$f\in {\rm BMO}_{e,w_+}(\mathbb{R}^n_+)$ if $f_e$ is in ${\rm BMO}_{w_{+,e}}(\mathbb{R}^n)$, and we define
$$ \|f\|_{{\rm BMO}_{e,w}(\mathbb{R}^n_+)} = \|f_e\|_{{\rm BMO}_{\Delta,w_{+,e}}(\mathbb{R}^n)}.$$
Symmetrically,
suppose $1<p<\infty$ and $w\in A^p(\mathbb R^n_-)$. We introduce the space
${\rm BMO}_{e,w_-}(\mathbb{R}^n_-)$ as follows:
a function $g\in {\rm BMO}_{e,w}(\mathbb{R}^n_-)$ if $g_e$ is in ${\rm BMO}_{w_{-,e}}(\mathbb{R}^n)$, and we define
$$ \|g\|_{{\rm BMO}_{e,w}(\mathbb{R}^n_-)} = \|g_e\|_{{\rm BMO}_{\Delta,w_{-,e}}(\mathbb{R}^n)}.$$
\end{definition}

We have the following observation.

\begin{theorem}\label{p-coinc Dz BMO}
Suppose $p\in (1,2]$ and $w\in A^p_{\Delta_N}(\mathbb R^n)$.
The spaces ${\rm BMO}_{\Delta_{N_+},w_+}(\mathbb{R}^n_+)$
and ${\rm BMO}_{e,w_+}(\mathbb{R}^n_+)$ coincide, with equivalent norms.
Similar result holds for ${\rm BMO}_{\Delta_{N_-},w_-}(\mathbb{R}^n_-)$ and ${\rm BMO}_{e,w_-}(\mathbb{R}^n_-)$.
\end{theorem}

\begin{proof}
Assume that $f\in {\rm BMO}_{e,w_+}(\mathbb{R}^n_+)$ first. To show $f\in {\rm BMO}_{\Delta_{N_+},w_+}(\mathbb{R}^n_+)$,
it suffices to show that for any cube $P\subset\mathbb R^n_+$,
\begin{align*}
\bigg({1\over w_+(P)} \sum_{\substack{Q\ {\rm dyadic} \\ Q\subset P }} \iint_{ \widehat{Q}}\Big| t^2\Delta_{N_+} e^{-t^2\Delta_{N_+}} f(y)\Big|^2 {t^n\over w_+(Q)} {dydt\over t}\bigg)^{1\over2}\ls \|f\|_{{\rm BMO}_{e,w_+}(\mathbb{R}^n_+)}.
\end{align*}
By \eqref{e:identity semigroup}, the fact that $w_{+, e}\in A^p(\mathbb R^n)$ and Theorem \ref{t: coinc clas BMO and Lapl}, we see that
\begin{align*}
&\bigg({1\over w_+(P)} \sum_{\substack{Q\ {\rm dyadic} \\ Q\subset P }} \iint_{ \widehat{Q}}\Big| t^2\Delta_{N_+} e^{-t^2\Delta_{N_+}} f(y)\Big|^2 {t^n\over w_+(Q)} {dydt\over t}\bigg)^{1\over2}\\
&\quad=\bigg({1\over w_+(P)} \sum_{\substack{Q\ {\rm dyadic} \\ Q\subset P }} \iint_{ \widehat{Q}}\Big| t^2\Delta e^{-t^2\Delta} f_e(y)\Big|^2 {t^n\over w_+(Q)} {dydt\over t}\bigg)^{1\over2}\\
&\quad\le\|f_e\|_{{\rm  BMO}_{\Delta, w_+,e}(\mathbb R^n)}\\
&\quad\sim\|f_e\|_{{\rm  BMO}_{w_+,e}(\mathbb R^n)}\\
&\quad\sim\|f\|_{{\rm BMO}_{e,w_+}(\mathbb{R}^n_+)}.
\end{align*}

Now assume that $f\in {\rm BMO}_{\Delta_{N_+},w_+}(\mathbb{R}^n_+)$. To show
$f\in {\rm BMO}_{e,w_+}(\mathbb{R}^n_+)$, by Theorem \ref{t: coinc clas BMO and Lapl},
it suffices to prove that $f_e\in {\rm BMO}_{w_{+, e}}(\mathbb{R}^n)$, that is, for any cube $P\subset \mathbb R^n$,
\begin{align}\label{e:inclusion weighted BMO}
\bigg({1\over w_{+, e}(P)} \sum_{\substack{Q\ {\rm dyadic} \\ Q\subset P }} \iint_{ \widehat{Q}}\Big| t^2\Delta e^{-t^2\Delta} f_e(y)\Big|^2 {t^n\over w_{+, e}(Q)} {dydt\over t}\bigg)^{1\over2}\ls\|f\|_{{\rm BMO}_{\Delta_{N_+},w_+}(\mathbb{R}^n_+)}.
\end{align}
We consider the following three cases:

 Case (i) $P\subset \mathbb R_+^n$.
In this case, we
 have
\begin{align*}
&\bigg({1\over w_{+, e}(P)} \sum_{\substack{Q\ {\rm dyadic} \\ Q\subset P }} \iint_{ \widehat{Q}}\Big| t^2\Delta e^{-t^2\Delta} f_e(y)\Big|^2 {t^n\over w_{+, e}(Q)} {dydt\over t}\bigg)^{1\over2}\\
&\quad=\frac1{w_{+}(P)}\sum_{\substack{Q\ {\rm dyadic} \\ Q\subset P }} \iint_{ \widehat{Q}}\Big| t^2\Delta e^{-t^2\Delta_{N, +}} f(y)\Big|^2 {t^n\over w_{+}(Q)} {dydt\over t}\bigg)^{1\over2}\\
&\quad\le\|f\|_{{\rm BMO}_{\Delta_{N_+},w_+}(\mathbb{R}^n_+)}.
\end{align*}

Case (ii) $P\subset \mathbb R^n_-$. In this case, since for any $x\in P$ and cube $Q\subset P$,
\begin{align}\label{e:identity neumann laplacian-2}
\exp(-\ell_Q^2\Delta_{N,+})f(\tilde x)
=\exp(-\ell_Q^2\Delta)f_e(x).
\end{align}
We then see that
\begin{align*}
&\bigg({1\over w_{+, e}(P)} \sum_{\substack{Q\ {\rm dyadic} \\ Q\subset P }} \iint_{ \widehat{Q}}\Big| t^2\Delta e^{-t^2\Delta} f_e(y)\Big|^2 {t^n\over w_{+, e}(Q)} {dydt\over t}\bigg)^{1\over2}\\
&\quad=\bigg({1\over w_{+, e}(P)} \sum_{\substack{Q\ {\rm dyadic} \\ Q\subset P }} \iint_{ \widehat{Q}}\Big| t^2\Delta_{N,+} e^{-t^2\Delta_{N,+}} f(\tilde y)\Big|^2 {t^n\over w_{+, e}(Q)} {dydt\over t}\bigg)^{1\over2}\\
&\quad=\bigg(\frac1{w_{+}(\widetilde P)}\sum_{\substack{Q\ {\rm dyadic} \\ Q\subset \widetilde P }} \iint_{ \widehat{Q}}\Big| t^2\Delta_{N,+} e^{-t^2\Delta_{N,+}} f(y)\Big|^2 {t^n\over w_{+}(Q)} {dydt\over t}\bigg)^{1\over2}\\
&\quad\le\|f\|_{{\rm BMO}_{\Delta_{N_+},w_+}(\mathbb{R}^n_+)},
\end{align*}
where $\widetilde{P}=\{\tilde x\in \mathbb R^n_+:  x\in P\}$.

Case (iii) $P_+=P\cap \mathbb R^n_-\not=\emptyset$ and $P_-=P\cap \mathbb R^n_+\not=\emptyset$.
In this case, let 
\begin{align}\label{e:cube hat posi}
\widehat{P_-}=\{(x', x_n):\quad x'\in P\cap \mathbb R^{n-1},\,-\ell_P<x_n\le 0\},
\end{align}
and
\begin{align}\label{e:cube hat nega}
\widehat{P_+}=\{(x', x_n):\quad x'\in P\cap \mathbb R^{n-1},\,0<x_n\le \ell_P\}.
\end{align}
As $w_{e,+}\in A^p(\mathbb R^n)$, by \eqref{e:identity semigroup} and
\eqref{e:identity neumann laplacian-2}, we have that
\begin{align*}
&\bigg({1\over w_{+, e}(P)} \sum_{\substack{Q\ {\rm dyadic} \\ Q\subset P }} \iint_{ \widehat{Q}}\Big| t^2\Delta e^{-t^2\Delta} f_e(y)\Big|^2 {t^n\over w_{+, e}(Q)} {dydt\over t}\bigg)^{1\over2}\\
&\quad=\frac1{w_{+,e}(P)}\bigg(\bigg[\sum_{\substack{Q\ {\rm dyadic} \\ Q\subset P_+ }}+\sum_{\substack{Q\ {\rm dyadic} \\ Q\subset P_- }}\bigg] \iint_{ \widehat{Q}}\Big| t^2\Delta e^{-t^2\Delta} f_e(y)\Big|^2 {t^n\over w_{+, e}(Q)} {dydt\over t}\bigg)^{1\over2}\\
&\quad\ls\frac1{w_{+}(\widehat{P_+})}\bigg(\sum_{\substack{Q\ {\rm dyadic} \\ Q\subset \widehat{P_+} }} \iint_{ \widehat{Q}}\Big| t^2\Delta_{N,+} e^{-t^2\Delta_{N, +}} f(y)\Big|^2 {t^n\over w_{+}(Q)} {dydt\over t}\bigg)^{1\over2}\\
&\quad\quad+\frac1{w_{+}(\widehat{P_+})}\bigg(\sum_{\substack{Q\ {\rm dyadic} \\ Q\subset P_- }} \iint_{ \widehat{Q}}\Big| t^2\Delta_{N,+} e^{-t^2\Delta_{N, +}} f(\tilde y)\Big|^2 {t^n\over w_{+,e}(Q)} {dydt\over t}\bigg)^{1\over2}\\
&\quad\ls\|f\|_{{\rm BMO}_{\Delta_{N_+},w_+}(\mathbb{R}^n_+)}+
\frac1{w_{+}(\widehat{P_+})}\bigg(\sum_{\substack{Q\ {\rm dyadic} \\ Q\subset \widehat{P_+} }} \iint_{ \widehat{Q}}\Big| t^2\Delta_{N,+} e^{-t^2\Delta_{N, +}} f(y)\Big|^2 {t^n\over w_{+}(Q)} {dydt\over t}\bigg)^{1\over2}\\
&\quad\ls\|f\|_{{\rm BMO}_{\Delta_{N_+},w_+}(\mathbb{R}^n_+)}.
\end{align*}
Combining the three estimates above, we conclude that \eqref{e:inclusion weighted BMO} holds.
\end{proof}

\begin{prop} \label{p-coinc Dz-BMO-2}Suppose $p\in (1,2]$ and $w\in A^p_{\Delta_N}(\mathbb R^n)$.
Then the Neumann {\rm BMO} space ${\rm BMO}_{\Delta_N,w}(\mathbb{R}^n)$
can be described in the following way:
$${\rm BMO}_{\Delta_N,w}(\mathbb{R}^n)=
\left\{f\in\mathcal{M}: \ f_+\in {\rm BMO}_{e,w_+}(\mathbb{R}^n_+) {\rm\ and\ }
 f_-\in {\rm BMO}_{e,w_-}(\mathbb{R}^n_-)  \right\};$$
 Moreover, we have that
 $$\|f\|_{{\rm BMO}_{\Delta_N,w}(\mathbb{R}^n)}\approx \|f_+\|_{{\rm BMO}_{e,w_+}(\mathbb{R}^n_+)}
 +\|f_-\|_{{\rm BMO}_{e,w_-}(\mathbb{R}^n_-)}.$$
\end{prop}

\begin{proof}
Firstly, let $f\in {\rm BMO}_{\Delta_N,w}(\mathbb{R}^n)$. Then by
Proposition \ref{p-coinc Dz BMO}, and the properties in \eqref{Delta N} and \eqref{Delta N-exp},
we see that
$f_+\in {\rm BMO}_{e,w_+}(\mathbb{R}^n_+)$ and $f_-\in {\rm BMO}_{e,w_-}(\mathbb{R}^n_-)$
and
$$\|f_+\|_{{\rm BMO}_{e,w_+}(\mathbb{R}^n_+)}+\|f_-\|_{{\rm BMO}_{e,w_-}(\mathbb{R}^n_-)}
\ls \|f\|_{{\rm BMO}_{\Delta_N,w}(\mathbb{R}^n)}.$$

Conversely, for $w\in A^p_{\Delta_N}(\mathbb R^n)$ and $f$ on $\mathbb R^n$
 such that $f_+\in {\rm BMO}_{e,w_+}(\mathbb{R}^n_+)$ and $f_-\in {\rm BMO}_{e,w_-}(\mathbb{R}^n_-)$.
 Another  application of Proposition \ref{p-coinc Dz BMO} shows that
 $f_+\in {\rm BMO}_{\Delta_{N_+},w_+}(\mathbb{R}^n_+)$ and $f_-\in {\rm BMO}_{\Delta_{N_-},w_-}(\mathbb{R}^n_-)$.
 To show $f\in {\rm BMO}_{\Delta_N,w}(\mathbb{R}^n)$, it suffices to show that for any
 cube $P\subset \mathbb R^n$,
\begin{align}\label{e:Ip}
{\rm I}_P&:=\bigg({1\over w(P)} \sum_{\substack{Q\ {\rm dyadic} \\ Q\subset P }}
 \iint_{ \widehat{Q}}\Big| t^2\Delta_{N} e^{-t^2\Delta_{N}}
 f(y)\Big|^2 {t^n\over w(Q)} {dydt\over t}\bigg)^{1\over2} \\
&\ls \|f_+\|_{{\rm BMO}_{\Delta_{N_+},w_+}(\mathbb{R}^n_+)}
+\|f_-\|_{{\rm BMO}_{\Delta_{N_-},w_-}(\mathbb{R}^n_-)}.\nonumber
\end{align}
If $P\subset \mathbb R^n_+$, then from \eqref{Delta N} and \eqref{Delta N-exp}, it follows that
\begin{align*}
{\rm I}_P&=\bigg({1\over w_+(P)} \sum_{\substack{Q\ {\rm dyadic} \\ Q\subset P }}
 \iint_{ \widehat{Q}}\Big| t^2\Delta_{N_+} e^{-t^2\Delta_{N_+}}
 f_+(y)\Big|^2 {t^n\over w_+(Q)} {dydt\over t}\bigg)^{1\over2} \\
&\le \|f_+\|_{{\rm BMO}_{\Delta_{N_+},w_+}(\mathbb{R}^n_+)}.
\end{align*}
Similarly, if $P\subset \mathbb R^n_-$, then
\begin{align*}
{\rm I}_P\le\|f_-\|_{{\rm BMO}_{\Delta_{N_-},w_-}(\mathbb{R}^n_-)}.
\end{align*}
Now assume that $P_+=P\cap \mathbb R_+^n\not=\emptyset$ and $P_-=P\cap \mathbb R_-^n\not=\emptyset$.
We first have
\begin{align*}
{\rm I}_P&\le \bigg({1\over w(P)} \sum_{\substack{Q\ {\rm dyadic} \\ Q\subset P_+ }}
 \iint_{ \widehat{Q}}\Big| t^2\Delta_{N_+} e^{-t^2\Delta_{N_+}}
 f_+(y)\Big|^2 {t^n\over w_+(Q)} {dydt\over t}\bigg)^{1\over2} \\
 &\quad+\bigg({1\over w(P)} \sum_{\substack{Q\ {\rm dyadic} \\ Q\subset P_- }}
 \iint_{ \widehat{Q}}\Big| t^2\Delta_{N_-} e^{-t^2\Delta_{N_-}}
 f_-(y)\Big|^2 {t^n\over w_-(Q)} {dydt\over t}\bigg)^{1\over2} .
 \end{align*}
Observe that the interior $\PP$ can be written as the
union of a sequence $\{P_k\}_k$ of maximal dyadic cubes such that
$\PP=\cup_k P_k$ and ${\mathop P \limits^{ \circ}}_k\cap {\mathop P \limits^{ \circ}}_i=\emptyset$ if $k\not=i$.
Therefore, we obtain that
\begin{align*}
&\bigg({1\over w(P)} \sum_{\substack{Q\ {\rm dyadic} \\ Q\subset P_- }}
 \iint_{ \widehat{Q}}\Big| t^2\Delta_{N_-} e^{-t^2\Delta_{N_-}}
 f_-(y)\Big|^2 {t^n\over w_-(Q)} {dydt\over t}\bigg)^{1\over2} \\
 &\quad=\bigg({1\over w(P)} \sum_k\sum_{\substack{Q\ {\rm dyadic} \\ Q\subset P_k }}
 \iint_{ \widehat{Q}}\Big| t^2\Delta_{N_-} e^{-t^2\Delta_{N_-}}
 f_-(y)\Big|^2 {t^n\over w_-(Q)} {dydt\over t}\bigg)^{1\over2} \\
  &\quad\le\bigg({1\over w(P)} \sum_k w_-(P_k)\bigg)^{1\over2} \|f_-\|_{{\rm BMO}_{\Delta_{N_-},w_-}(\mathbb{R}^n_-)}\\
  &\quad\le\|f_-\|_{{\rm BMO}_{\Delta_{N_-},w_-}(\mathbb{R}^n_-)}.
 \end{align*}
 The estimate for $f_+$ is similar and omitted. Thus we see that \eqref{e:Ip} holds.
\end{proof}

\subsection{\texorpdfstring{The Weighted Hardy Space $H^1_{\Delta_N,w}(\mathbb{R}^n)$}{The Weighted Hardy Space associated to the Neumann Laplacian}}

Consider 
the  Littlewood--Paley area function associated with $\Delta_N$ as follows.
\begin{align}\label{SfDeltaN}
S_{\Delta_N}(f)(x) =\bigg( \iint_{\Gamma_{\Delta_N}(x)} \Big| t^2\Delta_N e^{-t^2\Delta_N}f(y) \Big|^2 {dydt\over t^{n+1}}\bigg)^{1\over2},
\end{align}
where $\Gamma_{\Delta_N}(x)$ is the cone defined as
$$\Gamma_{\Delta_N}(x):= \{ (y,t)\in \R^n\times (0,\infty): |x-y|<t, H(x_ny_n)=1\},$$
where $H(t)$ is the Heaviside function defined as in \eqref{e:hvyside}.

\begin{definition}
Suppose $1<p<\infty$ and $w\in A^p_{\Delta_N}(\mathbb R^n)$.
We define the weighted Hardy space $H^1_{\Delta_N,w}(\mathbb{R}^n)$
as
$ H^1_{\Delta_N,w}(\R^n):=\{ f\in L^1_w(\R^n):\ S_{\Delta_N}(f)\in L^1_w(\R^n)\} $
with the norm  $\|f\|_{H^1_{\Delta_N,w}(\R^n)}:=\|S_{\Delta_N}(f)\|_{L^1_w(\R^n)}$.
\end{definition}

We also introduce the following auxiliary Hardy spaces on half spaces.
\begin{definition}\label{d-H1 even}
Suppose $1<p<\infty$ and $w$ is a weight on $\mathbb R^n_+$.
We say a function $f\in L^1_w(\mathbb R^n_+)$ belongs to
$H^1_{e,w}(\mathbb{R}^n_+)$ if $f_e\in H^1_{w_{+,e}}(\R^n) $
with the norm $\|f\|_{H^1_{e,w}(\mathbb{R}^n_+)}:=\|f_e\|_{H^1_{w_{+,e}}(\R^n)}$.
Symmetrically, suppose $w$ is a weight on $\mathbb R^n_-$. We say a function $g\in L^1_w(\mathbb R^n_-)$
belongs to $H^1_{e,w}(\mathbb{R}^n_-)$ if $g_e\in  H^1_{w_{-,e}}(\R^n)$ with the norm
$\|g\|_{H^1_{e,w}(\mathbb{R}^n_-)}:=\|g_e\|_{H^1_{w_{-,e}}(\R^n)}$.
\end{definition}

\begin{prop} \label{p-coinc Dz H1}
Let $p\in (1,2]$ and $w\in A^p_{\Delta_N}(\mathbb R^n)$.
Then the  space $H^1_{\Delta_N,w}(\R^n)$
can be characterised in the following way:
$$H^1_{w,S_{\Delta_N}}(\R^n)=
\left\{f\in L^1_w(\mathbb R^n): \ f_+\in H^1_{e,w_+}(\mathbb{R}^n_+) {\rm\ and\ }
 f_-\in H^1_{e,w_-}(\mathbb{R}^n_-)  \right\}.$$
\end{prop}

\begin{proof}
Suppose $f\in L^1_w(\mathbb R^n)$ such that $ \ f_+\in H^1_{e,w_+}(\mathbb{R}^n_+) {\rm\ and\ }
 f_-\in H^1_{e,w_-}(\mathbb{R}^n_-)$.
Note that
\begin{eqnarray*}
&&t^2\Delta_N\exp(-t^2\Delta_N)f(x)=t^2\Delta\exp(-t^2\Delta)f_{+,e}(x)\quad {\rm for}\ x\in \mathbb{R}^n_+,\ {\rm and} \\
&&t^2\Delta_N\exp(-t^2\Delta_N)f(x)=t^2\Delta\exp(-t^2\Delta)f_{-,e}(x)\quad {\rm for}\ x\in \mathbb{R}^n_-.
\end{eqnarray*}

Moreover, by a change of variable,
\begin{eqnarray*} 
&&t^2\Delta_N\exp(-t^2\Delta_N)f(x)=t^2\Delta\exp(-t^2\Delta)f_{+,e}(\widetilde{x})\
\ {\rm for\ any\ }t>0, \ x\in \mathbb{R}^n_+;\\
&&t^2\Delta_N\exp(-t^2\Delta_N)f(x)=t^2\Delta\exp(-t^2\Delta)f_{-,e}(\widetilde{x})\
\ {\rm for\ any\ }t>0, \ x\in \mathbb{R}^n_-,\nonumber
\end{eqnarray*}
where for every $x=(x_1,\ldots,x_{n-1},x_n)$, we use $\widetilde x$ to denote the reflection of $x$, i.e.,
$\widetilde x=(x_1,\ldots,x_{n-1},-x_n)$.

Then  we have that for $x\in \R^n_+$,
\begin{align*}
S_{\Delta_N}(f)(x)^2
&=\iint_{\Gamma_{\Delta_N}(x)}  |t^2\Delta_N\exp(-t^2\Delta_N)f(y)|^2 \,\frac{dydt}{t^{n+1}}\\
&=
   \int_0^\infty \int_{ |x-y|<t, y\in\mathbb{R}_+^n }  |t^2\Delta_N\exp(-t^2\Delta_N)f(y)|^2 \,\frac{dydt}{t^{n+1}} \\
&=\int_0^\infty \int_{ |x-y|<t, y\in\mathbb{R}_+^n }  |t^2\Delta\exp(-t^2\Delta)f_{+,e}(y)|^2\, \frac{dydt}{t^{n+1}}\\
&=\frac{1}{2}\int_0^\infty \int_{ |x-y|<t }  |t^2\Delta\exp(-t^2\Delta)f_{+,e}(y)|^2\, \frac{dydt}{t^{n+1}},
\end{align*}
which implies that $ S_{\Delta_N}(f)(x)= \frac{\sqrt{2}}{2}  S_\Delta(f_{+,e})(x) $.
Symmetrically, for $x\in \R^n_-$, we have $ S_{\Delta_N}(f)(x)= \frac{\sqrt{2}}{2}  S_\Delta(f_{-,e})(x) $.

As a consequence, we have
\begin{align*}
\|f\|_{H^1_{\Delta_N,w}(\mathbb{R}^n)}&=\int_{\mathbb{R}^n}  |S_{\Delta_N}(f)(x)|\,w(x)dx\\
&=\int_{\mathbb{R}^n_+}  |S_{\Delta_N}(f)(x)|\,w_+(x)dx+\int_{\mathbb{R}^n_-}  |S_{\Delta_N}(f)(x)|\,w_-(x)dx\\
&\approx \int_{\mathbb{R}^n_+}  |S_\Delta(f_{+,e})(x)|\,w_+(x)dx+ \int_{\mathbb{R}^n_-}  |S_\Delta(f_{-,e})(x)|\,w_-(x)dx\nonumber\\
&\approx \int_{\mathbb{R}^n}  |S_\Delta(f_{+,e})(x)|\,w_{+,e}(x)dx+ \int_{\mathbb{R}^n}  |S_\Delta(f_{-,e})(x)|\,w_{-,e}(x)dx\nonumber\\
 &= \|f_{+,e}\|_{H^1_{\Delta,w_{+,e}}(\mathbb{R}^n)}+\|f_{-,e}\|_{H^1_{\Delta,w_-,e}(\mathbb{R}^n)} \nonumber\\
 &= \|f_{+}\|_{H^1_{e,w_+}(\mathbb{R}^n_+)}+\|f_{-}\|_{H^1_{e,w_-}(\mathbb{R}^n_-)}.
\end{align*}
Here we have used the facts that
$$\int_{\mathbb{R}^n_+}  |S_\Delta(f_{+,e})(x)|\,w_+(x)dx\approx \int_{\mathbb{R}^n}  |S_\Delta(f_{+,e})(x)|\,w_{+,e}(x)dx$$
and that
$$\int_{\mathbb{R}^n_-}  |S_\Delta(f_{-,e})(x)|\,w_-(x)dx\approx \int_{\mathbb{R}^n}  |S_\Delta(f_{-,e})(x)|\,w_{-,e}(x)dx,$$
both of which follows from changing of variables and the fact that both $w_{+,e}$ ($w_{-,e}$ resp.) and
$f_{+,e}$ ($f_{-,e}$ resp.) are even functions with respect to the $n$th component.

Conversely, suppose $f\in H^1_{\Delta_N,w}(\mathbb{R}^n)$.
Actually, we just need to reverse the calculations above, and then we get
$\|f_{+}\|_{H^1_{e,w_+}(\mathbb{R}^n_+)}+\|f_{-}\|_{H^1_{e,w_-}(\mathbb{R}^n_+)}
\lesssim \|f\|_{H^1_{\Delta_N,w}(\mathbb{R}^n)}.
$
\end{proof}


\begin{theorem}\label{t-dual}
 Let $p\in (1,2]$ and $w\in A^p_{\Delta_N}(\mathbb R^n)$.
Then $[H^1_{\Delta_N,w}(\mathbb{R}^n)]'={\rm BMO}_{\Delta_N,w}(\mathbb{R}^n)$
with equivalent norms.
\end{theorem}

\begin{proof}
By \cite{Ca}*{p.22}, Propositions \ref{p-coinc Dz H1} and \ref{p-coinc Dz-BMO-2}
it suffices to show that  $[H^1_{e,w_+}(\mathbb{R}^n_+)]'={\rm BMO}_{e,w_+}(\mathbb{R}^n_+)$
 and $[H^1_{e,w_-}(\mathbb{R}^n_-)]'={\rm BMO}_{e,w_-}(\mathbb{R}^n_-).$
 Moreover, from Definition \ref{d-H1 even}, we get that
the mapping from  $H^1_{e,w_+}(\mathbb{R}^n_+)$($H^1_{e,w_-}(\mathbb{R}^n_-)$ resp.)
to $H^1_{w_{+,e}}(\R^n)$ ($H^1_{w_{-,e}}(\R^n)$) is an isometry homomorphism, and so is
the map from ${\rm BMO}_{e,w_+}(\mathbb{R}^n_+)$ (${\rm BMO}_{e,w_-}(\mathbb{R}^n_-)$ resp.)
to ${\rm BMO}_{w_{e, +}}(\mathbb{R}^n)$ (${\rm BMO}_{w_{e, -}}(\mathbb{R}^n)$ resp.)
by Definition \ref{d-BMO even}. Observe that
$[H^1_{w_{+,e}}(\R^n)]'={\rm BMO}_{w_{+,e}}(\mathbb{R}^n)$ and
$[H^1_{w_{-,e}}(\R^n)]'={\rm BMO}_{w_{-,e}}(\mathbb{R}^n)$ with equivalent norms, respectively.
Then the proof of Theorem \ref{t-dual} is completed.
\end{proof}

\section{Proofs of Theorems \ref{t:upp-low Neum Lap ha-sp commu} and \ref{t:upp-low Neum Lap wh-sp commu}}
\setcounter{equation}{0}
\label{s:upp-low Neum Lap commu}

\begin{proof}[\bf Proof of Theorem \ref{t:upp-low Neum Lap wh-sp commu}]
Suppose $1<p<\infty$ and $\mu,\lambda \in  A^p_{\Delta_N}(\mathbb R^n)$. Set $\nu = \mu^{1\over p} \lambda^{-{1\over p}}$.

{\bf Proof of the upper bound}:

Suppose $b\in  {\rm BMO}_{\Delta_N,\nu}({\mathbb R}^n)$.  We claim that
for each $l\in\{1, 2, \ldots, n\}$,
there is a positive constant $C$, depending only on $n,p,\mu,\lambda$ such that
\begin{align}\label{upper}
\| [b,R_{N, l}] \|_{ L^{p}_{\mu}(\mathbb R^n)\to  L^{p}_{\lambda}(\mathbb R^n)}\leq C \|b\|_{ {\rm BMO}_{\Delta_N,\nu}({\mathbb R}^n)}.
\end{align}

To begin with, for $b\in  {\rm BMO}_{\Delta_N,\nu}({\mathbb R}^n)$,  according to Proposition \ref{p-coinc Dz-BMO-2}, we have that
$b_{+,e}\in {\rm BMO}_{\nu_{+,e}}(\mathbb{R}^n)$ and $b_{-,e}\in \rm BMO_{\nu_{-,e}}(\mathbb{R}^n)$, and moreover,
$$ \|b\|_{{\rm BMO}_{\Delta_N,\nu}(\mathbb{R}^n)} \approx  \| b_{+,e}\|_{\rm BMO_{\nu_{+,e}}(\mathbb{R}^n)} + \|b_{-,e}\|_{\rm BMO_{\nu_{-,e}}(\mathbb{R}^n)}. $$

For every $f\in L^p_\mu(\mathbb{R}^n)$, we have
\begin{align*}
\| [b,R_{N,l}](f) \|_{   L^{p}_{\lambda}(\mathbb R^n)}^p &= \int _{\mathbb{R}^n_+}\big| [b,R_{N,l}](f)(x)\big|^p \,\lambda(x)dx+ \int _{\mathbb{R}^n_-} \big|[b,R_{N,l}](f)(x)\big|^p \,\lambda(x)dx\\
&=: I+II.
\end{align*}
For the term $I$, note that when $x\in\mathbb{R}^n_+$, we have  that $\lambda(x) = \lambda_{+,e}(x)$ and that
\begin{align*}
[b,R_{N,l}](f)(x)&= b(x)R_{N,l}(f)(x) - R_{N,l}(bf)(x)\\
&= b_{+,e}(x) R_l(f_{+,e})(x) - R_l(b_{+,e}f_{+,e})(x)= [b_{+,e},R_l](f_{+,e})(x),
\end{align*}
which implies that
\begin{align*}
I 
\leq  \int _{\mathbb{R}^n} \big|[b_{+,e},R_l](f_{+,e})(x)\big|^p \,\lambda_{+,e}(x)dx\leq C   \| b_{+,e} \|_{\rm BMO_{\nu_{+,e}}(\mathbb{R}^n)}^p   \| f_{+,e} \|_{L^p_{\mu_{+,e}}(\mathbb{R}^n)}^p,
\end{align*}
where $R_l$ is the classical $l$-th Riesz transform ${\partial\over\partial x_l} \Delta^{-{1\over2}}$,
and for the last estimate we use the result \cite{HLW}*{Theorem 1.1}.  Similarly we can obtain that
\begin{align*}
II
&\leq C    \| b_{-,e} \|_{\rm BMO_{\nu_{-,e}}(\mathbb{R}^n)}^p   \| f_{-,e} \|_{L^p_{\mu_{-,e}}(\mathbb{R}^n)}^p.
\end{align*}
Combining the estimates for $I$ and $II$ above, we obtain that
\begin{align*}
\| [b,R_{N,l}](f) \|_{   L^{p}_{\lambda}(\mathbb R^n)}^p &\leq C   \| b_{+,e} \|_{\rm BMO_{\nu_{+,e}}(\mathbb{R}^n)}^p   \| f_{+,e} \|_{L^p_{\mu_{+,e}}(\mathbb{R}^n)}^p+ C \| b_{-,e} \|_{\rm BMO_{\nu_{-,e}}(\mathbb{R}^n)}^p   \| f_{-,e} \|_{L^p_{\mu_{-,e}}(\mathbb{R}^n)}^p\\
&\leq  C\| b \|_{{\rm BMO}_{\Delta_N,\nu}(\mathbb{R}^n)}^p  \Big( \| f_{+,e} \|_{L^p_{\mu_{+,e}}(\mathbb{R}^n)}^p+ \| f_{-,e} \|_{L^p_{\mu_{-,e}}(\mathbb{R}^n)}^p\Big)\\
&\leq  C\| b \|_{{\rm BMO}_{\Delta_N,\nu}(\mathbb{R}^n)}^p  \| f \|_{L^p_{\mu}(\mathbb{R}^n)}^p,
\end{align*}
which yields that
\eqref{upper} holds.

\medskip

{\bf Proof of the lower bound}:

Suppose $1<p<\infty$ and $\mu,\lambda \in  A^p_{\Delta_N}(\mathbb R^n)$.
Then from the definition of $A^p_{\Delta_N}(\R^n)$, we also obtain that $\mu_{+,e}$, $\mu_{-,e}$, $\lambda_{+,e}$ and $\lambda_{-,e}$ are in $A^p(\R^n)$ and we have
$$ [\mu]_{A^p_{\Delta_N}(\mathbb R^n)} = [\mu_{+,e}]_{A^p(\R^n)} + [\mu_{-,e}]_{A^p(\R^n)} $$
and
$$ [\lambda]_{A^p_{\Delta_N}(\mathbb R^n)} = [\lambda_{+,e}]_{A^p(\R^n)} + [\lambda_{-,e}]_{A^p(\R^n)}. $$
We now set $\nu = \mu^{1\over p} \lambda^{-{1\over p}}$ and hence we have $\nu_{+,e} = \mu_{+,e}^{1\over p} \lambda_{+,e}^{-{1\over p}}$ and $\nu_{-,e} := \mu_{-,e}^{1\over p} \lambda_{-,e}^{-{1\over p}}$. Then, from the property of $A^p(\R^n)$ as mentioned in Section 3, we see that both $\nu_{+,e} $ and $\nu_{-,e} $ are in $A^2(\R^n)$. This again, implies that $\nu$ itself is in $A^2_{\Delta_N}(\mathbb R^n)$.

Suppose
$b\in L^1_{\loc}(\mathbb R^n)$.
Suppose that for $l=1,\ldots,n$,  there is a positive constant $C_l$, depending only on $n,p,\mu,\lambda$ such that
\begin{align}\label{upper1}
\| [b,R_{N,l}] \|_{ L^{p}_{\mu}(\mathbb R^n)\to  L^{p}_{\lambda}(\mathbb R^n)}= C_l<\infty.
\end{align}
We will show that $b$ is in ${\rm BMO}_{\Delta_N,\nu}(\mathbb{R}^n)$ with the norm satisfying
\begin{align}\label{lower}
\|b\|_{{\rm BMO}_{\Delta_N,\nu}(\mathbb{R}^n)}\ls C_l.
\end{align}
To this end, we first claim that for any $f\in L^p_{\mu_+, e}(\R^n)$,
\begin{align}\label{e:bdd commt even ext}
\|[b_{ +, e}, R_l](f)\|_{ L^{p}_{\lambda_{+,e}}(\mathbb R^n)}\ls C_l\|f\|_{L^p_{\mu_{+,e}}(\R^n)}.
\end{align}
In fact, for every $f \in L^p_\mu(\R^n)$, from \eqref{upper1}, we have
\begin{align}\label{upper2}
\| [b,R_{N, l}] (f)\|_{ L^{p}_{\lambda}(\mathbb R^n)}\leq C_l\|f\|_{L^p_\mu(\R^n)}.
\end{align}
In particular, we consider $f_+$ and $f_-$
which are the restrictions of $f$ onto $\R^n_+$ and $\R^n_-$, respectively.
It is clear that both $f_+$ and $f_-$ are in $L^p_\mu(\R^n)$. Now by substituting $f_+$ into \eqref{upper2} to replace $f$, we obtain that
\begin{align}\label{upper3}
\| [b,R_{N,l}] (f_+)\|_{ L^{p}_{\lambda}(\mathbb R^n)}\leq C_l\|f_+\|_{L^p_\mu(\R^n)}.
\end{align}
Next, note that from the definition of the commutator
\begin{align}\label{eeeee r}
[b,R_{N,l}](f_+)(x)&= b(x)R_{N,l}(f_+)(x) - R_{N,l}(bf_+)(x),
\end{align}
and then from the kernel condition on $R_{N,l}$,  we see that the variable $x$ in \eqref{eeeee r} above is actually restricted on $\R^n_+$, which further implies that
\begin{align*}
[b,R_{N,l}](f_+)(x)
&= b_{+,e}(x) R_l(f_{+,e})(x) - R_l(b_{+,e}f_{+,e})(x)= [b_{+,e},R_l](f_{+,e})(x).
\end{align*}
This, together with \eqref{upper3},
 implies that
\begin{align}\label{upper4}
\| [b_{+,e},R_l] (f_{+,e})\|_{ L^{p}_{\lambda_+}(\mathbb R^n_+)}\leq C_l\|f_{+,e}\|_{L^p_{\mu_+}(\R^n_+)}, \quad l=1,\ldots,n.
\end{align}
Moreover, for every $x=(x_1,\ldots,x_{n-1},x_n)\in \R^n_+$, we use $\widetilde x=(x_1,\ldots,x_{n-1}, -x_n)$ to denote the reflection of $x$
in $\R^n_-$.
Then we obtain that for $l=1,\ldots,n-1$, $$[b_{+,e},R_l] (f_{+,e})(\widetilde x) = [b_{+,e},R_l] (f_{+,e})(x),$$ and that for $l=n$,
 $$[b_{+,e},R_n] (f_{+,e})(\widetilde x) = -[b_{+,e},R_n] (f_{+,e})(x).$$
Combining these two equalities, the upper bound in \eqref{upper4} and the fact that $\|f_{+,e}\|_{L^p_{\mu_+}(\R^n_+)} \approx \|f_{+,e}\|_{L^p_{\mu_{+,e}}(\R^n)}$, we obtain that
\begin{align}\label{upper5}
\| [b_{+,e},R_l] (f_{+,e})\|_{ L^{p}_{\lambda_{+,e}}(\mathbb R^n)}\ls C_l\|f_{+,e}\|_{L^p_{\mu_{+,e}}(\R^n)}, \quad l=1,\ldots,n.
\end{align}

Moreover, let $f_{+,o}$ be the odd extension of $f_+$ to $\mathbb R^n$. We see that
for $l=1,\ldots,n-1$, $$[b_{+,e},R_l] (f_{+,o})(\widetilde x) = -[b_{+,e},R_l] (f_{+,o})(x),$$ and that for $l=n$,
 $$[b_{+,e},R_n] (f_{+,o})(\widetilde x) = [b_{+,e},R_n] (f_{+,o})(x).$$
This implies that
\begin{align*}
\| [b_{+,e},R_l] (f_{+,o})\|_{ L^{p}_{\lambda_{+,e}}(\mathbb R^n)}\ls C_l\|f_{+,o}\|_{L^p_{\mu_{+,e}}(\R^n)}, \quad l=1,\ldots,n.
\end{align*}
From this and \eqref{upper5}, we conclude that the claim \eqref{e:bdd commt even ext} holds true.

We borrow an idea from \cite{J}.  Observe that for any $l\in\{1,\,\cdots,\,n\}$, $1/R_l\in
C^\infty(\mathbb R^n\setminus\{0\})$. Therefore, there exist $z_0\in
\mathbb R^n\setminus\{0\}$ and $\delta\in(0, \infty)$ such that
$1/{R_l(z)}$ is expressed as an absolutely convergent Fourier series in the ball
$B(z_0, \sqrt n\delta)$ (see, for example, \cite[Theorem 3.2.16]{Gra}). That is, there exist $\{\sigma_k\}_{k\in\mathbb N}\subset\mathbb R^n$
and numbers $\{a_k\}_{k\in\mathbb N}$ with $\sum_{k=1}^\infty|a_k|<\infty$ such that for
all $z\in B(z_0, \sqrt n\delta)$,
$1/{R_l(z)}=\sum_{k=1}^\infty a_ke^{i\sigma_k\cdot  z}$.
Let $z_1:=\delta^{-1}z_0$. If $|z-z_1|<\sqrt n$, then we have that
$|\delta z-z_0|<\sqrt n\delta$ and
\begin{equation}\label{e3.4}
\frac1{R_l(z)}=\frac{\delta^{-n}}{R_l(\delta z)}=\delta^{-n}\sum_{k=1}^\infty
a_ke^{i\sigma_k \cdot (\delta z)}.
\end{equation}

For any cube $Q:= Q(x_0, r)\subset\mathbb R^n$,
 let $y_0:= x_0-2rz_1$ and $Q':= Q(y_0, r)$.
Then we obtain that for all $x\in Q$ and $y\in Q'$,
\begin{equation}\label{e3.5}
\left|\frac{x-y}{2r}-z_1\right|\le \frac{|x-x_0|}{2r}+\frac{|y-y_0|}{2r}<\sqrt n.
\end{equation}
 From this,
\eqref{e3.4}, \eqref{e3.5}, the H\"older inequality and \eqref{e:bdd commt even ext}, we then deduce that
\begin{eqnarray*}
&&\dint_Q\left|b_{+,e}(x)-\langle b_{+,e}\rangle_{Q'}\right|\,dx\\
&&\quad=\dint_{\mathbb R^n}\left[b_{+,e}(x)-\langle b_{+,e}\rangle_{Q'}\right]sgn\left(b-\langle b_{+,e}\rangle_{Q'}\right)\unit_Q(x)\,dx\\
&&\quad=\frac1{|Q'|}\dint_{\mathbb R^n}\dint_{\mathbb R^n}[b_{+,e}(x)-b_{+,e}(y)]sgn(b_{+,e}(x)-\langle b_{+,e}\rangle_{Q'})
\unit_Q(x)\unit_{Q'}(y)\\
&&\quad\quad\times
\frac{(2r)^nR_l(x-y)}{R_l(\frac{x-y}{2r})}\,dy\,dx\\
&&\quad\ls\dint_{\mathbb R^n}\dint_{\mathbb R^n}[b_{+,e}(x)-b_{+,e}(y)]sgn(b_{+,e}(x)-\langle b_{+,e}\rangle_{Q'})
\unit_Q(x)\unit_{Q'}(y)\\
&&\quad\quad\times R_l(x, y)
\sum_{k=1}^\infty a_ke^{i\frac{\delta\sigma_k}{2r}\cdot  (x-y)}\,dy\,dx\\
&&\quad\ls \sum_{k=1}^\infty|a_k|\dint_{\mathbb R^n}\left|\left[b, R_l\right]\left(\unit_{Q'}
e^{-i\frac{\delta\sigma_k}{2r} }\right)(x)\right|\unit_Q(x)\,dx\\
&&\quad\ls \sum_{k=1}^\infty|a_k|\left\|\left[b, R_l\right]\left(\unit_{Q'}
e^{-i\frac{\delta\sigma_k}{2r} }\right)\right\|_{L^{p}_{\lambda_{+,e}}(\mathbb R^n)}
\left[\lambda_{+,e}^{1-p'}(Q)\right]^{1/{p'}}\\
&&\quad\ls C_l\left[\mu_{+,e}(Q)\right]^{1/{p}}\left[\lambda_{+,e}^{1-p'}(Q)\right]^{1/{p'}},
\end{eqnarray*}
which together with \eqref{Bloom weight2} implies that
\begin{equation}\label{e3.7}
\dint_Q|b_{+,e}(x)-\langle b_{+,e}\rangle_Q|\,dx\ls C_l\nu_{+, e}(Q).
\end{equation}
This shows $b_{+,e}\in {\rm BMO}_{\nu_{+, e}}(\R^n)$.

Symmetrically we obtain that
$b_{-,e}$  is in ${\rm BMO}_{\nu_{-,e}}(\R^n)$  with $\|b_{-,e}\|_{{\rm BMO}_{\nu_{-,e}}(\R^n)}\ls C_l$.  Combining these two facts and Proposition \ref{p-coinc Dz-BMO-2}, we obtain that
$b$  is in ${\rm BMO}_{\Delta_N, \nu}(\R^n)$ with
$$\|b\|_{{\rm BMO}_{\Delta_N,\nu}(\mathbb{R}^n)}\ls C_l,
$$
i.e., the claim \eqref{lower} holds.
\end{proof}

\begin{remark}\label{remark Th1.1}
Since $\Delta_{N_+}$ is the part of the $\Delta_N$ on the upper half space $\R^n_+$,
based on the proof of the upper bound and lower bound above for Theorem \ref{t:upp-low Neum Lap wh-sp commu}, we can obtain the proof of Theorem \ref{t:upp-low Neum Lap ha-sp commu} just by tracking the
estimates for the even extension of the positive parts of the functions and the weights in the proof of Theorem \ref{t:upp-low Neum Lap wh-sp commu} above, i.e., tracking the process involving those $b_{+,e}$, $f_{+,e}$,
$\mu_{+,e}$, $\lambda_{+,e}$.
\end{remark}


\section{Dirichlet Laplacian and proof of Theorem \ref{t:Dirich commuta counterex}}
\label{s:Dirichlet}

By $\Delta_{D_+}$ we denote the Dirichlet Laplacian
on $\mathbb R^n_+$. The Dirichlet Laplacian is a positive definite self-adjoint
operator. By the spectral theorem one can define the semigroup generated
by this operator $\{\exp (-t\Delta_{D_+}): t \ge 0\}$.
By $p_{t,\Delta_{D_+}}(x, y)$ we denote the heat kernels corresponding
to the semigroup generated by $\Delta_{D_+}$.
From the reflection method (see \cite{S}*{(9), page 59 in Section 3.1}), we get
\begin{eqnarray*}
p_{t,\Delta_{D_+}}(x,y)  = \frac{1}{(4\pi t)^{\frac{n}{2}}}e^{-\frac{|x'-y'|^2}{4t}}\left( e^{-\frac{|x_n-y_n|^2}{4t}}-e^{-\frac{|x_n+y_n|^2}{4t}} \right), \ \ x,y\in \mathbb{R}^n_+.
\end{eqnarray*}

Denote by $R_{D,j}(x,y)$ the kernel of the $j$-th Riesz transform $\frac{\partial}{\partial x_j} \Delta_{D_+}^{-\frac{1}{2}}$ of $\Delta_{D_+}$ associated with the Dirichlet Laplacian. Then analogous to
Proposition \ref{p:RieszKernel}, we have
the following conclusions whose proofs are similar and omitted.

\begin{prop}\label{p:RieszKernel-Dirich}
 Then for $1\leq j\leq n-1$ and for $x,y\in\mathbb{R}^n_+$ we have:
$$
R_{D,j}(x,y)= - C_n \bigg( {x_j-y_j\over |x-y|^{n+1}} - \frac{x_j-y_j}{(|x'-y'|^2+|x_n+y_n|^2)^{\frac{n+1}{2}}}\bigg);
$$
and for $j=n$ we have:
$$
R_{D,n}(x,y)= - C_n \bigg( {x_j-y_j\over |x-y|^{n+1}} -  \frac{x_n+y_n}{(|x'-y'|^2+|x_n+y_n|^2)^{\frac{n+1}{2}}}\bigg),
$$
where $C_n=\frac{\Gamma\big(\frac{n+1}{2}\big)}{(\pi )^{\frac{n+1}{2}}}  $.
\end{prop}

From Proposition \ref{p:RieszKernel-Dirich}, we deduce that for each $j\in\{1, 2,\,\cdots, n\}$,
$R_{D, j}$ is a Calder\'on--Zygmund kernel which satisfies the following conditions:
for any $x, y\in \mathbb  R^n_+$ with $x\not=y$,
\begin{align*}
|R_{D, j}(x,y)|\leq C  \frac{1}{|x-y|^{n}},
\end{align*}
and for $x,x_0,y\in\mathbb{R}^n_+$   with $|x-x_0|  \leq \frac{1}{2} |x-y|$,
\begin{align*}
|R_{D, j}(x,y)-R_{D, j}(x_0,y)|+|R_{D, j}(y,x)-R_{D, j}(y,x_0)|\leq C\frac{|x-x_0|}{|x-y|^{n+1}}.
\end{align*}
Moreover, by the fact that for any $x\in\mathbb R^n_+$ and $t>0$,
$\nabla \Delta_{D_+}^{-\frac12}f(x)=\nabla \Delta^{-\frac12}f_o(x)$ (see \cite{DDSY}*{(2.6)}), we have that
for any $f\in L^2(\mathbb R_+^n)$,
$$\|\nabla \Delta_{D_+}^{-\frac12}f\|_{L^2(\mathbb R_+^n)}\le \|\nabla \Delta^{-\frac12}f_o\|_{L^2(\mathbb R^n)}
\ls\|f_o\|_{L^2(\mathbb R^n)}\sim \|f\|_{L^2(\mathbb R_+^n)},$$
where $f_o$ is the odd extension of $f$ to $\mathbb R^n$. This implies that $\nabla \Delta_{D_+}^{-\frac12}$
is bounded on $L^2(\mathbb R_+^n)$.

Now let ${\rm BMO}(\mathbb R^n_+)$ be the classical BMO space on $\mathbb R^n_+$, that is,
$${\rm BMO}(\mathbb R^n_+):=\{f\in L^1_{loc}(\mathbb R^n_+):\,\, \|f\|_{{\rm BMO}(\mathbb R^n_+)}<\infty\},$$
where
$$\|f\|_{{\rm BMO}(\mathbb{R}^n_+)}:=\sup_{Q\subset \mathbb{R}^n_+}\frac1{|Q|}\int_Q\left|f(x)-\langle f\rangle_Q\right|dx;$$
see \cite{CW}.  Then we have that for any $j\in\{1, 2, \cdots, n\}$ and $b\in {\rm BMO}(\mathbb R^n_+)$,
the commutator $[b, R_{D, j}]$ is bounded on $L^2(\mathbb R^n_+)$, see for example the upper bound of the commutators showed  in \cite{CRW}, i.e., for every $b\in  {\rm BMO}(\mathbb R^n_+)$,
\begin{align}\label{section8 upper bound}
\| [b,\nabla \Delta_{D_+}^{-{1\over2}}] : L^{p}(\mathbb{R}^n_+)\to  L^{p}(\mathbb{R}^n_+)\| \leq C\|b\|_{{\rm BMO}(\mathbb R^n_+)}.
\end{align}

Now let ${\rm BMO}_{\Delta_{D_+}}(\mathbb R^n_+)$ be the BMO space associated with $\Delta_{D_+}$
on $\mathbb R^n_+$, which was introduced in \cite{DDSY}. Recall that
${\rm BMO}_{\Delta_{D_+}}(\mathbb R^n_+)$ coincides with ${\rm BMO}_{o}(\mathbb R^n_+)$, where
${\rm BMO}_{o}(\mathbb R^n_+)$ is the set of functions on $\mathbb R_+^n$ whose odd extension
belong to ${\rm BMO}(\mathbb R^n)$, and is a proper subspace of ${\rm BMO}(\mathbb R^n_+)$.
Then we know that
$${\rm BMO}_{\Delta_{D_+}}(\mathbb R^n_+) \subsetneq {\rm BMO}(\mathbb R^n_+),$$
(see for example  \cite{DDSY} or
\cite{DLWY}).

This strict inclusion, together with \eqref{section8 upper bound}, shows that there exists a function $b_0\in {\rm BMO}(\mathbb R^n_+)\setminus {\rm BMO}_{\Delta_{D_+}}(\mathbb R^n_+)$
such that     
$$\| [b_0,\nabla \Delta_{D_+}^{-{1\over2}}] : L^{p}(\mathbb{R}^n_+)\to  L^{p}(\mathbb{R}^n_+)\| \leq C_{b_0}<\infty$$
with $C_{b_0}:=\|b_0\|_{{\rm BMO}(\mathbb R^n_+)}$.
Thus,
Theorem \ref{t:Dirich commuta counterex} holds.

\bigskip
{\bf Acknowledgement:} X.T. Duong  and J. Li are supported by ARC DP 160100153. I. Holmes is supported by the National Science Foundation under Award \# 1606270.
B.D. Wick is supported in part by National Science Foundation grant DMS \#1560955. D. Yang  is supported by the NNSF of China (Grant No. 11571289).

\end{document}